\documentclass[12pt]{amsart}

\usepackage{graphicx}
\usepackage{multicol,multirow}
\usepackage{amsmath,amssymb,amsfonts,amsthm}
\usepackage{mathrsfs}
\usepackage{rotating}
\usepackage{appendix}
\usepackage{ifpdf}
\usepackage[T1]{fontenc}
\usepackage{newtxtext}
\usepackage{newtxmath}
\usepackage{textcomp}
\usepackage[colorlinks,citecolor=blue]{hyperref}
\usepackage{color, tikz}
\usepackage{tabularx}
\usepackage{cancel}
\usepackage{lscape}
\usepackage{caption}
\usepackage{rotating}
\usepackage{enumitem}
\usepackage[normalem]{ulem}
\usepackage{enumitem}
\usetikzlibrary{arrows.meta}
\usetikzlibrary{calc}
\usetikzlibrary{patterns}
\usepackage[margin=2.5cm]{geometry}

\usepackage{algorithm}
\usepackage{algpseudocode}

\newtheorem{theorem}{Theorem}[section]
\newtheorem{prop}[theorem]{Proposition}

\newtheorem{conj}[theorem]{Conjecture}

\newtheorem{lemma}[theorem]{Lemma}
\newtheorem{corollary}[theorem]{Corollary}
\newtheorem{remark}[theorem]{Remark}
\newtheorem{definition}[theorem]{Definition}

\newtheorem{example}[theorem]{Example}
\newtheorem{question}[theorem]{Question}

\newcommand{\Z}{{\mathbb Z}}

\newcommand{\Bcal}{{\mathcal B}}

\newcommand{\Gcal}{{\mathcal G}}

\newcommand{\Jcal}{{\mathcal J}}

\newcommand{\Scal}{{\mathcal S}}
\newcommand{\T}{{\mathcal T}}

\newcommand{\Div}{\operatorname{Div}}

\newcommand{\Pic}{\operatorname{Pic}}

\newcommand{\im}{\operatorname{im}}
\newcommand{\onechip}{[c-s]}

\newcommand{\what}{\widehat}

\DeclareMathOperator{\BBY}{BBY}
\newcommand{\wtilde}{\widetilde}

\newcommand{\la}{{\langle}}
\newcommand{\ra}{{\rangle}}

\newcommand{\rotdir}[1]{\langle #1\rangle}

\newcommand{\ic}{n}
\newcommand{\jest}[3]{\Jcal_{#1,#2,#3}}
\newcommand{\tele}[2]{\mathcal T^{#1}_{(#2)}}

\newcommand{\rt}{\mathfrak{r}}

\usepackage{hyperref}
\renewcommand{\i}{{\mathfrak i}}

\author[Ankan Ganguly]{Ankan Ganguly\textsuperscript{*}}
\address{\textsuperscript{*}Wharton Department of Statistics and Data Science, University of Pennsylvania: ankan@wharton.upenn.edu}
\author[Alex McDonough]{ Alex McDonough\textsuperscript{\dag}}
\address{\textsuperscript{\dag}Department of Mathematics, University of California, Davis: amcd@ucdavis.edu}

\date{January 31}

\begin{document}
\title[There is a Unique Consistent Sandpile Torsor Structure]{Rotor-Routing Induces the Only Consistent Sandpile Torsor Structure on Plane Graphs}

\date{April 4, 2022 (revised January 31, 2023 and June 16, 2023)}
\maketitle
\begin{abstract} We make precise and prove a conjecture of Klivans about actions of the sandpile group on spanning trees. More specifically, the conjecture states that there exists a unique ``suitably nice'' sandpile torsor structure on plane graphs which is induced by rotor-routing.

First, we rigorously define a sandpile torsor algorithm (on plane graphs) to be a map which associates each plane graph (i.e., planar graph with an appropriate ribbon structure) with a free transitive action of its sandpile group on its spanning trees. Then, we define a notion of consistency, which requires a torsor algorithm to be preserved with respect to a certain class of contractions and deletions. Using these definitions, we show that the rotor-routing sandpile torsor algorithm is consistent. Furthermore, we demonstrate that there are only three other consistent algorithms on plane graphs, which all have the same structure as rotor-routing. 

We also define sandpile torsor algorithms on regular matroids and suggest a notion of consistency in this context. We conjecture that the Backman-Baker-Yuen algorithm is consistent, and that there are only three other consistent sandpile torsor algorithms on regular matroids, all with the same structure.
\end{abstract}

%\blfootnote{The first author was supported by the United States Army Research Office under grant number W911NF2010133.}

\section{Introduction}\label{sec:intro} 

Let $G$ be a finite connected multigraph. The \emph{sandpile group} of $G$, which we denote $\Pic^0(G)$, is a finite abelian group given by the cokernal of the \emph{graph Laplacian matrix} over the integers. A remarkable fact, which follows from Kirchhoff's matrix-tree theorem, is that the size of $\Pic^0(G)$ is equal to the number of \emph{spanning trees} of $G$ (see \cite[Theorem 7.3]{Biggs}). The traditional proof of this result is non-bijective, and for many graphs, no  automorphism invariant bijections exist unless $G$ is given additional structure (see \cite[Theorem 8.1]{Wagner}). This problem is often resolved by working with \emph{ribbon graphs}. 

Let $\chi$ specify a cyclic ordering on the edges adjacent to each vertex of $G$. The pair $(G,\chi)$ is called a \emph{ribbon graph}. Holroyd et al. demonstrated that the \emph{rotor-routing algorithm} induces a \emph{free transitive group action} of $\Pic^0(G)$ on $\T(G)$, which depends on the ribbon structure $\chi$ and a choice of \emph{sink} vertex~\cite{Holroyd}. In particular, this makes $\T(G)$ a $\Pic^0(G)$-\emph{torsor}, so we will call such an action a \emph{sink-parameterized sandpile torsor action} on $(G,\chi)$.

%Because $\Pic^0(G)$ has a group structure, one natural approach to associate $\Pic^0(G)$ with $\T(G)$ is by constructing a \emph{free transitive group action} of $\Pic^0(G)$ on $\T(G)$. In particular, this makes $\T(G)$ a $\Pic^0(G)$-\emph{torsor}, so we will call such an action a \emph{sandpile torsor action}. This approach was taken by Holroyd et al. who demonstrated that the \emph{rotor-routing algorithm} induces a sandpile torsor action on any ribbon graph~\cite{Holroyd}. However, this action 

%Because $\Pic^0(G)$ has a group structure, in addition to bijections between $\Pic^0(G)$ and $\T(G)$, we can also define \emph{free transitive group actions} of $\Pic^0(G)$ on $\T(G)$. Suppose that $(G,\chi)$ is a \emph{ribbon graph}, where $\chi$ specifies a cyclic ordering on the edges adjacent to each vertex. Furthermore, suppose that we distinguish a vertex which we call the \emph{sink}. Then, there exist natural free transitive actions of $\Pic^0(G)$ on $\T(G)$. In particular, the \emph{rotor-routing algorithm} and the \emph{Bernardi algorithm} assign free transitive actions to any ribbon graph and sink. 

On a MathOverflow post, Ellenberg asked if there exist certain classes of ribbon graphs for which natural sandpile torsor actions can be defined without requiring a sink vertex~\cite{JSE}. Chan, Church, and Grochow answered this question by showing that the rotor-routing sandpile torsor action on $(G,\chi)$ is independent of the sink vertex if and only if $(G,\chi)$ is a \emph{plane graph}~\cite[Theorem 2]{CCG}.\footnote{Following the language of~\cite{BW}, we distinguish between a \emph{plane graph}, which is a particular planar embedding, and a \emph{planar graph} which is a graph where such an embedding exists.} With this result in mind, we define a \emph{sandpile torsor algorithm} (on plane graphs) to be a function which assigns a sandpile torsor action to every plane graph $(G,\chi)$ independent of the choice of sink vertex. 

Another sink-parameterized sandpile torsor action is given by the \emph{Bernardi algorithm}~\cite{Bernardi}, which is distinct from rotor-routing on non-planar ribbon graphs. Baker and Wang showed that the restriction of the Bernardi algorithm to plane graphs is also a sandpile torsor algorithm~\cite[Theorem 5.1]{BW}. Furthermore, they showed that this sandpile torsor algorithm is equivalent to the rotor-routing algorithm (when restricted to plane graphs)~\cite[Theorem 7.1]{BW}. This surprising equivalence motivated the following conjecture by Klivans, which we prove in this paper. 

\begin{conj}{\cite[Conjecture 4.7.17]{Klivans}}\label{conj:klivans}
For plane graphs, there is only one sandpile torsor structure. 
\end{conj}

One of the challenges in resolving this conjecture is that one must introduce a reasonable definition of \emph{sandpile torsor structure} on plane graphs. The rotor-routing sandpile torsor algorithm follows a simple set of rules that can be applied to any plane graph. With this in mind, we posit that any definition of \emph{sandpile torsor structure} should include some notion of consistency of induced sandpile torsor actions between different plane graphs. However, it is difficult to find a suitable consistency condition because there is no general map between the sandpile groups of different graphs which preserves their structure. The key insight that motivated the writing of this paper is the formulation of a contraction-deletion based definition of consistency (see Definition~\ref{def:consistent} and Figure~\ref{fig:consistent}).\\

\textbf{There are two main results in this paper}. In Theorem~\ref{thm:rrconsistent}, we show that rotor-routing is \emph{consistent}. In Theorem~\ref{thm:onestructure} we show that there are only 3 other consistent sandpile torsor algorithms on plane graphs, which can be obtained from rotor-routing by reversing the cyclic order on the vertices, taking the inverse action, or both. We say that these algorithms all have the same \emph{structure} as rotor-routing. One powerful tool that we introduce is Theorem~\ref{thm:sourceturn}, which implies that on any 2-connected graph, it is possible to transform one spanning tree to any other by repeatedly swapping \emph{leaf edges} with edges not on the tree (Corollary~\ref{cor:leafswap}).

In Section~\ref{sec:background}, we define relevant terms and introduce the rotor-routing algorithm. In Section~\ref{sec:funfacts}, we give several properties of rotor-routing that are applied in later sections. In Section~\ref{sec:consistent}, we introduce \emph{consistency} and prove that the rotor-routing algorithm is consistent. In Section~\ref{sec:unique}, we show that every other consistent sandpile torsor algorithm on plane graphs has the same structure as rotor-routing. In Section~\ref{sec:matroids}, we define sandpile torsor algorithms on \emph{regular matroids} as well as a notion of \emph{consistency} in this context. We conjecture that the Backman-Baker-Yuen matroidal sandpile torsor algorithm (constructed in \cite{BBY}) is consistent (Conjecture~\ref{conj:BBYconsistent}) and that all other consistent matroidal sandpile torsor algorithms have the same structure (Conjecture~\ref{conj:allconsistent}). In Appendix~\ref{app:unicycles}, we introduce terminology used in~\cite{Holroyd} and~\cite{CCG} and provide proofs of the properties presented in Section~\ref{sec:funfacts}. Finally, Appendix~\ref{app:case4} is devoted to building the framework to prove the final case of Theorem~\ref{thm:onestructure}.

\section{Background and Definitions}\label{sec:background}
\subsection{Divisors, Ribbon Graphs, and the Sandpile Group}

In this section, we use much of the notation from~\cite{CCG}. See ~\cite{divisorsbook} and \cite{Klivans} for more information on divisors, sandpile groups, and chip-firing.

Let $G$ be a finite, connected, undirected graph with vertices $V(G)$ and edges $E(G)$. We allow parallel edges but not loop edges. For an edge $e \in E(G)$, we write $\i(e)$ for the unordered pair of vertices incident to $e$ (in set notation). We say that edges $e$ and $f$ are \emph{parallel} if $\i(e) = \i(f)$. For $v \in V(G)$, we write $\deg(v)$ for the number of edges incident to $v$. For $x,y \in V(G)$, a \emph{path} $P$ from $x$ to $y$ is a collection of distinct $(e_1,\dots,e_k) \in E(G)^k$ for some $k$ such that $x \in \i(e_1)$, $y \in \i(e_k)$, and 
\[|\i(e_i) \cap \i(e_j)| = \begin{cases}2 &\text{if $i=j$,}\\1 &\text{if $|i-j|=1$,}\\0 &\text{otherwise.}\\\end{cases}\]

\begin{definition}
The group $\Div(G)$ of \emph{divisors} of $G$ is given by

\[\Div(G) := \left\{\sum_{v \in V(G)}n_vv\mid n_v\in\Z\right \}.\]

The subgroup $\Div^0(G)$ of \emph{degree-0 divisors}\footnote{In general, the \emph{degree} of a divisor is the sum of the coefficients on the vertices, and is not related to the degree of the vertices of $G$. For this paper, all of our divisors will have degree $0$, and any future instances of the term \emph{degree} always refer to the degree of a vertex.} of $G$ is given by

\[\Div^0(G) := \left\{\sum_{v \in V(G)}n_vv\mid n_v\in\Z, \sum_{v \in V(G)}n_v = 0\right\}.\]

For any $s \in V(G)$, let

\[\Div^0_s(G) := \left\{\sum_{v \in V(G)}n_vv\mid n_v\in\Z_{\ge 0} \text{ for $v \not= s$}, \sum_{v \in V(G)}n_v = 0\right\}.\]
\end{definition}

Note that while $\Div(G)$ and $\Div^0(G)$ are groups, $\Div_s^0(G)$ is a \emph{monoid} because its nonzero elements lack inverses. The vertex $s$ is called the \emph{sink vertex}. 

Let $A$ be the adjacency matrix of $G$ (i.e. the $|V(G)|\times |V(G)|$ symmetric matrix with $A_{vw}$ equal to the number of edges between $v$ and $w$). Let $\text{deg}(G)$ be the $|V(G)|\times |V(G)|$ diagonal matrix such that $\text{deg}(G)_{vv} = \text{deg}(v)$ for every $v \in V(G)$. The \emph{Laplacian matrix} $\Delta$ of $G$ is the symmetric matrix defined by $\Delta := \text{deg}(G) - A$.

\begin{definition}
The \emph{sandpile group} of $G$, denoted $\Pic^0(G)$, is the quotient
\[\Pic^0(G) := \Div^0(G)/\im_{\Z}(\Delta).\]
\end{definition}

The sandpile group has been rediscovered in many contexts and is also called the \emph{graph Picard group}, the \emph{graph Jacobian}, the \emph{critical group}, the \emph{group of components}, or the \emph{group of bicycles}. One way to explore properties of $\Pic^0(G)$ is in terms of \emph{chip-firing}, also called \emph{the dollar game}, which we describe below. 

Each element $D = \sum_{v\in V(G)} n_vv\in \Div^0(G)$ may be viewed as a configuration of ``chips'' placed on the vertices of $G$, also allowing negative ``debt'' chips (where the total number of chips matches the total amount of debt). Note that elements of $\Div^0_s(G)$ have all of their debt on the sink vertex. In this context, divisors are also called \emph{chip configurations}. Given a chip configuration $D$, we may ``fire'' a vertex $v \in V(G)$ by removing $\deg(v)$ chips from $v$ and placing one chip on each of the neighbors of $v$. That is, $v$ ``gifts'' a chip to each of its neighbors. For our purposes, we allow a vertex to fire regardless of the number of chips, even if this puts the vertex in debt. 

Each row of the Laplacian matrix corresponds to a vertex firing. Thus, the image of the Laplacian over $\Z$ describes an equivalence relation on chip configurations, where $D$ and $D'$ are \emph{firing equivalent} if $D'$ can be obtained from $D$ by a sequence of vertex firings. This means that two chip configurations are firing equivalent precisely when they are both representatives of the same equivalence class of $\Pic^0(G)$. 

Throughout this paper, we will put brackets around a divisor to indicate the corresponding equivalence class of $\Pic^0(G)$. As discussed in the previous paragraph, for $D,D' \in \Div^0(G)$, we have $[D] = [D']$, or equivalently $[D-D'] = [0]$, if and only if $D$ and $D'$ are firing equivalent. We consider $D,D' \in \Div^0_s(G)$ to be firing equivalent if they are firing equivalent as elements of $\Div^0(G)$.

We will also use brackets to indicate sets of integers, where $[a,b] := \{z \in \Z \mid a \le z \le b\}$. For clarity, we will write $[1,n]$ in place of the more common $[n]$ so that this is not mistaken for an equivalence class.

\begin{lemma}\label{lem:anysink}
For any $s \in V(G)$ and $D \in \Div^0(G)$, there exists some $D' \in \Div^0_s(G)$ such that $[D] = [D']$. 
\end{lemma}
\begin{proof}
Let $\delta \in \Div^0_s(G)$ assign $\deg(v)$ chips to each $v \in V(G) \setminus s$ and $-\sum_{v\not=s} \deg(v)$ chips to $s$. Let $\delta^\circ$ be the divisor we obtain from $\delta$ by repeatedly firing any vertices $v$ that have more than $\deg(v)$ many chips until no such vertices exist.\footnote{The notation for $\delta$ and $\delta^\circ$, as well as the general method of this proof, come from~\cite{Holroyd} (where $\delta^\circ$ is called the \emph{stabilization} of $\delta$). The divisor $\delta^\circ$ is well-defined by~\cite[Lemma~2.2]{Holroyd} (see also~\cite{Dhar90}).}

Because $\delta$ and $\delta^\circ$ are firing equivalent, $[\delta - \delta^\circ] = 0$. Furthermore, by construction, $\delta-\delta^\circ$ has a positive number of chips at every $v \not= s$. Let $-m$ be the minimal number of chips $D$ assigns to any vertex. Then, $D + m(\delta - \delta^\circ) \in \Div^0_s(G)$ and $[D + m(\delta - \delta^\circ)] = [D]$. 
\end{proof}

A \emph{spanning tree} of $G$ is a maximal subset of $E(G)$ that contains no cycles. Let $\mathcal T(G)$ be the set of spanning trees of $G$. The following is a version of Kirchhoff's matrix-tree theorem. 

\begin{theorem}[sandpile matrix-tree theorem for graphs~{\cite[Theorem 7.3]{Biggs}}]\label{thm:mtt}
\[|\Pic^0(G)| = |\mathcal T(G)|\] 
\end{theorem}

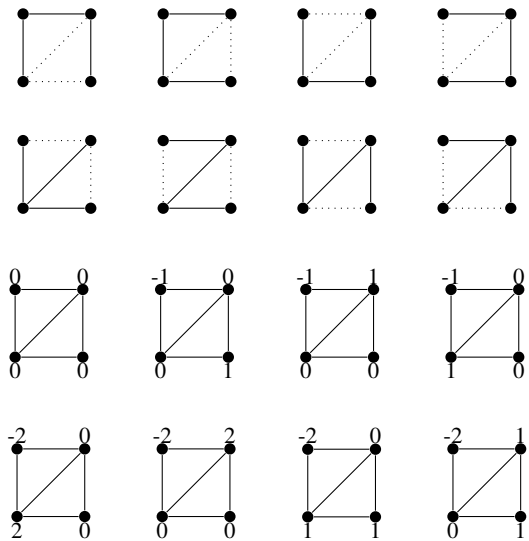
\begin{figure}
\begin{center}
\begin{tikzpicture}[scale= 0.3]
    
    \tikzstyle{every node} = [circle,fill,inner sep=1pt,minimum size = 1.5mm]
    \node(a) at (0,0) {};
    \node(b) at (0,3){};
    \node(c) at (3,3) {};
    \node(d) at (3,0){};

    \tikzstyle{every node} = [draw = none,fill = none,scale = .6]
    \draw (a) -- (b);
    \draw (b) -- (c);
    \draw (c) -- (d);
    \draw[dotted](a) -- (c);
    \draw[dotted] (a) -- (d);

\end{tikzpicture}
\hspace{.6 cm}
\begin{tikzpicture}[scale= 0.3]
    
    \tikzstyle{every node} = [circle,fill,inner sep=1pt,minimum size = 1.5mm]
    \node(a) at (0,0) {};
    \node(b) at (0,3){};
    \node(c) at (3,3) {};
    \node(d) at (3,0){};

    \tikzstyle{every node} = [draw = none,fill = none,scale = .6]
    \draw (a) -- (b);
    \draw (b) -- (c);
    \draw [dotted](c) -- (d);
    \draw [dotted](a) -- (c);
    \draw (a) -- (d);

\end{tikzpicture}
\hspace{.6 cm}
\begin{tikzpicture}[scale= 0.3]
    
    \tikzstyle{every node} = [circle,fill,inner sep=1pt,minimum size = 1.5mm]
    \node(a) at (0,0) {};
    \node(b) at (0,3){};
    \node(c) at (3,3) {};
    \node(d) at (3,0){};

    \tikzstyle{every node} = [draw = none,fill = none,scale = .6]
    \draw (a) -- (b);
    \draw [dotted](b) -- (c);
    \draw (c) -- (d);
    \draw [dotted](a) -- (c);
    \draw (a) -- (d);

\end{tikzpicture}
\hspace{0.6 cm}
\begin{tikzpicture}[scale= 0.3]
    
    \tikzstyle{every node} = [circle,fill,inner sep=1pt,minimum size = 1.5mm]
    \node(a) at (0,0) {};
    \node(b) at (0,3){};
    \node(c) at (3,3) {};
    \node(d) at (3,0){};

    \tikzstyle{every node} = [draw = none,fill = none,scale = .6]
    \draw [dotted](a) -- (b);
    \draw (b) -- (c);
    \draw (c) -- (d);
    \draw [dotted](a) -- (c);
    \draw (a) -- (d);

\end{tikzpicture}\\
\vspace{0.6 cm}
\begin{tikzpicture}[scale= 0.3]
    
    \tikzstyle{every node} = [circle,fill,inner sep=1pt,minimum size = 1.5mm]
    \node(a) at (0,0) {};
    \node(b) at (0,3){};
    \node(c) at (3,3) {};
    \node(d) at (3,0){};

    \tikzstyle{every node} = [draw = none,fill = none,scale = .6]
    \draw (a) -- (b);
    \draw [dotted](b) -- (c);
    \draw [dotted](c) -- (d);
    \draw (a) -- (c);
    \draw (a) -- (d);

\end{tikzpicture}
\hspace{0.6 cm}
\begin{tikzpicture}[scale= 0.3]
    
    \tikzstyle{every node} = [circle,fill,inner sep=1pt,minimum size = 1.5mm]
    \node(a) at (0,0) {};
    \node(b) at (0,3){};
    \node(c) at (3,3) {};
    \node(d) at (3,0){};

    \tikzstyle{every node} = [draw = none,fill = none,scale = .6]
    \draw [dotted](a) -- (b);
    \draw (b) -- (c);
    \draw [dotted](c) -- (d);
    \draw (a) -- (c);
    \draw (a) -- (d);

\end{tikzpicture}
\hspace{0.6 cm}
\begin{tikzpicture}[scale= 0.3]
    
    \tikzstyle{every node} = [circle,fill,inner sep=1pt,minimum size = 1.5mm]
    \node(a) at (0,0) {};
    \node(b) at (0,3){};
    \node(c) at (3,3) {};
    \node(d) at (3,0){};

    \tikzstyle{every node} = [draw = none,fill = none,scale = .6]
    \draw (a) -- (b);
    \draw [dotted](b) -- (c);
    \draw (c) -- (d);
    \draw (a) -- (c);
    \draw [dotted](a) -- (d);

\end{tikzpicture}
\hspace{0.6 cm}
\begin{tikzpicture}[scale= 0.3]
    
    \tikzstyle{every node} = [circle,fill,inner sep=1pt,minimum size = 1.5mm]
    \node(a) at (0,0) {};
    \node(b) at (0,3){};
    \node(c) at (3,3) {};
    \node(d) at (3,0){};

    \tikzstyle{every node} = [draw = none,fill = none,scale = .6]
    \draw [dotted](a) -- (b);
    \draw (b) -- (c);
    \draw (c) -- (d);
    \draw (a) -- (c);
    \draw [dotted](a) -- (d);

\end{tikzpicture}\\
\vspace{0.6 cm}
\begin{tikzpicture}[scale= 0.3]
    
    \tikzstyle{every node} = [circle,fill,inner sep=1pt,minimum size = 1.5mm]
    \node(a) at (0,0) {};
    \node(b) at (0,3){};
    \node(c) at (3,3) {};
    \node(d) at (3,0){};

    \tikzstyle{every node} = [draw = none,fill = none,scale = .6]
    \draw (a) -- (b);
    \draw (b) -- (c);
    \draw (c) -- (d);
    \draw (a) -- (c);
    \draw (a) -- (d);
    
    \node (o) at (0,3.6){{\large 0}};
    \node (o) at (0,-.6){{\large 0}};
    \node (o) at (3,3.6){{\large 0}};
    \node (o) at (3,-.6){{\large 0}};

\end{tikzpicture}
\hspace{.45 cm}
\begin{tikzpicture}[scale= 0.3]
    
    \tikzstyle{every node} = [circle,fill,inner sep=1pt,minimum size = 1.5mm]
    \node(a) at (0,0) {};
    \node(b) at (0,3){};
    \node(c) at (3,3) {};
    \node(d) at (3,0){};

    \tikzstyle{every node} = [draw = none,fill = none,scale = .6]
    \draw (a) -- (b);
    \draw (b) -- (c);
    \draw (c) -- (d);
    \draw (a) -- (c);
    \draw (a) -- (d);
    
    \node (o) at (0,3.6){{\large -1}};
    \node (o) at (0,-.6){{\large 0}};
    \node (o) at (3,3.6){{\large 0}};
    \node (o) at (3,-.6){{\large 1}};

\end{tikzpicture}
\hspace{.45 cm}
\begin{tikzpicture}[scale= 0.3]
    
    \tikzstyle{every node} = [circle,fill,inner sep=1pt,minimum size = 1.5mm]
    \node(a) at (0,0) {};
    \node(b) at (0,3){};
    \node(c) at (3,3) {};
    \node(d) at (3,0){};

    \tikzstyle{every node} = [draw = none,fill = none,scale = .6]
    \draw (a) -- (b);
    \draw (b) -- (c);
    \draw (c) -- (d);
    \draw (a) -- (c);
    \draw (a) -- (d);
    
    \node (o) at (0,3.6){{\large -1}};
    \node (o) at (0,-.6){{\large 0}};
    \node (o) at (3,3.6){{\large 1}};
    \node (o) at (3,-.6){{\large 0}};

\end{tikzpicture}
\hspace{0.45 cm}
\begin{tikzpicture}[scale= 0.3]
    
    \tikzstyle{every node} = [circle,fill,inner sep=1pt,minimum size = 1.5mm]
    \node(a) at (0,0) {};
    \node(b) at (0,3){};
    \node(c) at (3,3) {};
    \node(d) at (3,0){};

    \tikzstyle{every node} = [draw = none,fill = none,scale = .6]
    \draw (a) -- (b);
    \draw (b) -- (c);
    \draw (c) -- (d);
    \draw (a) -- (c);
    \draw (a) -- (d);
    
    \node (o) at (0,3.6){{\large -1}};
    \node (o) at (0,-.6){{\large 1}};
    \node (o) at (3,3.6){{\large 0}};
    \node (o) at (3,-.6){{\large 0}};

\end{tikzpicture}\\
\vspace{0.45 cm}
\begin{tikzpicture}[scale= 0.3]
    
    \tikzstyle{every node} = [circle,fill,inner sep=1pt,minimum size = 1.5mm]
    \node(a) at (0,0) {};
    \node(b) at (0,3){};
    \node(c) at (3,3) {};
    \node(d) at (3,0){};

    \tikzstyle{every node} = [draw = none,fill = none,scale = .6]
    \draw (a) -- (b);
    \draw (b) -- (c);
    \draw (c) -- (d);
    \draw (a) -- (c);
    \draw (a) -- (d);
    
    \node (o) at (0,3.6){{\large -2}};
    \node (o) at (0,-.6){{\large 2}};
    \node (o) at (3,3.6){{\large 0}};
    \node (o) at (3,-.6){{\large 0}};

\end{tikzpicture}
\hspace{0.45 cm}
\begin{tikzpicture}[scale= 0.3]
    
    \tikzstyle{every node} = [circle,fill,inner sep=1pt,minimum size = 1.5mm]
    \node(a) at (0,0) {};
    \node(b) at (0,3){};
    \node(c) at (3,3) {};
    \node(d) at (3,0){};

    \tikzstyle{every node} = [draw = none,fill = none,scale = .6]
    \draw (a) -- (b);
    \draw (b) -- (c);
    \draw (c) -- (d);
    \draw (a) -- (c);
    \draw (a) -- (d);
    
    \node (o) at (0,3.6){{\large -2}};
    \node (o) at (0,-.6){{\large 0}};
    \node (o) at (3,3.6){{\large 2}};
    \node (o) at (3,-.6){{\large 0}};

\end{tikzpicture}
\hspace{0.45 cm}
\begin{tikzpicture}[scale= 0.3]
    
    \tikzstyle{every node} = [circle,fill,inner sep=1pt,minimum size = 1.5mm]
    \node(a) at (0,0) {};
    \node(b) at (0,3){};
    \node(c) at (3,3) {};
    \node(d) at (3,0){};

    \tikzstyle{every node} = [draw = none,fill = none,scale = .6]
    \draw (a) -- (b);
    \draw (b) -- (c);
    \draw (c) -- (d);
    \draw (a) -- (c);
    \draw (a) -- (d);
    
    \node (o) at (0,3.6){{\large -2}};
    \node (o) at (0,-.6){{\large 1}};
    \node (o) at (3,3.6){{\large 0}};
    \node (o) at (3,-.6){\large 1};

\end{tikzpicture}
\hspace{0.45 cm}
\begin{tikzpicture}[scale= 0.3]
    
    \tikzstyle{every node} = [circle,fill,inner sep=1pt,minimum size = 1.5mm]
    \node(a) at (0,0) {};
    \node(b) at (0,3){};
    \node(c) at (3,3) {};
    \node(d) at (3,0){};

    \tikzstyle{every node} = [draw = none,fill = none,scale = .6]
    \draw (a) -- (b);
    \draw (b) -- (c);
    \draw (c) -- (d);
    \draw (a) -- (c);
    \draw (a) -- (d);
    
    \node (o) at (0,3.6){\large -2};
    \node (o) at (0,-.6){\large 0};
    \node (o) at (3,3.6){\large 1};
    \node (o) at (3,-.6){\large 1};

\end{tikzpicture}

\end{center}
\caption{Above are the 8 spanning trees for the graph $K_4\setminus e$ as well as representatives for the 8 elements of $\Pic^0(K_4 \setminus e)$.\protect\footnotemark~By Theorem~\ref{thm:mtt}, these sets are always the same size. However, there is not a natural choice of bijection between them.}
\label{fig:mttexample}
\end{figure}

Theorem~\ref{thm:mtt} implies the existence of bijections between equivalence classes of $\Pic^0(G)$ and elements of $\mathcal T(G)$. However, in order to define sufficiently natural bijections, we will need to give $G$ some additional structure (see~\cite{Wagner} and \cite{CCG} for further discussion). Figure~\ref{fig:mttexample} demonstrates Theorem~\ref{thm:mtt} for a particular graph. 

\begin{definition}
A \emph{ribbon structure} $\chi$ on a graph $G$, is a map from each $v \in V(G)$ to a cyclic order of the edges incident to $v$. A \emph{ribbon graph} is a pair $(G,\chi)$, where $G$ is a graph and $\chi$ is a ribbon structure on $G$. 
\end{definition}

Ribbon graphs are also called \emph{combinatorial maps}. Suppose that $G$ can be drawn in a plane with no crossings such that for every $v \in V(G)$, $\chi(v)$ gives the edges incident to $v$ in counterclockwise order. Then, $(G,\chi)$ is called a \emph{plane graph} or a \emph{planar embedding}. A graph $G$ is called \emph{planar} if there exists some $\chi$ such that $(G,\chi)$ is a plane graph. 

Suppose that $x \in V(G)$ and $e \in E(G)$ are incident (i.e., $x \in \i(e)$). We will write $\chi(x,e)$ for the edge after $e$ in the cyclic order $\chi(x)$. 

\begin{definition}
\label{def::auto}
Let $(G,\chi)$ and $(G',\chi')$ be two ribbon graphs. A \emph{ribbon graph isomorphism} is a bijection, $\phi: V(G)\cup E(G) \to V(G')\cup E(G')$ such that $\phi(V(G)) = V(G')$, $\phi(E(G)) = E(G')$, and for any $e \in E(G)$,  the following properties are satisfied (where $\i(e) = \{x,y\}$):

\begin{enumerate}
    \item $\i(\phi(e)) = \{\phi(x),\phi(y)\}$.
    \item $\phi(\chi(x,e)) = \chi'(\phi(x),\phi(e))$.
    \item $\phi(\chi(y,e)) = \chi'(\phi(y),\phi(e))$.
\end{enumerate}

\noindent A \emph{ribbon graph automorphism} is a ribbon graph isomorphism from a ribbon graph to itself.
\end{definition}

\footnotetext{The representatives are the \emph{superstable divisors} with a particular \emph{sink}. For more information on choosing a distinguished set of representatives for the sandpile group, see~\cite[Chapter 2]{Klivans}.}

\begin{lemma}\label{lem:ribiso}
Any ribbon graph isomorphism $\phi$ from $(G,\chi)$ to $(G',\chi')$ induces an bijection from $\T(G)$ to $\T(G')$. Furthermore, \(\phi\) induces an isomorphism $\phi_{\rm D}$ from $\Div^0(G)$ to $\Div^0(G')$ which also generates isomorphisms from $\Div^0_s(G)$ to $\Div^0_{\phi(s)}(G')$ and $\Pic^0(G)$ to $\Pic^0(G')$.

\end{lemma}
\begin{proof}
Let $\phi$ be the isomorphism in question. The bijection between $\T(G)$ and $\T(G')$ follows from the fact that $\phi$ induces a graph isomorphism between $G$ and $G'$. Next, we define $\phi_{\rm D}$ by the map

\[\sum_{v \in V(G)} n_v v \mapsto \sum_{v' \in V(G')} n_{\phi^{-1}(v')} v'.\] 

It is simple to check that $\phi_{\rm D}$ is an isomorphism between $\Div^0(G)$ and $\Div^0(G')$. Because $\Div_s^0(G)$ and $\Div_{\phi(s)}^0(G')$ are respective subsets of $\Div^0(G)$ and $\Div^0(G')$ which respect group operations, it follows that $\phi_{\rm D}|_{\Div_s^0(G)}$ is an isomorphism between $\Div_s^0(G)$ and $\Div_{\phi(s)}^0(G')$.

Finally, if $\Delta$ and $\Delta'$ denote the respective Laplacian matrices of $G$ and $G'$, then $\phi_{\rm D}$ is also an isomorphism on the subgroups of $\Div^0(G)$ and $\Div^0(G')$ restricted to $\im_\Z(\Delta)$ and $\im_\Z(\Delta')$ respectively. Thus, $\phi_{\rm D}$ induces an isomorphism between $\Pic^0(G)$ and $\Pic^0(G')$.
\end{proof}

In fact, Lemma~\ref{lem:ribiso} holds even if $\phi$ is a graph isomorphism that does not respect the ribbon structures of $G$ and $G'$. However, we work exclusively with ribbon graph isomorphisms throughout this article due to the fact that the ribbon structures of graphs play a key roll in the definition of sandpile torsor actions and algorithms (see the next subsection).

\subsection{Sandpile Torsor Actions and Sandpile Torsor Algorithms}
Given a ribbon graph $(G,\chi)$ and a vertex $s \in V(G)$, we will work with a function $\what \alpha_{(G,\chi,s)}: \Div^0_s(G)\times \T(G) \to \T(G)$ which belongs to a specific class of \emph{monoid actions} (i.e. maps satisfying Properties 2 and 3 in Definition~\ref{def:torsor})

\begin{definition}\label{def:torsor}
A function of the form $\what \alpha_{(G,\chi,s)}$ is called a \emph{sink-parameterized sandpile torsor action (on ribbon graphs)}  if for any $T\in \T(G)$, and $D,D' \in \Div_s^0(G)$, the following properties are satisfied:
\begin{enumerate}
\item $\what \alpha_{(G,\chi,s)}(D,T) = \what\alpha_{(G,\chi,s)}(D',T) \Longleftrightarrow [D] = [D'].$
\item $\what \alpha_{(G,\chi,s)}(0,T) = T$.
    \item  $\what \alpha_{(G,\chi,s)}(D+D',T) = \what \alpha_{(G,\chi,s)}(D,\what \alpha_{(G,\chi,s)}(D',T))$. 
    \item For any $T' \in \T(G)$, there exists some $D'' \in \Div_s^0(G)$ such that $\what \alpha_{(G,\chi,s)}(D'',T) = T'$. 
\end{enumerate}
\end{definition}

\begin{definition}
The function $\what \alpha$ is called a \emph{sink-parameterized sandpile torsor algorithm (on ribbon graphs)} if for any ribbon graph $(G,\chi)$ and $s \in V(G)$, the following conditions hold:
\begin{enumerate}
\item  $\what \alpha_{(G,\chi,s)}$ is a sink-parameterized sandpile torsor action.

\item For any ribbon graph isomorphism $\phi$ from $(G,\chi)$ to $(G',\chi')$, and any $D \in \Div_s^0(G)$,
\[\phi(\what \alpha_{(G,\chi,s)}(D,T)) = \what \alpha_{(G',\chi',\phi(s))}(\phi_{\rm D}(D),\phi(T)).\]
\end{enumerate}
\end{definition}

\begin{definition}\label{def:STA}
For every \emph{plane graph} $(G,\chi)$, let $\alpha_{(G,\chi)}$ be an action of $\Pic^0(G)$ on $\T(G)$. We say $\alpha$ is a \emph{sandpile torsor algorithm (on plane graphs)} if there is a sink-parameterized sandpile torsor algorithm $\what \alpha$ such that for every \emph{plane graph} $(G,\chi)$, $s \in V(G)$, $D \in \Div_s^0(G)$, and $T \in \T(G)$, 
\[\alpha_{(G,\chi)}([D],T) = \what \alpha_{(G,\chi,s)}(D,T).\]
We call a specific action of the form $\alpha_{(G,\chi)}$ a \emph{sandpile torsor action}. 
\end{definition}

Notice that a sandpile torsor algorithm $\alpha$ is uniquely defined by a sink-parameterized sandpile torsor algorithm $\what \alpha$ such that for every plane graph $(G,\chi)$, $s,s' \in V(G)$, $D \in \Div_s^0(G)$, $D' \in \Div_{s'}^0(G)$, and $T \in \T(G)$, 
\begin{equation}\label{eq:sinkinvariant}\what\alpha_{(G,\chi,s)}(D,T) = \what \alpha_{(G,\chi,s')}(D',T) \Longleftrightarrow [D] = [D'].\end{equation}

\begin{remark}
It is important that we restrict to \emph{plane graphs} in Definition~\ref{def:STA} since we claim that it is impossible for \eqref{eq:sinkinvariant} to be satisfied by all \emph{ribbon graphs}. 

To prove this claim, let $(G,\chi)$ be a non-planar ribbon graph with 2 vertices $x$ and $y$ and 3 edges between them. Let $\phi$ be the ribbon graph automorphism which swaps $x$ and $y$ and act as the identity on the edges. It is straightforward to see that the induced automorphism $\phi_D$ maps each divisor to its inverse.

Suppose that $\what{\alpha}$ is a sink-parameterized sandpile torsor algorithm. By Definition 2.9, this map must satisfy the equality $\what \alpha_{(G,\chi,x)}(D,T) = \what \alpha_{(G,\chi,y)}(-D,T)$ for all $D \in \Div_x^0(G)$. However, $[D] \not= [-D]$ unless $[D]$ is the identity. In particular, it is not possible for $\what{\alpha}$ to satisfy~\eqref{eq:sinkinvariant}. This example is also discussed in~\cite[Figure 1]{alextorsors}. 
\end{remark}

%Notice that a sandpile torsor algorithm $\alpha$ is uniquely defined by a sink-parameterized sandpile torsor algorithm $\what \alpha$ which satisfies the following condition: For every plane graph $(G,\chi)$, $s,s' \in V(G)$, $D \in \Div_s^0(G)$, $D' \in \Div_{s'}^0(G)$, and $T \in \T(G)$, we have:

Two well-known sink-parameterized sandpile torsor algorithms are the \emph{rotor-routing algorithm} and the \emph{Bernardi algorithm}. These algorithms are distinct in general, and their constructions appear unrelated~\cite[Figure 9]{BW}.\footnote{In fact, Baker and Wang conjectured that for every non-planar ribbon graph, there is a sink vertex where the two algorithms are distinct~\cite{BW}. This conjecture was independently proven by Ding as well as by Shokrieh and Wright~\cite{Ding,SW}. } However, both algorithms satisfy \eqref{eq:sinkinvariant}~\cite{CCG,BW}, and they both define the same sandpile torsor algorithm~\cite[Theorem 7.1]{BW}. As discussed in Section~\ref{sec:intro}, this surprising equivalence inspired Conjecture~\ref{conj:klivans}, which asks if this sandpile torsor algorithm is in some sense unique. We will not explicitly work with the Bernardi algorithm in this paper, but we encourage the curious reader to see~\cite{BW} for its construction, as well as~\cite{Yuengeometric} and~\cite{Lilla} for alternate perspectives when restricting to plane graphs. 

\begin{remark}
It is possible to define sandpile torsor algorithms directly instead of first defining sink-parameterized sandpile torsor algorithms. However, it is convenient to introduce sink-parameterized sandpile torsor algorithms because we make heavy use of the rotor-routing algorithm (defined in Section \ref{sec:rr}), which is traditionally defined on ribbon graphs that have a sink vertex. In fact, the proof of one of our main results, the consistency of the rotor-routing algorithm (Theorem \ref{thm:rrconsistent}), involves the application of the rotor-routing algorithm to non-planar ribbon graphs.
\end{remark}

\subsection{Rotor-Routing}
\label{sec:rr}

Let $(G,\chi)$ be a ribbon graph and $s \in V(G)$. 
\begin{definition}
A \emph{rotor configuration with sink $s$} is an assignment of an incident edge to every vertex of $G$ except for $s$. Given a rotor configuration $\rho$ and $x \in V(G) \setminus s$, we write $\rho\la x \ra$ for the edge assigned to $x$. 
\end{definition}

We refer to edges of the form $\rho\la x \ra$ as \emph{rotors}. It is useful to visualize a rotor $\rho\la x\ra$ as a directed edge incident to $x$ and oriented away from $x$. We now define a function $r_x$ which rotates the rotor at $x$ one position. 

\begin{definition}\label{def:onestep}
Let $\rho$ be a rotor configuration with sink $s$ and $x \in V(G) \setminus s$. We write $r_x(\rho)$ for the rotor configuration defined by:
\[r_x(\rho)\la v \ra = \begin{cases} \chi(v,\rho\la v \ra) & \text{ if $v = x$, } \\ \rho\la v \ra & \text{ otherwise.}
\end{cases}\]
\end{definition}

\begin{definition}
We say that a collection of rotors forms a \emph{directed cycle} $C$ if there is a path along them that returns to its original vertex. Given a directed cycle $C$, we write $E(C)$ for the set of edges corresponding to the rotors that make up $C$ and $V(C)$ for the set of vertices incident to these edges. A collection of rotors is \emph{acyclic} if it contains no cycles. 
\end{definition}

\begin{definition}\label{def:Ts}
Let $s \in V(G)$ and $T \in \T(G)$. We write $T_s$ for the rotor configuration with sink $s$ such that for every $x \in V(G) \setminus s$, the rotor $T_s\la x\ra$ is the unique edge in $T$ incident to $x$ along the path from $x$ to $s$.
\end{definition}

In other words, $T_s$ describes the orientation of the edges of $T$ such that each edge is directed toward $s$. The following lemma is immediate.  
\begin{lemma}\label{lem:acyclic} 
Let $\rho$ be a rotor configuration with sink $s$ on a ribbon graph $(G,\chi)$. There exists some $T \in \T(G)$ such that $T_s = \rho$ if and only if $\rho$ is acyclic. In particular, $T$ is the set of edges which form the rotors of $\rho$. 
\end{lemma}

Given a sandpile element of the form $c-s \in \Div^0_s(G)$ for some $c,s\in V(G)$, and a spanning tree $T \in \T(G)$, Algorithm \ref{alg:rotor} outputs a rotor configuration with sink $s$. At each step of this algorithm, a single rotor rotates according to the ribbon structure $\chi$ (on a plane graph, this would be a counterclockwise rotation). When the algorithm terminates, we obtain a new rotor configuration. See Figure~\ref{fig:consistent}(a) for a demonstration of Algorithm~\ref{alg:rotor}. 

\begin{algorithm}
\caption{Single-chip rotor-routing}\label{alg:rotor}
\begin{algorithmic}
\State \textbf{Input:} A tree $T\in \T(G)$ and a divisor of the form $c-s \in \Div_s^0(G)$. 
\State \textbf{Output:} The rotor configuration $\rho$ with sink vertex $s$.
\State Initialize the rotor configuration $\rho = T_s$.
\State Initialize a ``traveling vertex''  $x$ at $c$. We will call $x$ a \emph{chip}. 
\While{ $x \not= s$}
\State Replace $\rho$ with $r_x(\rho)$. \Comment{rotate the rotor at $x$ by one position.}
\State Replace $x$ with the other vertex incident to $\rho \la x \ra$. \Comment{move $x$ across the rotor.}
\EndWhile\\
\Return $\rho$.
\end{algorithmic}
\end{algorithm}

There is also a variant of Algorithm~\ref{alg:rotor} involving \emph{unicycles}, which we will discuss in Appendix~\ref{app:unicycles}. 

Algorithm~\ref{alg:rotor} was first explored by Priezzhev et al. under the name \emph{Eulerian walkers model}~\cite{eulerianwalkers}. By \cite[Lemma 3.10]{Holroyd}, the final rotor $\rho$ in Algorithm~\ref{alg:rotor} is acyclic. By Lemma~\ref{lem:acyclic}, this means there is a unique tree $T'\in \T(G)$ such that $T'_s = \rho$. For the inputs $c,s \in V(G)$ and $T \in \T(G)$, we define the mapping $\what r_{(G,\chi,s)}(c-s,T) := T'$ where $T'_s$ is the output of Algorithm~\ref{alg:rotor}.

Holroyd et al. showed how this model could be used to define a sink-parameterized sandpile torsor action on any ribbon graph $(G,\chi)$. Fix any $D = \sum_{v \in V(G)}n_vv\in \Div^0_s(G)$ and let $N := -n_s$. Note that $D = \sum_{i=1}^{N} (c_i - s)$ for some sequence of vertices $(c_1,\dots,c_{N})$. Let $T^0 = T$, and for each $i \in [1,N]$, let $T^i = \what r_{(G,\chi,s)}(c_i-s,T^{i-1})$. Then, we define $\what r_{(G,\chi,s)}(D,T) := T^{N}$. By \cite[Corollary 2.6]{Holroyd}, $T^{N}$ does not depend on the ordering of $(c_1,\dots,c_{N})$. Thus, $\what r_{(G,\chi,s)}$ is a well-defined monoid action of $\Div^0_s(G)$ on $\T(G)$. 

\begin{theorem}{\cite[Theorem 2.5]{LL}}
The rotor-routing algorithm $\what r$ constructed above is a sink-parameterized sandpile torsor algorithm. 
\end{theorem}

Chan, Church, and Grochow proved that, on plane graphs, the rotor-routing action does not depend on the choice of sink vertex, which implies the following result.

\begin{theorem}{\cite[Theorem 2]{CCG}}\label{thm:rristorsor}
The rotor-routing algorithm $\what r$ satisfies (\ref{eq:sinkinvariant}). In particular, it defines a sandpile torsor algorithm $r$. 
\end{theorem}

%Let $(G,\chi)$ be a plane graph, $S \in \Pic^0(G)$, and $T \in \T(G)$. In order to compute $r_{(G,\chi)}(S,T)$, we first choose any $D$ in $\Div^0(G)$ which represents $S$. We then write $D$ in the form Theorem~\ref{thm:rristorsor} tells us that 

Theorem~\ref{thm:rristorsor} tells us that for any plane graph $(G,\chi)$, there is a well-defined free transitive action $r_{(G,\chi)}$ of $\Pic^0(G)$ on $\T(G)$. To evaluate $r_{(G,\chi)}(S,T)$ for some $S \in \Pic^0(G)$ and $T \in \T(G)$, we first choose a sink vertex $s$ and some $D \in \Div_s^0(G)$ such that $[D] = S$ (such a $D$ must exist by Lemma~\ref{lem:anysink}). Then, we calculate $r_{(G,\chi)}(S,T) = \what r_{(G,\chi,s)}(D,T)$ through repeated applications of Algorithm~\ref{alg:rotor}.

Throughout this paper, we will use the term \emph{rotor-routing algorithm} to refer to both $\what r$ and $r$. Context should resolve any potential ambiguity. 

\section{Useful Properties of Rotor-Routing}\label{sec:funfacts}

In the previous section, we defined the rotor-routing sandpile torsor algorithm. In this section, we introduce several properties of rotor-routing that we will use in later sections. Some of our proofs for results in this section require additional terminology which we will delay until Appendix~\ref{app:unicycles}. As in the previous section, $G$ is a finite connected graph which has no loops but may have multiple edges.  

\begin{lemma}\label{lem:adjverts}
For any $D \in \Div^0(G)$, we can write
\[D = \sum_{i=1}^k (c_i-s_i),\]
where for every $i$, there exists an edge $f_i$ such that $\i(f_i) = (c_i,s_i)$. 
\end{lemma}
\begin{proof}
This lemma follows immediately from the definition of $\Div^0(G)$ and the fact that $G$ is connected.
\end{proof}

\begin{corollary}\label{cor:adjverts}
Suppose that $\alpha$ and $\beta$ are sandpile torsor algorithms such that for all plane graphs $(G,\chi)$, for all $f \in E(G)$, and for all $T \in \T(G)$, we have, 
\[\alpha_{(G,\chi)}([c-s],T) = \beta_{(G,\chi)}([c-s],T),\]
where $\{c,s\} = \i(f)$. Then $\alpha = \beta$. 
\end{corollary}

Corollary~\ref{cor:adjverts} will allow us to focus on the case where the chip and sink are adjacent. 
\begin{prop}\label{prop:nooverspins}
If there is some $f \in E(G)$ such that $\i(f) = \{c,s\}$, then no rotor can make more than a full turn and no edge may be crossed more than once when applying Algorithm~\ref{alg:rotor} with input $T$ and $c-s$. 
\end{prop}
Proposition~\ref{prop:nooverspins} follows from~\cite[Lemma 4.9]{Holroyd}, which we will prove Appendix~\ref{app:unicycles}. Note that this result requires an edge between $c$ and $s$ in $G$. 

Proposition~\ref{prop:nooverspins} holds for any ribbon graph. Next, we will give two lemmas that require planarity. We will prove these lemmas in Appendix~\ref{app:unicycles} using ideas from~\cite{CCG}. 

A defining fact about plane graphs is that any directed cycle $C$ partitions the vertices $V(G) \setminus V(C)$  and edges $E(G) \setminus E(C)$ into two classes. In particular, they are either to the ``left'' of $C$ or to the ``right'' of $C$ (where the ``left'' edges and vertices are in the interior of a cycle oriented counterclockwise or the exterior of a cycle oriented clockwise). Note that the edge case where $C$ is a single bi-directed edge will never be an issue.

Let $(G,\chi)$ be a plane graph, $f \in E(G)$, $\i(f) = \{c,s\}$, and $T \in \T(G)$. Suppose that $f\not\in T$, but $T_s$ contains some path $x_1,x_2,\dots,x_k$ of rotors from $c$ to $s$. Let $C$ be the directed cycle formed by $\{f,x_1,x_2,\dots,x_k\}$. 

\begin{lemma}\label{lem:looprev}
In the construction above, if Algorithm~\ref{alg:rotor} is applied with input $T$ and $c-s$, then the chip never crosses any edges to the right of $C$. 
\end{lemma}

Let $(G,\chi)$ be a plane graph, $f \in E(G)$, $\i(f) = \{c,s\}$, and $T \in \T(G)$. Apply Algorithm~\ref{alg:rotor} with input $T$ and $c-s$. Suppose that it takes $n$ steps through the while loop for the algorithm to terminate. For $k \in [0,n]$, let $c_k$ and $\rho_k$ be the chip location and rotor configuration respectively after $k$ passes through the while loop of Algorithm~\ref{alg:rotor}. 

\begin{lemma}\label{lem:looprev2}
In the construction above, suppose that for some $i \in [0,n]$, $\rho_i$ includes a directed cycle $C$. Then there exists some $j \in [0,n]$ such that $\rho_j$ includes all of the rotors of $C$ in reverse order.
\end{lemma}
See Appendix~\ref{app:unicycles} for proofs of Lemmas~\ref{lem:looprev} and~\ref{lem:looprev2}. 

We conclude this section with a technical lemma which gives conditions where the cyclic order of edges around a vertex can be rearranged without affecting specific rotor-routing outputs. This lemma will be used repeatedly in the proof of Theorem~\ref{thm:rrconsistent}. 

Let $c,s \in V(G)$ and $f \in E(G)$ with $\i(f) = \{c,s\}$. Let $T' = r_{(G,\chi)}(\onechip,T)$. Consider any $x \in V(G)$. We can write $\chi(x)$ in the form $(T_s\la x\ra, e_1,e_2,\dots, e_p, T'_s\la x\ra, \what e_1,\dots, \what e_l)$. Let $\sigma$ be a permutation on the set $[1,p]$ and let $\pi$ be a permutation on the set $[1,l]$. Define $\wtilde{\chi}$ to be a ribbon structure such that $\wtilde{\chi}(x) = (T_s\la x\ra, e_{\sigma(1)},e_{\sigma(2)},\dots, e_{\sigma(p)}, T'_s\la x\ra, \what e_{\pi(1)},\dots, \what e_{\pi(l)})$ and $\wtilde{\chi}(v) = \chi(v)$ for all other $v \in V(G)$. 

\begin{lemma}\label{lem:rearrange}
In the construction above, we get the following equality: 
\[\what r_{(G,\wtilde\chi,s)}(c-s,T) = \what r_{(G,\chi,s)}(c-s,T) = r_{(G,\chi)}([c-s],T).\] 
\end{lemma}
\begin{proof}
We will prove the lemma by contradiction. The second equality follows from Definition~\ref{def:STA}. For the first equality, suppose that for some vertex $v \in V(G)$, the rotor at $v$ moves past $T'_s\la v\ra$ on $(G,\wtilde\chi)$ (otherwise the proof works analogously after switching the roles of $(G,\chi)$ and $(G,\wtilde\chi)$). Consider the first instance where such an event occurs. At this moment, the chip must have reached $v$ more times than it did during rotor-routing on $(G,\chi)$. Because no edge is crossed more than once in each direction (by Proposition~\ref{prop:nooverspins}), the chip must enter $v$ along an edge that it did not cross during rotor-routing on $(G,\chi)$. It follows that the rotor on some vertex $y$ that is adjacent to $v$ must have already passed $T'_s \la y \ra$. This contradicts our minimality assumption. 
\end{proof}

\section{Rotor-Routing is Consistent}\label{sec:consistent}

In this section, we introduce the notion of \emph{consistency} of sandpile torsor algorithms and show that the rotor-routing algorithm is consistent. 

Throughout this section, $(G,\chi)$ is a plane graph, $f \in E(G)$, and $\i(f) = \{c,s\}$. For $e \in E(G)$, we write $G \setminus e$ for the graph we obtain from $G$ by \emph{deleting} edge $e$ and $G/e$ for the graph we obtain by \emph{contracting} the edge $e$. After contracting $e$, we also remove all loop edges because these cannot occur on any spanning tree and have no effect on the rotor-routing algorithm. 

\begin{definition}\label{def:condelchi} Let $(G,\chi)$ be a ribbon graph with $e \in E(G)$, where $\i(e) = \{x,y\}$. 
\begin{itemize}
    \item Define $\chi\setminus e$ to be equal to $\chi$, except with $e$ removed from $\chi(x)$ and $\chi(y)$.
    \item Suppose that after removing edges parallel to $e$, $\chi(x) = (e,e_1,e_2,\dots,e_p)$ and $\chi(y) = (e,\what e_1,\what e_2,\dots,\what e_l)$. Define $\chi/e$ to be equal to $\chi$ except that on the contracted edge $z$, we have 
    \[(\chi/e)(z) := (e_1,e_2,\dots,e_p,\what e_1,\what e_2,\dots,\what e_l).\] 
\end{itemize}
\end{definition}

The following lemma is a simple exercise. 

\begin{lemma}
If $(G,\chi)$ is a plane graph, then $(G\setminus e,\chi \setminus e)$ and $(G/ e,\chi /e)$ are also plane graphs.
\end{lemma}

We will refer to any plane graph $(G',\chi')$ that can be obtained by contracting and deleting edges of $(G,\chi)$ as a \emph{minor} of the plane graph $(G,\chi)$. Note that whenever $(G',\chi')$ is a minor of $(G,\chi)$, there is a natural surjective map from $V(G)$ to $V(G')$ and from $E(G)$ to $E(G')$. When there is no ambiguity, we will use the same labeling for the vertices and edges of $G$ as we do for their image under this map. 

\begin{definition}\label{def:consistent}
A sandpile torsor algorithm $\alpha$ is \emph{consistent} if for every plane graph $(G,\chi)$, every choice of $f \in E(G)$, and every choice of $T \in \T(G)$, the following three properties hold (where we define $\{c,s\} = \i(f)$):
\begin{enumerate}
\item For any $e \in E(G)$ such that $\i(e) \neq \{c,s\}$, if $e \in T\cap \alpha_{(G,\chi)}(\onechip,T)$, then
    \[\alpha_{(G,\chi)}(\onechip,T)\setminus {e} = \alpha_{(G/ {e},\chi/e)}(\onechip,T\setminus e).\]
\item For any $e \in E(G)$, if $e \notin T\cup \alpha_{(G,\chi)}(\onechip,T)$, then
    \[ \alpha_{(G,\chi)}(\onechip,T) = \alpha_{(G\setminus {e},\chi\setminus {e})}(\onechip,T).\]
    \item For any $e \in E(G)\setminus f$, if there is a cut vertex $x$ such that all paths from $e$ to $f$ pass through $x$, then
    \[ e \in T \Longleftrightarrow e \in \alpha_{(G,\chi)}(\onechip, T).\]
    \end{enumerate}
\end{definition}

\begin{figure}
\begin{center}
\begin{tikzpicture}
    
    \tikzstyle{pt} = [circle,fill,inner sep=1pt,minimum size = 1.5mm]
    \tikzstyle{coin} = [draw,circle,inner sep=1pt,minimum size = 2mm]
    
    \begin{scope}[shift={(0,0)}]
    \node[pt] (ha) at (0,0) {};
    \node[pt](hb) at (1,0){} ;
    \node[pt,label={west:{$s$}}](hc) at (0,1){};
    \node[pt,label={east:{$c$}}](hd) at (1,1){};
    \end{scope} 
    
    \begin{scope}[shift={(2.5,0)}]
    \node[pt] (a) at (0,0) {};
    \node[pt](b) at (1,0){} ;
    \node[pt](c) at (0,1){};
    \node[coin](d) at (1,1){};
    \end{scope} 
    
    \begin{scope}[shift={(5,0)}]
    \node[coin](2a) at (0,0) {};
    \node[pt](2b) at (1,0){} ;
    \node[pt](2c) at (0,1){};
    \node[pt](2d) at (1,1){};
    \end{scope}
    
    \begin{scope}[shift={(7.5,0)}]
    \node[pt](3a) at (0,0) {};
    \node[pt](3b) at (1,0){} ;
    \node[coin](3c) at (0,1){};
    \node[pt](3d) at (1,1){};
    \end{scope}
    
    \begin{scope}[shift={(10,0)}]
    \node[pt](3ea) at (0,0) {};
    \node[pt](3eb) at (1,0){} ;
    \node[pt,label={west:{$s$}}](3ec) at (0,1){};
    \node[pt,label={east:{$c$}}](3ed) at (1,1){};
    \end{scope}
    
    \begin{scope}[shift={(0,-3)}]
    \node[pt] (4ha) at (0,0) {};
    \node[pt,label={west:{$s$}}](4hc) at (0,1){};
    \node[pt,label={east:{$c$}}](4hd) at (1,1){};
    \end{scope}
    
    \begin{scope}[shift={(2.5,-3)}]
    \node[pt](4a) at (0,0) {};
    \node[pt](4c) at (0,1){};
    \node[coin](4d) at (1,1){};
    \end{scope}
    
    \begin{scope}[shift={(5,-3)}]
    \node[coin](5a) at (0,0) {};
    \node[pt](5c) at (0,1){};
    \node[pt](5d) at (1,1){};
    \end{scope}
    
    \begin{scope}[shift={(7.5,-3)}]
    \node[pt](6a) at (0,0) {};
    \node[coin](6c) at (0,1){};
    \node[pt](6d) at (1,1){};
    \end{scope}
    
    \begin{scope}[shift={(10,-3)}]
    \node[pt](6ea) at (0,0) {};
    \node[pt,label={west:{$s$}}](6ec) at (0,1){};
    \node[pt,label={east:{$c$}}](6ed) at (1,1){};
    \end{scope}
    
    \begin{scope}[shift={(0,-6)}]
    \node[pt](7ha) at (0,0) {};
    \node[pt](7hb) at (1,0){} ;
    \node[pt,label={west:{$s$}}](7hc) at (0,1){};
    \node[pt,label={east:{$c$}}](7hd) at (1,1){};
    \end{scope}
    
    \begin{scope}[shift={(2.5,-6)}]
    \node[pt](7a) at (0,0) {};
    \node[pt](7b) at (1,0){} ;
    \node[pt](7c) at (0,1){};
    \node[coin](7d) at (1,1){};
    \end{scope}
    
    \begin{scope}[shift={(5,-6)}]
    \node[coin](8a) at (0,0) {};
    \node[pt](8b) at (1,0){} ;
    \node[pt](8c) at (0,1){};
    \node[pt](8d) at (1,1){};
    \end{scope}
    
    \begin{scope}[shift={(7.5,-6)}]
    \node[pt](9a) at (0,0) {};
    \node[pt](9b) at (1,0){} ;
    \node[coin](9c) at (0,1){};
    \node[pt](9d) at (1,1){};
    \end{scope}
    
    \begin{scope}[shift={(10,-6)}]
    \node[pt](9ea) at (0,0) {};
    \node[pt](9eb) at (1,0){} ;
    \node[pt,label={west:{$s$}}](9ec) at (0,1){};
    \node[pt,label={east:{$c$}}](9ed) at (1,1){};
    \end{scope}

    %\begin{scope}[shift={(0,-3)}]
    %\node[label={west:{\large$s$}}] (4a) at (0,0) {};
    %\node(4b) at (1,0){} ;
    %\node(4c) at (2,0){};
    %\node(4d) at (3,0){};
    
    %\node(4a2) at (0,1) {};
    %\node(4b2) at (1,1){} ;
    %\node(4c2) at (2,1){};
    %\node(4d2) at (3,1){};
    
    %\node(4a3) at (0,2) {};
    %\node(4b3) at (1,2){} ;
    %\node(4c3) at (2,2){};
    %\node(4d3) at (3,2){};
    %\end{scope}
    
    %\begin{scope}[shift={(5,-3)}]
    %\node[label={west:{\large$s$}}] (5a) at (0,0) {};
    %\node(5b) at (1,0){} ;
    %\node(5c) at (2,0){};
    %\node(5d) at (3,0){};
    
    %\node(5a2) at (0,1) {};
    %\node(5b2) at (1,1){} ;
    %\node(5c2) at (2,1){};
    %\node(5d2) at (3,1){};
    
    %\node(5a3) at (0,2) {};
    %\node(5b3) at (1,2){} ;
    %\node(5c3) at (2,2){};
    %\node(5d3) at (3,2){};
    %\end{scope}
    
    %\begin{scope}[shift={(10,-3)}]
    %\node[label={west:{\large$s$}}] (6a) at (0,0) {};
    %\node(6b) at (1,0){} ;
    %\node(6c) at (2,0){};
    %\node(6d) at (3,0){};
    
    %\node(6a2) at (0,1) {};
    %\node(6b2) at (1,1){} ;
    %\node(6c2) at (2,1){};
    %\node(6d2) at (3,1){};
    
    %\node(6a3) at (0,2) {};
    %\node(6b3) at (1,2){} ;
    %\node(6c3) at (2,2){};
    %\node(6d3) at (3,2){};
    %\end{scope}

%%%%%%%%%%%%%%%%%%%%%%%%%%%%%%%%%%%%%%%%%%%%%%%%%%%%%%%%%%%%%5
    \tikzstyle{every node} = [draw = none,fill = none,scale = .8]
    
    \node(o) at (0.5,1.5){\Large$T$}; \node(o) at (10.5,1.5) {\Large$T'$};
    
    \node(o) at (0.5,-1.5){\Large$T\setminus {e_1}$};
    \node(o) at (10.5,-1.5) {\Large$T'\setminus {e_1}$};
    
    \node(o) at (0.5,-4.5){\Large$T$};
    \node(o) at (10.5,-4.5){\Large$T'$};
    
    \node(o) at (-1.5,0.5){\Large(a):};
    \node(o) at (-1.5,-2.5){\Large(b):};
    \node(o) at (-1.5,-5.3){\Large(c):};
    
    \draw[dotted] (hc) -- (ha) -- node [below]{$e_2$}(hb);
    
    \draw[dotted] (c) -- (a) -- (b);
    
    \draw[dotted] (2d) -- (2c) -- (2a) -- (2b);
    
    \draw[dotted] (3d) -- (3c);
    \draw[dotted] (3a) -- (3b);
    
    \draw[dotted] (3ed) -- (3ec);
    \draw[dotted] (3ea) -- (3eb);
    
    \draw[dotted] (4hc) -- (4ha);
    \draw[dotted] (4ha) to [bend right = 45] (4hd);
    
    \draw[dotted] (4c) -- (4a);
    \draw[dotted] (4a) to [bend right = 45] (4d);
    
    \draw[dotted] (5d) -- (5c) -- (5a);
    \draw[dotted] (5a) to [bend right = 45] (5d);
    
    \draw[dotted] (6d) -- (6c);
    \draw[dotted] (6a) to [bend right = 45] (6d);
    
    \draw[dotted] (6ed) -- (6ec);
    \draw[dotted] (6ea) to [bend right = 45] (6ed);
    
    \draw[dotted] (7hc) -- (7ha);
    
    \draw[dotted] (7c) -- (7a);
    
    \draw[dotted] (8d) -- (8c) -- (8a);
    
    \draw[dotted] (9d) -- (9c);
    
    \draw[dotted] (9ed) -- (9ec);
    
    %\draw[dotted](4a) --(4d)--(4d3) -- (4a3) -- (4a);
    %\draw[dotted](4b) --(4b3);
    %\draw[dotted](4c) --(4c3);
    %\draw[dotted](4a2) --(4d2);
    
    %\draw[dotted](5a) --(5d)--(5d3) -- (5a3) -- (5a);
    %\draw[dotted](5b) --(5b3);
    %\draw[dotted](5c) --(5c3);
    %\draw[dotted](5a2) --(5d2);
    
    %\draw[dotted](6a) --(6d)--(6d3) -- (6a3) -- (6a);
    %\draw[dotted](6b) --(6b3);
    %\draw[dotted](6c) --(6c3);
    %\draw[dotted](6a2) --(6d2);
    
    \draw (ha) -- (hd);
    \draw (hb) -- node[right]{$e_1$}(hd);
    \draw (hd) -- (hc) ;
    
    \draw [-{Latex[length=2mm,width=1.5mm]}](a) --(d);
    \draw [-{Latex[length=2mm,width=1.5mm]}](b) -- (d);
    \draw [-{Latex[length=2mm,width=1.5mm]}](d) -- (c) ;
    
    \draw [-{Latex[length=2mm,width=1.5mm]}](2a) --(2d);
    \draw [-{Latex[length=2mm,width=1.5mm]}](2b) --(2d);
    \draw [-{Latex[length=2mm,width=1.5mm]}](2d) --(2a);
    
    \draw [-{Latex[length=2mm,width=1.5mm]}](3a) --(3c);
    \draw [-{Latex[length=2mm,width=1.5mm]}](3b) --(3d);
    \draw [-{Latex[length=2mm,width=1.5mm]}](3d) --(3a);
    
    \draw (3ea) --(3ec);
    \draw (3eb) --(3ed);
    \draw (3ed) --(3ea);
    
    \draw (4ha) --(4hd);
    \draw (4hd) --(4hc);
    
    \draw [-{Latex[length=2mm,width=1.5mm]}](4a) --(4d);
    \draw [-{Latex[length=2mm,width=1.5mm]}](4d) --(4c);
    
    \draw [-{Latex[length=2mm,width=1.5mm]}](5a) --(5d);
    \draw [-{Latex[length=2mm,width=1.5mm]}](5d) --(5a);
    
    \draw [-{Latex[length=2mm,width=1.5mm]}](6a) --(6c);
    \draw [-{Latex[length=2mm,width=1.5mm]}](6d) --(6a);
    
    \draw (6ea) --(6ec);
    \draw (6ed) --(6ea);
    
    \draw (7ha) --(7hd);
    \draw (7hb) --(7hd);
    \draw (7hd) --(7hc);
    
    \draw [-{Latex[length=2mm,width=1.5mm]}](7a) --(7d);
    \draw [-{Latex[length=2mm,width=1.5mm]}](7b) --(7d);
    \draw [-{Latex[length=2mm,width=1.5mm]}](7d) --(7c);
    
    \draw [-{Latex[length=2mm,width=1.5mm]}](8a) --(8d);
    \draw [-{Latex[length=2mm,width=1.5mm]}](8b) --(8d);
    \draw [-{Latex[length=2mm,width=1.5mm]}](8d) --(8a);
    
    \draw [-{Latex[length=2mm,width=1.5mm]}](9a) --(9c);
    \draw [-{Latex[length=2mm,width=1.5mm]}](9b) --(9d);
    \draw [-{Latex[length=2mm,width=1.5mm]}](9d) --(9a);
    
    \draw (9ea) --(9ec);
    \draw (9eb) --(9ed);
    \draw (9ed) --(9ea);
    
\end{tikzpicture}
\caption{An illustration of the implementation and consistency of the rotor-routing algorithm on the graph $(G,\chi)$, where $\chi$ denotes counterclockwise rotation. We denote the chip by a hollow vertex. Figure~\ref{fig:consistent}(a) demonstrates how Algorithm \ref{alg:rotor} can be used to compute $T' := r_{(G,\chi)}([c-s],T)$. In Figure~\ref{fig:consistent}(b), the rotor-routing algorithm commutes with contraction: $T'/e_1 = r_{(G/e_1,\chi/e_1)}([c-s],T \setminus e_1)$ (see Definition \ref{def:consistent} (1)). In Figure~\ref{fig:consistent}(c), the rotor-routing algorithm commutes with deletion: $T' = r_{(G\setminus e_2,\chi\setminus e_2)}([c-s],T)$ (see Definition \ref{def:consistent} (2)). See Clips 8-10 from \url{https://www.youtube.com/watch?v=tSdVSk5o4Kg} for animated examples of consistency on a larger plane graph (note that the video uses a clockwise convention for rotor-routing).}
\label{fig:consistent}
\end{center}
\end{figure}
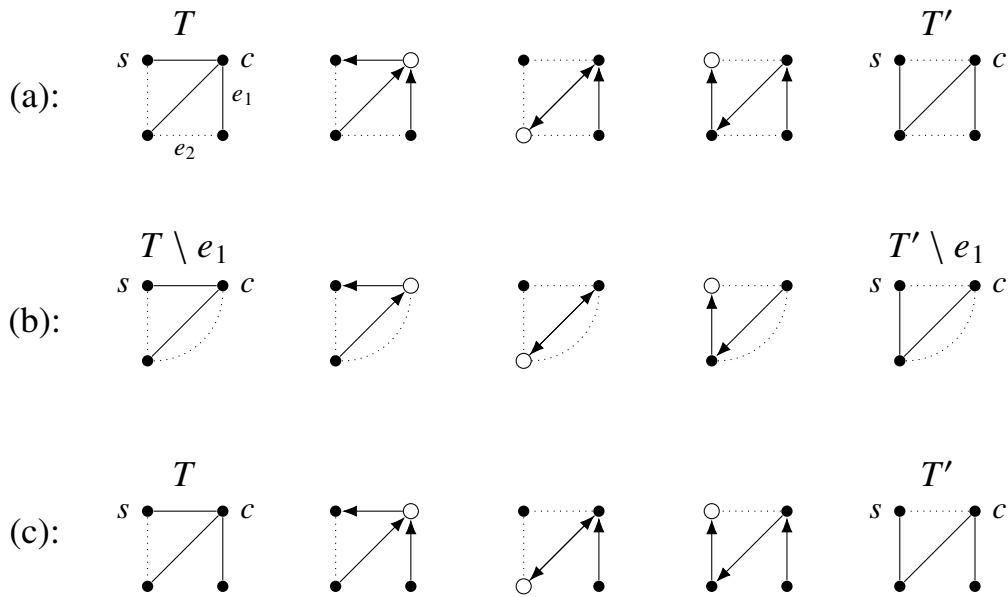

The first two conditions of Definition~\ref{def:consistent} are illustrated by Figure~\ref{fig:consistent}. The third condition says that we can treat two subgraphs joined at a vertex independently. 

\begin{prop}\label{prop:rrpart3}
Condition 3 of Definition~\ref{def:consistent} is satisfied by the rotor-routing algorithm. 
\end{prop}
\begin{proof}
Let $x$ be a cut vertex. We say $x$ \emph{separates} two edges or vertices $g$ and $h$ if all paths between $g$ and $h$ must contain $x$. We claim that there is some $D \in \Div^0_x(G)$ such that $[D] = [c-s]$ and $D = \sum_{v \in V} (v - x)$, where $V$ is a multiset of vertices which $x$ does not separate from $f$. Assume that the claim holds. Then, by Definition \ref{def:STA}, it follows that
\[r_{(G,\chi)}([c-s],T) = \what r_{(G,\chi,x)}(D,T).\]

When rotor-routing with sink $x$, all chips halt upon reaching $x$. As a result, any rotor that $x$ separates from $f$ will not rotate during the implementation of Algorithm \ref{alg:rotor}. It follows that for any edge $e$ that is separated from $f$, we have 
\[e \in T \Longleftrightarrow e \in \what r_{(G,\chi,x)}(D,T) \Longleftrightarrow e \in r_{(G,\chi)}([c-s],T).\]

To prove the claim, we use a modified version of the proof of Lemma~\ref{lem:anysink}. If $s = x$, we can take $D = c-s$. Otherwise, let $V_1 \subseteq V(G \setminus x)$ be the set of vertices in the same component as $s$ in $G \setminus x$, and let $V_2 = V(G \setminus x) \setminus V_1$. Let $\delta_x \in \Div_x^0(G)$ be the divisor which places $\deg(v)$ chips on every $v \in V_1$, 0 chips on every $w \in V_2$, and $-\sum_{v \in V_1} \deg(v)$ chips on $x$. Define $\delta_x^\circ$ to be the divisor we obtain after repeatedly firing any vertex $v$ with at least $\deg(v)$ many chips. Notice that $\delta_x - \delta_x^\circ$ has a positive value on each $v \in V_1$ and is equal to 0 on each $w \in V_2$. Furthermore, $[\delta_x - \delta_x^\circ]= 0$. We prove the claim by letting $D = c - s + \delta_x - \delta_x^\circ$. 
\end{proof}

By requiring consistent sandpile torsor algorithms to satisfy Condition 3, we simplify many of the arguments in Section~\ref{sec:unique}. However, we are not convinced that this condition is necessary. 
\begin{question}\label{quest:con3unneeded}
Is Condition 3 of Definition~\ref{def:consistent} implied by Conditions 1 and 2? 
\end{question}

\begin{theorem}\label{thm:rrconsistent} The rotor-routing sandpile torsor algorithm is consistent. 
\end{theorem}
\begin{proof}
We begin with a paragraph giving the outline of our argument. We reduce the result to 6 specific cases. If an edge $e$ is in both $T$ and $T' := r_{(G,\chi)}(\onechip,T)$, then the rotor at $e$ may be untouched during rotor-routing, it may complete a full rotation, or it could be oriented in opposite directions in $T_s$ and $T'_s$ (note that by Proposition \ref{prop:nooverspins}, the rotor will never move more than a full rotation). When the edge $e$ is in neither tree, the travelling vertex may never have crossed $e$, it may have crossed $e$ in both directions, or it may have only crossed $e$ in one direction. For some of these cases, the path of the traveling vertex will change after deleting or contracting $e$, but we use Lemma~\ref{lem:rearrange} to show that it still crosses the same edges, just in a different order. The planarity condition comes into play for the final case (where the edge is crossed in one direction). We apply Lemma~\ref{lem:looprev2} to show that this case is never realized.
 
Throughout this proof, we let $\i(e) = \{x,y\}$ and $T' = r_{(G,\chi)}([c-s],T)$. Let $\mathcal R$ be the ordered list of directed edges crossed when rotor-routing on $(G,\chi)$ with input $c-s$ and $T$. 

We will consider $3$ cases to prove that the first condition of Definition~\ref{def:consistent} holds, and then $3$ cases to prove that the second condition holds. The theorem then follows from Proposition~\ref{prop:rrpart3}. For the first $3$ cases, assume that $e\in T \cap T'$. We further assume without loss of generality that $T_s\rotdir{x} = e$. Let $T^e = r_{(G/ {e},\chi/e)}(\onechip,T \setminus e)$ and let $ {\mathcal R_e}$ be the ordered list of directed edges crossed when rotor-routing on $(G/e,\chi/e)$ with input $c-s$ and $T\setminus e$. Finally let $z$ be the vertex formed by contracting $e$. Notice that $(T\setminus e)_s\rotdir{z} = T_s\rotdir{y}$ if $y \neq s$ and $z$ does not have a rotor in $(T\setminus e)$ when $y = s$. \\

\textbf{Case 1: }The chip never enters the vertex $x$ while rotor-routing.\\

For this case, we claim that $\mathcal R$ and ${\mathcal R_e}$ are identical. 

Assume for the sake of induction that this statement holds for the first $k$ edges. After crossing these rotors, the chip must be on the same vertex of both graphs (or the chip is on $y$ in $G$ and $z$ in $G/e$). By assumption, this vertex has been entered the same number of times for each graph, so the rotor here must be in the same position. It follows that the chip will also exit along the same rotor for both graphs.

This proves our claim that $\mathcal R = {\mathcal R_e}$. For every $v \in V(G) \setminus \{s,x,y\}$, it follows immediately from the claim that $T'_s\la v\ra = T^e_s\la v\ra$. Furthermore, by construction, we also have $T'_s\la y\ra = T^e_s\la z\ra$ and $T'_s\la x\ra = e$. This implies that $T^e = T' \setminus e$ as desired.  \\

\textbf{Case 2: }The rotor at $x$ spins around completely. In other words, $T_s\rotdir{x} = T'_s\rotdir{x}$ but $T_s\rotdir{x} \not= \rho\rotdir{x}$ for some intermediate rotor configuration $\rho$. \\

We begin with a brief, high-level explanation of the proof of this case. We start by showing that $x \neq c$ and $y\neq s$ to avoid any edge cases that may arise. We then carefully define an alternative ribbon structure $\wtilde{\chi/e}$ for the contracted graph $G/e$ for which the steps of the rotor-routing algorithm used to compute $\what{r}_{(G/e,\wtilde{\chi/e},s)}(c-s,T\setminus e)$ and $\what{r}_{(G,\chi,s)}(c-s,T)$ are nearly identical. It is then easily argued that $\what{r}_{(G/e,\wtilde{\chi/e},s)}(c-s,T\setminus e) = T'\setminus e$. We then apply Lemma \ref{lem:rearrange} to show that 
\[T^e = \what{r}_{(G/e,\chi/e,s)}(c-s,T\setminus e) = \what{r}_{(G/e,\wtilde{\chi/e},s)}(c-s,T\setminus e) = T'\setminus e.\]

Now, we give a complete proof. First, we note that by Proposition~\ref{prop:nooverspins}, it is impossible for $x$ to complete more than one full rotation. Let $\chi(x) = (e,e_1,\dots,e_p)$ and $\chi(y) = (e,\what e_1,\dots,\what e_l)$ after removing edges parallel to $e$, so that $(\chi/e)(z) = (e_1,\dots,e_p,\what e_1,\dots,\what e_l)$. Recall that we require $e \not= f$ for the first condition of Definition~\ref{def:consistent}. Suppose that $x=c$. By assumption, the rotor must cross $f$ from $c$ to $s$ at some point during rotor-routing. As soon as this happens, the algorithm terminates, and thus $T'_s\la x \ra = f$. This contradicts the assumption that $T'_s\la x \ra = e$, so we must have $x \not= c$.

Likewise $y \neq s$. Otherwise, since $x \neq c$, for the rotor at $x$ to make a full rotation, the chip must enter $x$ precisely $\deg(x)$ many times. By Proposition \ref{prop:nooverspins}, it must enter $x$ along every edge, including the edge $e$ from $y$ to $x$, before the rotor makes a full rotation. However, this is clearly a contradiction as the rotor-routing algorithm halts as soon as the chip reaches $y=s$.

Let $T_s \la y\ra = \what e_i$ and $T'_s\la y\ra = \what e_j$. By Proposition~\ref{prop:nooverspins}, each edge may only be crossed once in each direction. By assumption, the chip exits $x$ along each incident edge. Because $x \not= c$, this means that $x$ must be entered its degree many times. Thus, the chip must traverse every edge incident to $x$ in both directions. In particular, the rotor at $y$ must cross $e$, which implies that $i\ge j$. When rotor-routing on $(G,\chi)$, keep an ordered list of every edge other than $e$ that the chip crosses when moving out of either $x$ or $y$ and call this list $\mathcal E$. This list must consist precisely of the edges:
\[\{e_1,\dots,e_p\} \cup \{\what e_1,\dots,\what e_j\} \cup \{\what e_{i+1},\dots,\what e_l\}.\]
Furthermore, the last edge in $\mathcal E$ must be $\what e_j$. This is because after exiting $x$ for the final time, the rotor enters $y$ across $e$, and $T'_s\la y\ra = \what e_j$. 

We now define an alternate ribbon structure on $G/e$, which we will call $\wtilde{\chi/e}$. For each vertex $v \not= z$, let $(\wtilde{\chi/e})(v) = (\chi/e)(v)$. Let $(\wtilde{\chi/e})(z)$ begin with $\mathcal E$ and then be ordered arbitrarily.  Let $\mathcal R - {\overline e}$ be the sublist of $\mathcal R$ that is obtained after removing any directed edges between $x$ and $y$. Finally, let $\wtilde {\mathcal R_e}$ be the ordered list of directed edges crossed when rotor-routing on $(G/e,\wtilde{\chi/e})$ with input $c-s$ and $T\setminus e$. We claim that $\mathcal R - {\overline e}$ and $\wtilde {\mathcal R_e}$ are identical. 

Suppose that the claim holds and let $\wtilde T = \what r_{(G/e,\wtilde{\chi/e},s)}(c-s,T\setminus e)$. For all $v \in V(G) \setminus\{x,y,s\}$, we know that $T'_s\la v \ra = \wtilde T_s\la v \ra$. Furthermore, by construction, we also have $T'_s\la y \ra = \wtilde T_s\la z \ra$ and $T'_s\la x \ra= e$. The fact that $T'_s\la y \ra = \wtilde T_s\la z \ra$, along with our construction of $\wtilde {\chi/e}$, implies that $\wtilde{\chi/e}$ satisfies the requirements of Lemma~\ref{lem:rearrange} (when comparing to $\chi/e$ on $G/e$). Thus, we know that $\wtilde T = T^e$. The result follows. 

Now we just need to prove the claim that $\mathcal R - {\overline e} = \wtilde {\mathcal R_e}$. Assume for the sake of induction that the first $k$ rotors match. After crossing these rotors, the chip must be on the same vertex of both graphs (or the chip is on $z$ for $G/e$ and on either $x$ or $y$ for $G$). If the chip is on some $v \not=z$, then the argument is analogous to the argument from Case 1. Thus, it suffices to consider the case where the chip is on $z$. 

Out of the first $k$ entries of $\mathcal R - {\overline e}$, let $a$ be the number of edges that exit $x$ or $y$. Then, by assumption, for the first $k$ entries of $\wtilde {\mathcal R_e}$, there must be $a$ edges that exit $z$. By construction, the next edge crossed in either case must be the $(a+1)$th edge of $\mathcal E$. It follows by induction that $\mathcal R - {\overline e} = \wtilde {\mathcal R_e}$.\\

\textbf{Case 3: } The rotor $e$ is directed differently for $T_s$ and $T'_s$. In other words, $T'_s\la y \ra = e$.\\

This case is similar to the previous case, and we will use the same notation for $\chi(x)$ and $\chi(y)$. Suppose that $T_s \la y\ra = \what e_i$ and $T'_s\la x\ra  = e_j$. As before, let $\mathcal E$ be the ordered list of edges other than $e$ that the chip crosses when moving out of $x$ or $y$. This list must consist precisely of the edges:
\[\{\what e_{i+1},\dots,\what e_l\} \cup \{ e_1,\dots, e_j\}.\]

Furthermore, after exiting $y$ for the final time across $e$, the chip enters $x$. This means that the last edge in $\mathcal E$ must be $e_j$. Similarly to the previous case, we define a ribbon structure $\wtilde{\chi/e}$ on $G/e$. For $v \not= z$, let $(\wtilde{\chi/e})(v) = (\chi/e)(v)$, and let $(\wtilde{\chi/e})(z)$ begin with $\mathcal E$ and then be ordered arbitrarily.

The rest of the argument is analogous to the argument used for Case 2. In particular, we define $\mathcal R - {\overline e}$ and $\wtilde {\mathcal R_e}$ and then show that they are identical. We also use this idea to show that $\wtilde{\chi/e}$ satisfied the requirements of Lemma~\ref{lem:rearrange}. The only slight difference from the previous case is that the final positions of the rotors at $x$ and $y$ swap.\\

%The remainder of the argument is similar to the previous case. The only differences are that $\pi$ now only maps to $\{0,1\}$ and we now set 
%\[\theta(k) = \begin{cases}x & \text{if there is no $i>k$ such that $c_i \in \{x,y\}$ and $\rho_i\la c_i \ra \not= e$,}\\
%y & \text{otherwise.}\end{cases}\]
%With this new definition for $\theta$ and $\pi$, one obtains the same induction result as the one from Case 2. 

%It follows from our earlier reasoning that $m=n-1$ and, for all $v \in V(G)\setminus\{x,y\}$, we have $\rho_n\la v\ra = \rho^{e'}_m \la v \ra$ and $\rho_n\la y\ra = \rho^{e'}_m \la z \ra$. Thus, we have shown that $T^e = T'/e$.\\

For the last 3 cases, $e \not\in T\cup T'$. For these cases, we define ${}^eT = r_{(G\setminus e,\chi \setminus e)}([c-s],T)$. We need to show that ${}^eT = T'$. Also, let ${}_e\mathcal R$ be the ordered list of directed edges crossed when rotor-routing on $(G\setminus e,\chi \setminus e)$ with input $c-s$ and $T$. \\

%We also redefine $c^e_k$ and $\rho^e_k$ to refer to the steps of Algorithm~\ref{alg:rotor} on this new $T^e$. 

\textbf{Case 4: } The chip never crosses $e$.\\ 

Since the chip never crosses $e$, removing this edge has no effect on the output of the rotor-routing algorithm. We can use an analogous argument to what we used for Case 1. In particular, we show that $\mathcal R = {}_e\mathcal R$, and the result follows. \\ %In particular, it follows from a straightforward induction argument that for all $k \in [0,n]$, we have $c^e_k = c_k$ and $\rho^e_k = \rho_k$. This means that $\rho^e_n = \rho_n$ and $T^e = T'$.\\

\textbf{Case 5: } The edge $e$ is crossed in both directions, but $T'_s\la x \ra \not = e \not= T'_s\la y \ra$. \\

First, we observe that $\i(e) \not= \{c,s\}$: otherwise, it is impossible for the chip to cross $e$ in both directions. Thus, $c$ and $s$ must remain incident on $G\setminus e$. This allows us to apply Proposition~\ref{prop:nooverspins} to say that no edge is crossed more than once in each direction.

Define $\chi(x)$ and $\chi(y)$ as in Case 2. Without loss of generality, suppose that the first time the chip crosses $e$, it moves from $y$ to $x$. Let $\rho$ be the first rotor configuration reached while rotor routing on $(G,\chi)$ (with input $c-s$ and $T$) such that $\rho \la y \ra = e$.  Further, suppose that $\rho\la x \ra = e_q$. We introduce a new ribbon structure $\wtilde{\chi}$ on $G$. For $v \in V(G) \setminus x$, let $\wtilde{\chi}(v) = \chi(v)$, and let \[\wtilde{\chi}(x) = (e,e_{q+1},e_{q+2},\dots,e_p,e_1,e_2,\dots,e_q).\] 

Let $\rm E_1$ be the set of edges after $T_s \la x \ra$ and before $T'_s \la x \ra$ with respect to $\chi$  and $\rm E_2$ be the set of edges after $T_s \la x \ra$ and before $T'_s \la x \ra$ with respect to $\wtilde \chi$. Because we only rearranged the position of $e$, we must have $\rm E_1\setminus e=\rm E_2\setminus e$. Furthermore, $e \in \rm E_1$ by assumption, and $e \in \rm E_2$ by construction (because the rotor-routing algorithm is equivalent on $(G,\chi)$ and $(G,\wtilde \chi)$ until the chip crosses $e$ from $x$ to $y$). It follows that $\rm E_1 = \rm E_2$. 

By the reasoning in the previous paragraph, $\wtilde{\chi}$ satisfies the conditions for Lemma~\ref{lem:rearrange}, which means that $\what r_{(G,\wtilde{\chi},s)}(c-s,T) = T'$. 

Let $\wtilde {\mathcal R}$ be the ordered list of directed edges crossed when rotor-routing on $(G,\wtilde \chi)$ with input $c-s$ and $T$. Let $\wtilde {\mathcal R}-e$ be the remaining edges of $\wtilde {\mathcal R}$ after removing $e$ (but not edges parallel to $e$). We claim that $\wtilde {\mathcal R}-e = {}_e\mathcal R$. This follows from construction, because when the chip crosses $e$ from $y$ to $x$ in $\wtilde {\mathcal R}$, we know by construction that $e_q$ is the rotor at $x$. Therefore, the next step of the algorithm is to cross back along $e$ in the other direction. Thus, it is as though this edge did not exist. 

Given our claim, the result follows from the same reasoning we used for the previous cases. \\

\textbf{Case 6: } The edge $e$ is only crossed in one direction. \\

We will use Lemma~\ref{lem:looprev2} to show that this case is never realized. This is the one case where the planarity condition is vital. 

Without loss of generality, suppose that the chip crosses $e$ from $x$ to $y$. Then, by assumption, it must reach $x$ again without crossing $e$. Consider the next occurrence when the chip returns to $x$. Let the rotor configuration at this moment be $\rho$.

Consider the sequence $(x,y,v_1,v_2,\dots)$ of vertices where $\rho\la y\ra = v_1$ and each $v_i$ satisfies $\i(\rho \la v_i\ra ) = \{v_i,v_{i+1}\}$. Since there are a finite number of vertices, and the chip has just returned to $x$, there must be some $j$ such that $x \in \i(\rho \la v_j \ra)$. Thus, the rotors incident to these vertices form a directed cycle, which does not include the sink vertex. By Lemma~\ref{lem:looprev2}, this cycle must reverse before the rotor-routing algorithm terminates. In order for the cycle to reverse, the chip must cross $e$ in the other direction. This gives a contradiction.\end{proof}

We end this section with an example developed by T\'othm\'er\'esz, which demonstrates one of the subtleties of Definition~\ref{def:consistent}. In particular, Theorem~\ref{thm:rrconsistent} fails if we remove the requirement that $c$ and $s$ are adjacent.

\begin{example}{\cite[Remark 20]{Lilla2}}\label{ex:lilla} 
Figure~\ref{fig:lilla} gives a pair of vertices $c$ and $s$, and edge $e$, and a spanning tree $T$, such that $e \in T \cap r_{(G,\chi)}([c-s],T)$, but $r_{(G,\chi)}([c-s],T) \not= r_{(G/e,\chi/e)}([c-s],T\setminus e)$. The reason that this example does not contradict Theorem~\ref{thm:rrconsistent} is because $c$ and $s$ are not connected by an edge. In particular, Proposition~\ref{prop:nooverspins} is no longer satisfied. 
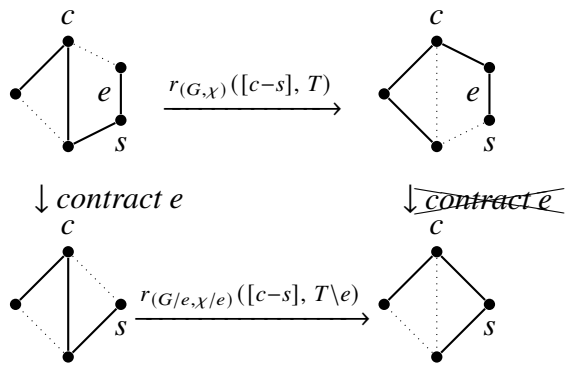
\begin{figure}
\begin{center}
\begin{tikzpicture}[scale = 0.7]
    
    \tikzstyle{pt} = [circle,fill,inner sep=1pt,minimum size = 1.5mm]
    \tikzstyle{coin} = [draw,circle,inner sep=1pt,minimum size = 2mm]
    
    \begin{scope}[shift={(0,0)}]
    \node[pt] (a1) at (0,0) {};
    \node[pt,label={north:{$c$}}](b1) at (1,1){} ;
    \node[pt](c1) at (2,.5){};
    \node[pt,label={south:{$s$}}](d1) at (2,-.5){};
    \node[pt](e1) at (1,-1){};
    
    \node[] at (4.5,0) {$\xrightarrow{r_{(G,\chi)}([c-s],~T)}$};
    \node[] at (0.5,-2) {$\downarrow$};
    \node[] at (2,-2) {contract $e$};
    \node(o) at (1.7,0){$e$};
    \end{scope} 
    
    \begin{scope}[shift={(7,0)}]
    \node[pt] (a2) at (0,0) {};
    \node[pt,label={north:{$c$}}](b2) at (1,1){} ;
    \node[pt](c2) at (2,.5){};
    \node[pt,label={south:{$s$}}](d2) at (2,-.5){};
    \node[pt](e2) at (1,-1){};
    \node[] at (0.5,-2) {$\downarrow$};
    \node[] at (2,-2) {\xcancel{~~~~~contract $e$~}};
    \node(o) at (1.7,0){$e$};
    \end{scope} 
    
    \begin{scope}[shift={(0,-4)}]
    \node[pt] (a3) at (0,0) {};
    \node[pt,label={north:{$c$}}](b3) at (1,1){} ;
    \node[pt,label={south:{$s$}}](cd3) at (2,0){};
    \node[pt](e3) at (1,-1){};
    
    \node[] at (4.5,0) {$\xrightarrow{r_{(G/ e,\chi/e)}([c-s],~T\setminus e)}$};
    \end{scope}
    
    \begin{scope}[shift={(7,-4)}]
    \node[pt] (a4) at (0,0) {};
    \node[pt,label={north:{$c$}}](b4) at (1,1){} ;
    \node[pt,label={south:{$s$}}](cd4) at (2,0){};
    \node[pt](e4) at (1,-1){};
    \end{scope}

%%%%%%%%%%%%%%%%%%%%%%%%%%%%%%%%%%%%%%%%%%%%%%%%%%%%%%%%%%%%%5
    \tikzstyle{every node} = [draw = none,fill = none,scale = .8]

    \draw[thick] (a1) -- (b1);
    \draw[dotted] (b1) -- (c1);
    \draw[thick] (c1) -- (d1);
    \draw[thick](d1) -- (e1);
    \draw[dotted](e1) -- (a1);
    \draw[thick] (b1) -- (e1);
    
    \draw[thick] (a2) -- (b2);
    \draw[thick] (b2) -- (c2);
    \draw[thick] (c2) -- (d2);
    \draw[dotted] (d2) -- (e2);
    \draw[thick](e2) -- (a2);
    \draw[dotted] (b2) -- (e2);
    
    \draw [thick](a3) -- (b3);
    \draw[dotted] (b3) -- (cd3);
    \draw[thick] (cd3) -- (e3);
    \draw[dotted](e3) -- (a3);
    \draw[thick] (b3) -- (e3);
    
    \draw[thick] (a4) -- (b4);
    \draw[thick] (b4) -- (cd4);
    \draw[thick] (cd4) -- (e4);
    \draw[dotted](e4) -- (a4);
    \draw[dotted] (b4) -- (e4);

\end{tikzpicture}
	\end{center}
	\caption{This figure shows the importance of the source and sink vertices being adjacent in the definition of consistency. As usual, the ribbon graph is oriented counterclockwise.}
	\label{fig:lilla}
\end{figure}
\end{example} 

T\'othm\'er\'esz's note (which was written after the first version of our paper) also includes an alternative proof of Theorem~\ref{thm:rrconsistent} (see~\cite[Proposition~17]{Lilla2}).

\section{Uniqueness of Consistent Torsor Algorithms}\label{sec:unique}
In this section, we classify the consistent sandpile torsor algorithms. In particular, we show that every consistent sandpile torsor algorithm must either be equivalent to rotor-routing, or one of three related algorithms which we say have the same \emph{structure} as rotor-routing. 

For any ribbon graph $(G,\chi)$, let $\overline \chi$ be the reverse cyclic order around each vertex. Notice that if $(G,\chi)$ is a plane graph, then $(G,\overline \chi)$ is a plane graph as well (simply reflect the planar embedding of $(G,\chi)$ to get a planar embedding of $(G,\overline{\chi})$).

\begin{definition}
Suppose $\alpha$ is a sandpile torsor algorithm. Define $\overline{\alpha}$, $\alpha^{-1}$, and $\overline{\alpha}^{-1}$ such that for any plane graph $(G,\chi)$, $S \in \Pic^0(G)$, and $T \in \mathcal T(G)$, we have:
\[\alpha_{(G,\chi)}(S,T) = \overline{\alpha}_{(G,\overline \chi)}(S,T) = \alpha^{-1}_{(G,\chi)}(-S,T) = \overline{\alpha}^{-1}_{(G,\overline \chi)}(-S,T).\]
\end{definition}

If $\alpha$ is the rotor-routing algorithm, then $\overline{\alpha}$ reverses the direction in which the rotors turn, $\alpha^{-1}$ switches the role of the chip and sink, and $\overline{\alpha}^{-1}$ makes both of these changes. 

\begin{prop}\label{prop:otherts}
If $\alpha$ is a (consistent) sandpile torsor algorithm, then $\overline{\alpha}$, $\alpha^{-1}$ and $\overline{\alpha}^{-1}$ are distinct (consistent) sandpile torsor algorithms.
\end{prop}
\begin{proof}
It is straightforward to show that the defining properties of sandpile torsor actions, as well as consistency, are preserved if we replace $\alpha$ with any of the other 3 possibilities. 

Let $E_3$ be the triple edge with a planar ribbon structure. Reversing the ribbon structure on $E_3$ is a ribbon graph automorphism that is the identity on $V(E_3)$ and $\Pic^0(E_3)$, but is not the identity on $E(E_3)$ or $\T(E_3)$. It follows that $\alpha_{(E_3)} \not= \overline \alpha_{(E_3)}$ and $\alpha^{-1}_{(E_3)} \not= \overline \alpha^{-1}_{(E_3)}$. 

Let $C_3$ be the circuit with 3 edges and its unique ribbon structure. Then $\Pic^0(C_3) \cong \Z/3\Z$, which means that there are elements of $\Pic^0(G)$ that are not equal to their own inverse. It follows that $\alpha_{(C_3)} \not= \alpha^{-1}_{(C_3)}$ and $\overline \alpha_{(C_3)} \not= \overline \alpha^{-1}_{(C_3)}$. 

By combining the arguments from the previous two paragraphs, we know that for any two of these algorithms, there exist plane graphs for which the sandpile torsor actions are distinct. 
\end{proof}

\begin{definition}
Sandpile torsor algorithms $\alpha$ and $\beta$ have the same \emph{sandpile torsor structure} if $\beta \in \{\alpha,\overline{\alpha},\alpha^{-1},\overline{\alpha}^{-1}\}$. 
\end{definition}

 Our goal will be to prove the following version of Conjecture~\ref{conj:klivans}. 

\begin{theorem}\label{thm:onestructure}
Every consistent sandpile torsor algorithm on plane graphs has the same sandpile torsor structure as the rotor-routing algorithm. 
\end{theorem}

We provide the proof of Theorem~\ref{thm:onestructure} at the end of the section. The first step of the proof is to note that we can restrict our attention to \emph{2-connected} plane graphs (i.e. plane graphs with no cut vertices). In particular, we show the following.

\begin{lemma}\label{lem:2congood}
If $\alpha$ and $\beta$ are consistent sandpile torsor algorithms such that $\alpha_{(G,\chi)} = \beta_{(G,\chi)}$ for all 2-connected plane graphs $(G,\chi)$, then $\alpha = \beta$. 
\end{lemma}
\begin{proof}
Suppose for the sake of contradiction that there is some plane graph $(G,\chi)$ such that $\alpha_{(G,\chi)} \not= \beta_{(G,\chi)}$. By the contrapositive of Corollary~\ref{cor:adjverts}, this means that for some adjacent vertices $c$ and $s$, and some $T\in\T(G)$, we have
\[\alpha_{(G,\chi)}([c-s],T) \not= \beta_{(G,\chi)}([c-s],T).\]

We can assume $(G,\chi)$ is not 2-connected, or else the contradiction is immediate. Let $f$ be an edge such that $\i(f) = \{c,s\}$. By Condition 3 of consistency (Definition~\ref{def:consistent}), any edges separated from $f$ after removing cut vertices must be unchanged when acting on $T$ by $[c-s]$. For all such edges, we contract the edges that are in $T$ and delete the edges that are not. We are left with a 2-connected plane graph $(G',\chi')$ and a spanning tree $T' \in \T(G')$. By Conditions 1 and 2 of consistency, the inequality is preserved. In particular, 
\[\alpha_{(G',\chi')}([c-s],T') \not= \beta_{(G',\chi')}([c-s],T').\]
This is a contradiction. 
\end{proof}

In general, the rotor-routing algorithm can have a complicated global effect on a spanning tree. Our next goal towards a proof of Theorem~\ref{thm:onestructure} will involve defining a distinguished set of (sandpile group element, spanning tree) pairs called \emph{source-turn pairs} (as well as a slightly less restrictive set of \emph{single-step pairs}). These pairs are useful because they have the following two properties:
\begin{itemize}
    \item The rotor-routing action on source-turn pairs and single-step pairs has a simple geometric characterization.
    \item To show that sandpile torsor algorithm is equivalent to the rotor-routing algorithm, it suffices to show that the two sandpile torsor algorithms agree on source-turn pairs (see Corollary~\ref{cor:sourceenough}). 
\end{itemize}

We first define source-turn pairs and single-step pairs. Given a graph $G$ and a spanning tree $T\in \T(G)$, a \emph{leaf vertex} is a vertex incident to only one edge of $T$, while a \emph{leaf edge} is an edge of $T$ incident to a leaf vertex. 

\begin{definition}\label{def:singstep}
Let $(G,\chi)$ be a 2-connected ribbon graph, $T \in \T(G)$, and $c,s\in V(G)$. 
\begin{itemize}
    \item The pair $(c-s,T)$ is called a \emph{single-step pair} at $c$ if the rotor-routing algorithm terminates after a single rotor turns one position. We call the map from $T$ to $r_{(G,\chi)}([c-s],T)$ a \emph{single-step move} at $c$. 
    \item The pair $(c-s,T)$ is called a \emph{source-turn pair} at $c$ if $c$ is a leaf vertex of $T$ and the equality $\{c,s\} = \i(\chi(c,T_s\la c \ra))$ is satisfied. We call the map from $T$ to $r_{(G,\chi)}([c-s],T)$ a \emph{source-turn move} at $c$. 
\end{itemize}
\end{definition}

\begin{lemma}\label{lem:singswap}
If $(c-s,T)$ is a source-turn pair, then $(c-s,T)$ is also a single-step pair. 
\end{lemma}
\begin{proof}
If $c$ is a leaf vertex, then the unique edge of $T$ incident to $c$ must be equal to $T_s\la c \ra$. The result follows from the definition of rotor-routing. 
\end{proof}

Recall from Definition \ref{def:onestep} that for $x \in V(G) \setminus s$, the function $r_x$ rotates the rotor at $x$ one position. It is immediate from definition that when $(c-s,T)$ is a single-step pair, we have $r_{(G,\chi)}([c-s],T_s) = r_c(T_s)$. Furthermore, if $T_s\la c \ra = g$ and $r_c(T_s)\la c \ra = f$, then $r_{(G,\chi)}([c-s],T) = T \setminus g \cup f$. We say that $(c-s,T)$ is a \emph{single-step pair from $g$ to $f$} (where a \emph{source-turn pair from $g$ to $f$} is defined analogously). See Figure~\ref{fig:swaps} for some examples or single-step pairs and source-turn pairs. 

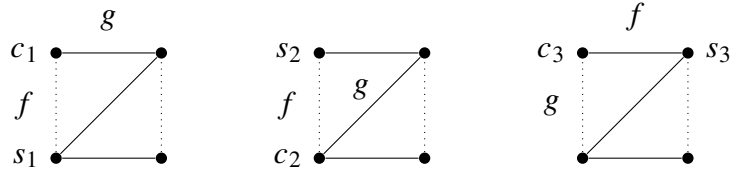
\begin{figure}[h]
\begin{center}
\begin{tikzpicture}[scale = 0.7]
    
    \tikzstyle{pt} = [circle,fill,inner sep=1pt,minimum size = 1.5mm]
    \tikzstyle{coin} = [draw,circle,inner sep=1pt,minimum size = 2mm]
    
    \begin{scope}[shift={(0,0)}]
    \node[pt,label={west:{$s_1$}}] (a1) at (0,0) {};
    \node[pt](b1) at (2,0){} ;
    \node[pt,label={west:{$c_1$}}](c1) at (0,2){};
    \node[pt](d1) at (2,2){};
    \node[label = {west:{$f$}}](f1) at (0,1){};
    \node[label = {north:{$g$}}](g1) at (1,2){};
    \end{scope} 
    
    \begin{scope}[shift={(5,0)}]
    \node[pt,label={west:{$c_2$}}] (a2) at (0,0) {};
    \node[pt](b2) at (2,0){} ;
    \node[pt,label={west:{$s_2$}}](c2) at (0,2){};
    \node[pt](d2) at (2,2){};
    \node[label = {west:{$f$}}](f2) at (0,1){};
    \node[label = {north:{$g$}}](g2) at (.8,.7){};
    \end{scope} 
    
    \begin{scope}[shift={(10,0)}]
    \node[pt] (a3) at (0,0) {};
    \node[pt](b3) at (2,0){} ;
    \node[pt,label={west:{$c_3$}}](c3) at (0,2){};
    \node[pt,label={east:{$s_3$}}](d3) at (2,2){};
    \node[label = {west:{$g$}}](g3) at (0,1){};
    \node[label = {north:{$f$}}](f3) at (1,2){};
    \end{scope}

%%%%%%%%%%%%%%%%%%%%%%%%%%%%%%%%%%%%%%%%%%%%%%%%%%%%%%%%%%%%%5
    \tikzstyle{every node} = [draw = none,fill = none,scale = .8]

    \draw[dotted] (a1) -- (c1);
    \draw[dotted] (b1) -- (d1);
    \draw         (c1) -- (d1) -- (a1) -- (b1);
    
    \draw[dotted] (a2) -- (c2);
    \draw[dotted] (b2) -- (d2);
    \draw         (c2) -- (d2) -- (a2) -- (b2);
    
    \draw[dotted] (a3) -- (c3);
    \draw[dotted] (b3) -- (d3);
    \draw         (c3) -- (d3) -- (a3) -- (b3);

\end{tikzpicture}
\caption{Let $T$ be the spanning tree depicted above (where the graph is given a counterclockwise ribbon structure). The pair $(c_1-s_1,T)$ is a source-turn pair and a single-step pair from $g$ to $f$. The pair $(c_2-s_2,T)$ is a single-step pair from $g$ to $f$, but not a source-turn pair. The pair $(c_3-s_3,T)$ is neither a source-turn pair nor a single-step pair. However, $(s_3-c_3,T)$ is a reverse single-step pair from $f$ to $g$ (see Definition~\ref{def:revstep}).}
\label{fig:swaps}
\end{center}
\end{figure}

We now introduce the key result which we use to prove Corollary \ref{cor:sourceenough}:

\begin{theorem}\label{thm:sourceturn}
%Let $(G,\chi)$ be a 2-connected ribbon graph. For any $T^{start},T^{goal} \in \T(G)$, there are sequences $(c_1,s_1,c_2,s_2,\dots,c_k,s_k)\in V(G)^{2k}$ and $(T^0 = T^{start},T^1,\dots,T^k = T^{goal}) \in \T(G)^{k+1}$ such that for every $i \in [1,k]$, we have
%\[T^i = r_{(G,\chi)}([c_i - s_i], T^{i-1}),\]
%and $(c_i-s_i,T^{i-1})$ is a source-turn pair. 
Let $(G,\chi)$ be a 2-connected ribbon graph. For any $T^{start},T^{goal} \in \T(G)$, there exists a sequence of source-turn moves whose composition maps $T^{start}$ to $T^{goal}$. 
\end{theorem}

Before proving the theorem, we give a high level outline of our argument. 
\begin{description}
\item[Step 1] Fix an arbitrary \emph{root vertex} $\rt \in V(G)$ and transform each $T \in \T(G)$ into rotor configurations $T_{\rt}$.
\item[Step 2] Observe that source-turn moves on trees correspond to \emph{source-rotations} in the corresponding rotor configurations (see Definition~\ref{def:rotatable}). 
\item[Step 3] Introduce a partial ordering on $\T(G)$ such that $T^{goal}$ is the unique minimal spanning tree.
\item[Step 4] Use an inductive argument to show that for any $T \neq T^{goal}$, it is possible to use source-rotations to transform $T_{\rt}$ into some $T_{\rt}'$ such that $T'$ is less than $T$ in the partial order on $\T(G)$.
\item[Step 5] Because $T^{goal}$ is the unique minimal spanning tree, the theorem follows. 
\end{description}    

%This is done via an inductive assumption that the theorem holds for all minors of $(G,\chi)$, in which case it suffices to understand how source-turn moves interact with contractions and deletions (see Lemmas \ref{lem:delconswap} and \ref{lem:contractsource} and Corollary \ref{cor:contracttreetotree}). 

In order to make this argument precise, it is convenient to introduce a few definitions and straightforward results. In particular, the inductive argument in Step 4 requires an understanding of the conditions under which source-turn moves applied to minors of $(G,\chi)$ correspond to source-turn moves in $(G,\chi)$. We establish these conditions in Lemmas \ref{lem:delconswap} and \ref{lem:contracttreetotree}, so it is convenient to delay the proof of Theorem~\ref{thm:sourceturn} until immediately after the proof of Lemma \ref{lem:contracttreetotree}.

\begin{definition}\label{def:vertorder}
Given a graph $G$, a spanning tree $T \in \T(G)$, and a vertex $\rt \in V(G)$, we write $x \prec_{T_\rt} y$ if the path from $x$ to $\rt$ along $T$ passes through $y$. 
\end{definition}

It is straightforward to check that $\prec_{T_\rt}$ always induces a partial ordering on $V(G)\setminus \rt$ and that a vertex is minimal if and only if it is a leaf vertex of $T$ (other than $\rt$). 
\begin{lemma}\label{lem:whichrotor}
Let $(G,\chi)$ be a 2-connected ribbon graph and $T \in \T(G)$. For $e \in T$ with $\i(e)=\{x,y\}$, $x \prec_{T_\rt} y$ if and only if $T_\rt\la x\ra = e$. 
\end{lemma}
\begin{proof}
This lemma follows immediately from Definitions~\ref{def:Ts} and~\ref{def:vertorder}. 
\end{proof}

\begin{lemma}\label{lem:singcriteria}
Let $(G,\chi)$ be a 2-connected ribbon graph, $T \in \T(G)$, $g \in T$, $f \in E(G) \setminus T$, $\i(f) = \{c,s\}$, $\i(g) = \{c,x\}$, and $\chi(c,g) = f$. The pair $(c-s,T)$ is a single-step pair from $g$ to $f$ if and only if $c \prec_{T_s} x$. 
\end{lemma}
\begin{proof}
For the if direction, Lemma~\ref{lem:whichrotor} says that $T_s \la x \ra = g$. Thus, Algorithm~\ref{alg:rotor} terminates after a single step and the output tree is $T \setminus g \cup f$. For the only if direction, $c \not\prec_{T_s} x$ implies that $T_s \la x \ra \not= g$. If the rotor-routing algorithm terminates after a single step, the resulting tree must include the edge $g$. Thus, it is not a single-step pair from $g$ to $f$. 
\end{proof}

Lemma \ref{lem:singcriteria} suggests that we can characterize single-step moves and source-turn moves in terms of rotors constructed from the underlying spanning tree $T$. We now make this characterization explicit in Definition \ref{def:rotatable} and show the correspondence between the two representations of single-step (or source-turn) moves in Lemma \ref{lem:rotatable}.

\begin{definition}\label{def:rotatable}
Let $(G,\chi)$ be a 2-connected ribbon graph with $T \in \T(G)$ and $c,\rt \in V(G)$. 
\begin{itemize}
    \item If $r_c(T_\rt)$ is acyclic, then we say that $T_\rt$ is \emph{rotatable at $c$}. We call the map from $T_\rt$ to $r_c(T_\rt)$ a \emph{rotation at $c$}. 
    \item For $g,f \in E(G)$, we specify that $T_\rt$ is \emph{rotatable at $c$ from $g$ to $f$} if $T_\rt$ is rotatable at $c$, $g = T_\rt\la c \ra$, and $f = \chi(c,g)$.
    \item The terms \emph{source-rotatable at $c$ (from $g$ to $f$)} and \emph{source-rotation} are defined analogously with the added requirement that $c$ is a leaf vertex.
\end{itemize}
\end{definition}

\begin{lemma}\label{lem:rotatable}
Let $T \in \T(G)$ and $\rt \in V(G)$. The configuration $T_\rt$ is rotatable (resp. source-rotatable) at $c$ from $g$ to $f$ if and only if there exists a single-step pair (resp. source-turn pair) $(c-s,T)$ from $g$ to $f$ and $T_\rt\la c \ra = g$. 
\end{lemma}
\begin{proof}
Suppose that $T_\rt$ is rotatable at $c$ from $g$ to $f$. By Definition~\ref{def:rotatable}, this implies that $g = T_\rt\la c \ra$, $f = \chi(c,g)$, and $r_c(T_\rt)$ is acyclic. Let $s = \i(f) \setminus c$ and  $x = \i(g) \setminus c$. Let $P$ be the unique path from $c$ to $s$ in $T$. If $g \not\in P$, then, $T \setminus g \cup f$ must contain a cycle, which contradicts the condition that $r_c(T_\rt)$ is acyclic. Thus, $g \in P$ and $c \prec_{T_{s}} x$. By Lemma~\ref{lem:singcriteria}, $(c-s,T)$ is a single-step pair from $g$ to $f$. 

Alternatively, suppose that there exists an $s \in V(G)$ such that $(c-s,T)$ is a single-step pair from $g$ to $f$ and $T_\rt \la c \ra = g$. A single step through the rotor-routing algorithm produces the rotor configuration $r_c(T_{s})$. By Lemma~\ref{lem:singcriteria}, we know that $c \prec_{T_s} x$, which implies $T_s \la c \ra = g$ by Lemma~\ref{lem:whichrotor}. Because $T_\rt \la c \ra = T_{s} \la c \ra$, it follows that $r_c(T_\rt)\la c \ra = r_c(T_s)\la c \ra = f$. This implies that $r_c(T_\rt) = (T \setminus g \cup f)_\rt$, which is acyclic by Lemma~\ref{lem:acyclic}.  

The proof is analogous when working with source-rotatability and source-turn pairs noting that the $c$ will remain a source vertex after rotating the rotor at $c$ from $g$ to $f$.
\end{proof}

For the induction step of Theorem~\ref{thm:sourceturn}, we need to describe how (source-)rotatability is impacted by taking deletions and contractions.  

\begin{lemma}\label{lem:delconswap}
Let $(G,\chi)$ be a 2-connected ribbon graph with $c,\rt \in V(G)$. Suppose that $T \in \T(G)$ and $T_{\rt}$ is (source-)rotatable at $c$ from $g$ to $f$ on $(G,\chi)$.
\begin{itemize}
\item For any $e \in E(G) \setminus (T\cup f)$, the configuration $T_{\rt}$ is (source-)rotatable at $c$ from $g$ to $f$ on $(G\setminus e,\chi\setminus e)$.  
\item For any $e \in T \setminus g$, the configuration $(T\setminus e)_{\rt}$ is (source-)rotatable at $c$ from $g$ to $f$ on $(G/e,\chi/e)$. 
\end{itemize}
\end{lemma}
\begin{proof}
Both of these results follow directly from Definition~\ref{def:rotatable}. All of the defining properties of (source) rotatability still hold after contracting an edge in $T$ or deleting an edge not in $T$ (as long as this edge is not $f$ or $g$). 
\end{proof}

\begin{lemma}
\label{lem:contracttreetotree}
Let $(G,\chi)$ be a 2-connected plane graph with two spanning trees $T,T' \in \T(G)$ and some $\rt \in V(G)$. Fix a connected $F \subset T\cap T'$ with at least one edge incident to $\rt$. Let $(\widetilde{G},\widetilde{\chi})$ be the ribbon graph obtained by contracting the edges in $F$, and $\widetilde{\rt}$ be the vertex these edges contract into. 

If there exists a sequence of source-rotations that take $(T\setminus F)_{\widetilde{\rt}}$ to $(T'\setminus F)_{\widetilde{\rt}}$ on $(\widetilde{G},\widetilde{\chi})$, then there exists a sequence of source-rotations that take $T_\rt$ to $T'_\rt$ on $(G,\chi)$ without turning any rotors corresponding to edges in $F$.
\end{lemma}
\begin{proof}
Because $F$ has an edge incident to $\rt$, each vertex of $(\widetilde{G},\widetilde{\chi})$ other than $\widetilde{\rt}$ has a single preimage in $(G,\chi)$. Furthermore, for any $\widetilde T \in \T(\wtilde G)$, and any $c \not= \wtilde{\rt} \in V(\widetilde G)$, it is immediate that $\widetilde T \cup F\in \T(G)$ and $\widetilde T_{\rt'}\la c \ra = (\widetilde T \cup F)_{\rt}\la c \ra$. 

Fix some tree $T'' \in \T(G)$ and suppose there exists a vertex $c \in V(\wtilde{G})\setminus \wtilde{\rt}$ and edges 
$f,g \in E(\wtilde{G})$ such that $(T''\setminus F)_{\wtilde{\rt}}$ is source-rotatable at $c$ from $g$ to $f$ in $(\wtilde{G},\wtilde{\chi})$. Then by Definition \ref{def:rotatable}, $g = (T''\setminus F)_{\wtilde{\rt}}\la c\ra = T''_{\rt}\la c\ra$, and $f = \wtilde{\chi}(c,g) = \chi(c,g)$. Lastly, 
\[r_c((T''\setminus F)_{\wtilde{\rt}}) = ((T''\setminus F)\cup f\setminus g)_{\wtilde{\rt}} = (T''\cup f\setminus g)_{\rt}\setminus F = r_c(T''_{\rt})\setminus F.\]
Thus, $r_c(T''_{\rt}) = r_c((T''\setminus F)_{\wtilde{\rt}})\cup F$. Given that $r_c((T''\setminus F)_{\wtilde{\rt}})$ is acyclic (by Definition \ref{def:rotatable}) and that $F$ contracts to a single vertex, $r_c(T''_{\rt})$ can only contain a cycle if $F$ contains a cycle. However, $F \subseteq T$, where $T$ is a spanning tree, so this is impossible. Therefore $r_c(T''_{\rt})$ is acyclic. Then by Definition \ref{def:rotatable}, $T''_{\rt}$ is rotatable at $c$ from $g$ to $f$ in $(G,\chi)$. Lastly, $T''_{\rt}$ is source-rotatable at $c$ because if $c\neq \wtilde{\rt}$ is a leaf of $T''\setminus F$, then it is a leaf in $T''$.

Given that for any $\wtilde{T} \in \T(\wtilde{G})$, $\wtilde{\rt}$ has no rotors in $\wtilde{T}_{\wtilde{\rt}}$, this implies that any sequence of source-rotations that take $(T\setminus F)_{\widetilde{\rt}}$ to $(T'\setminus F)_{\widetilde{\rt}}$ on $(\widetilde{G},\widetilde{\chi})$ must also take $T_\rt$ to $T'_\rt$ on $(G,\chi)$. Furthermore, the rotors corresponding to edges in $F$ are not present in $ (\widetilde{G},\widetilde{\chi})$, so they are not rotated by any of the source rotations taking  $T_\rt$ to $T'_\rt$ on $(G,\chi)$.
\end{proof}

Now, we are ready to prove Theorem~\ref{thm:sourceturn}. 
\begin{proof}[Proof of Theorem \ref{thm:sourceturn}]
First, choose an arbitrary \emph{root vertex} $\rt \in V(G)$ which we will fix throughout this proof. Our goal will be to show that there is a series of source-rotations whose composition maps from $T^{start}_{\rt}$ to $T^{goal}_{\rt}$. The theorem then follows from repeated applications of Lemma~\ref{lem:rotatable}.

Next, we set up an induction argument. When $G$ consists of a single edge, the result is trivial. Consider the following inductive assumption. \\

%Note that for any two trees $T'$ and $T''$ where $T''$ is the result of a source-turn move at $T'$ with $c\neq\rt$, $T''_{\rt}$ is obtained from $T'_{\rt}$ by rotating the rotor at $c$ by one position. Therefore, we make the following (slightly stronger) inductive assumption that implies the theorem:\\

\noindent \textbf{Inductive Assumption:} Suppose $(G',\chi')$ is a 2-connected proper minor of $(G,\chi)$. For any pair $\what{T}^{start}, \what{T}^{goal} \in \T(G')$, and any vertex $\rt' \in V(G')$, there is a series of source-rotations whose composition maps from $\what{T}^{start}_{\rt'}$ to $\what{T}^{goal}_{\rt'}$.\\

Let $\xi$ be the map from $\mathcal T(G) \times (V(G)\setminus \rt) \to \{0,1\}$ defined as follows:

\[ \xi(T,x) = \begin{cases} 0 & \text{if }T_\rt \la x \ra = T^{goal}_\rt \la x \ra, \\
1 & \text{otherwise.}\end{cases}\]

Notice that $\xi(T,x) = 0$ for all $x$ if and only if $T = T^{goal}$. We can use $\xi$ and the vertex partial order from Definition~\ref{def:vertorder} to define a partial order on spanning trees, which we will write as $\prec_{\xi}$ (or $\preceq_{\xi}$ for the corresponding weak partial order).

\begin{align*}T \preceq_{\xi} T' \Longleftrightarrow \text{ For all } &x \in (V(G)\setminus \rt) \text{ such that }\xi(T,x) > \xi(T',x) \text{, there exists }\\ & y \in (V(G)\setminus \rt)\text{ such that }\xi(T,y) < \xi(T',y) \text{ and } x \prec_{T^{goal}_\rt} y.\end{align*}

We now show that $\prec_{\xi}$ is a partial order. Irreflexivity and asymmetry are immediate, we just need to show transitivity. Let $T$, $T'$, and $T''$ be spanning trees such that $T \prec_\xi T'$ and $T' \prec_\xi T''$. Let $x$ be an arbitrary vertex such that $\xi(T,x) > \xi(T'',x)$ (if no such vertex exists, then $T \prec_\xi T''$ is immediate). 

First, suppose that $\xi(T',y) = \xi(T,y)$ for all $y$ where $x \prec_{T^{goal}_\rt} y$. Then, $\xi(T',x) \geq \xi(T,x) > \xi(T'',x)$. Because $T' \prec_\xi T''$, there is some $z$ such that $x \prec_{T^{goal}_\rt} z$ and $\xi(T',z) < \xi(T'',z)$. Furthermore, because $\xi(T',z) = \xi(T,z)$, we also have $\xi(T,z) < \xi(T'',z)$

Otherwise, there must be some $y$ such that $x \prec_{T^{goal}_\rt} y$ and $\xi(T,y) < \xi(T',y)$. We can take $y$ to be a maximal such vertex (with respect to $\prec_{T^{goal}_\rt}$). If $\xi(T',y) \le \xi(T'',y)$, then $\xi(T,y) < \xi(T'',y)$. Otherwise, because $T' \prec_\xi T''$, there is some $z$ such that $y \prec_{T^{goal}_\rt} z$ and $\xi(T',z) < \xi(T'',z)$. By maximality of $y$ and the definition of $\prec_\xi$, we know that $\xi(T,z) \le \xi(T',z)$. It follows that $\xi(T,z) < \xi(T'',z)$. Since $x$ was arbitrary, $x \prec_{T_\rt^{goal}} y$, and $x \prec_{T_\rt^{goal}} z$, we know that $T \prec_\xi T''$. \\

The partial ordering $\prec_\xi$ is used to describe deviations from $T^{goal}_\rt$, prioritizing rotors closest to the root. Because $|\T(G)|$ is finite and $T^{goal}$ is the unique minimal spanning tree, it suffices to show that given any $T \in \T(G)\setminus T^{goal}$, there is a sequence of source-rotations from $T_\rt$ to some $T_\rt''$ such that $T'' \prec_{\xi} T$.

Suppose that there exists some $x \in V(G)$ such that $\xi(T,x) = 1$ and $x$ is a leaf vertex of $T$. Then, we can freely rotate the rotor at $x$ using source-rotations until we get a configuration $T_\rt''$ such that $\xi(T'',x) = 0$. Because we only rotated the rotor at $x$, it follows that $T'' \prec_\xi T$ and we are done.

Otherwise, there must exist some $x \in V(G)\setminus \rt$ such that $\xi(T,x)=1$ and for all $v \prec_{T_\rt} x$, we have $\xi(T,v) = 0$. Let $g = T_\rt\la x \ra$, $f = \chi(x,g)$, and $\{x,y\} = \i(f)$. Consider the sets
\[V^x = \{v \mid v \prec_{T_\rt} x\} \text { and } E^x = \{T_\rt\la v \ra \mid v \in V^x\}.\]
By assumption, all of the rotors associated with vertices in $V^x$ are already in their final position and are directed toward $x$. This implies that
\[V^x \subseteq \{v \mid v \prec_{T^{goal}_\rt} x\}.\]

We now introduce the following claim:\\

\noindent\textbf{Claim: } There exists some $T'\in \T(G)$ such that $x$ is a leaf of $T'$ and $T_\rt'$ can be reached from $T_\rt$ after a sequence of source-rotations at vertices in $V^x$. \\

If the claim is true, then we can rotate the rotor at $x$ until it reaches $T^{goal}\la x \ra$ and call the new tree $T''$. By construction, $\xi(T'',x) < \xi(T',x) = \xi(T,x)$. Furthermore, since all source-rotations were at vertices in $V_x\subseteq \{v \mid v \prec_{T^{goal}_\rt} x\}$, it follows that $\xi(T'',v) = \xi(T,v)$ for all $v \not\prec_{T^{goal}_\rt} x$. Thus, $T'' \prec_\xi T$. 

All that is left is to prove the claim. By Lemma~\ref{lem:whichrotor}, the set $T \setminus E^x$ has only one edge incident to $x$. Because $(G,\chi)$ is 2-connected, $x$ is not a cut vertex. Thus, there exists a spanning tree $T' \in \T(G)$ such that $(T \setminus E^x) \subseteq T'$ and $x$ is a leaf edge of $T'$. Let $F$ be the set of edges in $T\setminus E^x$. Notice that $F$ consists of all of the edges of $T$ that are on the same component of $T$ as $\rt$ when the tree is severed at the point $x$. In particular, $F$ is connected and contains at least one edge incident to $\rt$.

If we contract all of the edges in $F$, we obtain a new ribbon graph $(\widetilde G,\widetilde \chi)$ such that $V(\widetilde G) = V^x\cup x$. Furthermore, $T \cap E(\widetilde{G})$ and $T' \cap E(\widetilde{G})$ are both in $\T(\widetilde G)$. Let $\wtilde \rt$ be the vertex formed by contraction. The vertices other than $\wtilde \rt$ cannot be cut vertices of $(\widetilde G,\widetilde \chi)$ because they aren't cut vertices of $(G,\chi)$. If $\wtilde \rt$ is a cut vertex, we can consider each 2-connected component separately. In either case, we can apply the induction hypothesis. 

By the induction hypothesis, there is a sequence of source-rotations whose composition maps from $T \cap E(\widetilde{G})$ to $T' \cap E(\widetilde{G})$ on $(\widetilde G,\widetilde \chi)$. Thus, by Lemma~\ref{lem:contracttreetotree}, there is a sequence of source-rotations that take $T_\rt$ to $T'_\rt$ on $(G,\chi)$ without turning any rotors corresponding to edges in $F$. The claim follows from the fact that the edges in $T \setminus F$ correspond to the rotors at vertices in $V^x$. 

%Furthermore, the vertices of $V^x$ cannot be incident to any edge of $T \setminus E^x$ on $G$ by the condition that they are all smaller than $x$ with respect to $\prec_{T_\rt}$. This means that all of the edges in $T\setminus E(\widetilde{G})$ are contracted into $x$. Additionally, the inductive assumption implies that none of these source-turn moves involve the vertex $x \in V(\widetilde{G})$. Thus, by Corollary \ref{cor:contracttreetotree}, there exists a sequence of source-turn moves from $T$ to $T'$ on $(G,\chi)$ that is restricted to the rotors at $V_x$. This proves the claim. 

\end{proof}

\begin{figure}
\begin{center}
\begin{tikzpicture}
    
    \tikzstyle{every node} = [circle,fill,inner sep=1pt,minimum size = 1.5mm]
    \node[label={west:{\large$\rt$}}] (a) at (0,0) {};
    \node(b) at (1,0){} ;
    \node(c) at (2,0){};
    \node(d) at (3,0){};
    
    \node(a2) at (0,1) {};
    \node[label={south east:{$x$}}](b2) at (1,1){} ;
    \node(c2) at (2,1){};
    \node(d2) at (3,1){};
    
    \node(a3) at (0,2) {};
    \node(b3) at (1,2){} ;
    \node(c3) at (2,2){};
    \node(d3) at (3,2){};

    \begin{scope}[shift={(5,0)}]
    \node[label={west:{\large$\rt$}}] (2a) at (0,0) {};
    \node(2b) at (1,0){} ;
    \node(2c) at (2,0){};
    \node(2d) at (3,0){};
    
    \node(2a2) at (0,1) {};
    \node(2b2) at (1,1){} ;
    \node(2c2) at (2,1){};
    \node(2d2) at (3,1){};
    
    \node(2a3) at (0,2) {};
    \node(2b3) at (1,2){} ;
    \node(2c3) at (2,2){};
    \node(2d3) at (3,2){};
    \end{scope}
    
    \begin{scope}[shift={(10,0)}]
    \node[label={west:{\large$\rt$}}] (3a) at (0,0) {};
    \node(3b) at (1,0){} ;
    \node(3c) at (2,0){};
    \node(3d) at (3,0){};
    
    \node(3a2) at (0,1) {};
    \node(3b2) at (1,1){} ;
    \node(3c2) at (2,1){};
    \node(3d2) at (3,1){};
    
    \node(3a3) at (0,2) {};
    \node(3b3) at (1,2){} ;
    \node(3c3) at (2,2){};
    \node(3d3) at (3,2){};
    \end{scope}
    
    %\begin{scope}[shift={(0,-3)}]
    %\node[label={west:{\large$\rt$}}] (4a) at (0,0) {};
    %\node(4b) at (1,0){} ;
    %\node(4c) at (2,0){};
    %\node(4d) at (3,0){};
    
    %\node(4a2) at (0,1) {};
    %\node(4b2) at (1,1){} ;
    %\node(4c2) at (2,1){};
    %\node(4d2) at (3,1){};
    
    %\node(4a3) at (0,2) {};
    %\node(4b3) at (1,2){} ;
    %\node(4c3) at (2,2){};
    %\node(4d3) at (3,2){};
    %\end{scope}
    
    %\begin{scope}[shift={(5,-3)}]
    %\node[label={west:{\large$\rt$}}] (5a) at (0,0) {};
    %\node(5b) at (1,0){} ;
    %\node(5c) at (2,0){};
    %\node(5d) at (3,0){};
    
    %\node(5a2) at (0,1) {};
    %\node(5b2) at (1,1){} ;
    %\node(5c2) at (2,1){};
    %\node(5d2) at (3,1){};
    
    %\node(5a3) at (0,2) {};
    %\node(5b3) at (1,2){} ;
    %\node(5c3) at (2,2){};
    %\node(5d3) at (3,2){};
    %\end{scope}
    
    %\begin{scope}[shift={(10,-3)}]
    %\node[label={west:{\large$\rt$}}] (6a) at (0,0) {};
    %\node(6b) at (1,0){} ;
    %\node(6c) at (2,0){};
    %\node(6d) at (3,0){};
    
    %\node(6a2) at (0,1) {};
    %\node(6b2) at (1,1){} ;
    %\node(6c2) at (2,1){};
    %\node(6d2) at (3,1){};
    
    %\node(6a3) at (0,2) {};
    %\node(6b3) at (1,2){} ;
    %\node(6c3) at (2,2){};
    %\node(6d3) at (3,2){};
    %\end{scope}

%%%%%%%%%%%%%%%%%%%%%%%%%%%%%%%%%%%%%%%%%%%%%%%%%%%%%%%%%%%%%5
    \tikzstyle{every node} = [draw = none,fill = none,scale = .8]
    
    \node(o) at (1.5,2.5){\Large$T_\rt$};
    \node(o) at (6.5,2.5){\Large$T_\rt'$};
    \node(o) at (11.5,2.5){\Large$T_\rt^{goal}$};

    \draw[dotted](a) --(d)--(d3) -- (a3) -- (a);
    \draw[dotted](b) --(b3);
    \draw[dotted](c) --(c3);
    \draw[dotted](a2) --(d2);
    
    \draw[dotted](2a) --(2d)--(2d3) -- (2a3) -- (2a);
    \draw[dotted](2b) --(2b3);
    \draw[dotted](2c) --(2c3);
    \draw[dotted](2a2) --(2d2);
    
    \draw[dotted](3a) --(3d)--(3d3) -- (3a3) -- (3a);
    \draw[dotted](3b) --(3b3);
    \draw[dotted](3c) --(3c3);
    \draw[dotted](3a2) --(3d2);

    %\draw[dotted](4a) --(4d)--(4d3) -- (4a3) -- (4a);
    %\draw[dotted](4b) --(4b3);
    %\draw[dotted](4c) --(4c3);
    %\draw[dotted](4a2) --(4d2);
    
    %\draw[dotted](5a) --(5d)--(5d3) -- (5a3) -- (5a);
    %\draw[dotted](5b) --(5b3);
    %\draw[dotted](5c) --(5c3);
    %\draw[dotted](5a2) --(5d2);
    
    %\draw[dotted](6a) --(6d)--(6d3) -- (6a3) -- (6a);
    %\draw[dotted](6b) --(6b3);
    %\draw[dotted](6c) --(6c3);
    %\draw[dotted](6a2) --(6d2); 
    
    \draw [-{Latex[length=2mm,width=1.5mm]}](a2) --(a);
    \draw [-{Latex[length=2mm,width=1.5mm]}](b2) --(a2);
    \draw [-{Latex[length=2mm,width=1.5mm]}](b3) --(b2);
    \draw [-{Latex[length=2mm,width=1.5mm]}](a3) --(b3);
    \draw [-{Latex[length=2mm,width=1.5mm]}](c3) --(b3);
    \draw [-{Latex[length=2mm,width=1.5mm]}](b) --(a);
    \draw [-{Latex[length=2mm,width=1.5mm]}](c) --(b);
    \draw [-{Latex[length=2mm,width=1.5mm]}](d) --(c);
    \draw [-{Latex[length=2mm,width=1.5mm]}](c2) --(c);
    \draw [-{Latex[length=2mm,width=1.5mm]}](d2) --(c2);
    \draw [-{Latex[length=2mm,width=1.5mm]}](d3) --(d2);
    
    \draw [-{Latex[length=2mm,width=1.5mm]}](2a2) --(2a);
    \draw [-{Latex[length=2mm,width=1.5mm]}](2b2) --(2b);
    \draw [-{Latex[length=2mm,width=1.5mm]}](2b3) --(2a3);
    \draw [-{Latex[length=2mm,width=1.5mm]}](2a3) --(2a2);
    \draw [-{Latex[length=2mm,width=1.5mm]}](2c3) --(2d3);
    \draw [-{Latex[length=2mm,width=1.5mm]}](2b) --(2a);
    \draw [-{Latex[length=2mm,width=1.5mm]}](2c) --(2b);
    \draw [-{Latex[length=2mm,width=1.5mm]}](2d) --(2c);
    \draw [-{Latex[length=2mm,width=1.5mm]}](2c2) --(2c);
    \draw [-{Latex[length=2mm,width=1.5mm]}](2d2) --(2c2);
    \draw [-{Latex[length=2mm,width=1.5mm]}](2d3) --(2d2);
    
    \draw [-{Latex[length=2mm,width=1.5mm]}](3a2) --(3a);
    \draw [-{Latex[length=2mm,width=1.5mm]}](3b2) --(3b);
    \draw [-{Latex[length=2mm,width=1.5mm]}](3b3) --(3b2);
    \draw [-{Latex[length=2mm,width=1.5mm]}](3a3) --(3b3);
    \draw [-{Latex[length=2mm,width=1.5mm]}](3c3) --(3b3);
    \draw [-{Latex[length=2mm,width=1.5mm]}](3b) --(3a);
    \draw [-{Latex[length=2mm,width=1.5mm]}](3c) --(3b);
    \draw [-{Latex[length=2mm,width=1.5mm]}](3d) --(3c);
    \draw [-{Latex[length=2mm,width=1.5mm]}](3c2) --(3c);
    \draw [-{Latex[length=2mm,width=1.5mm]}](3d2) --(3d);
    \draw [-{Latex[length=2mm,width=1.5mm]}](3d3) --(3d2);
    
\end{tikzpicture}
\caption{In this example, with $\rt$ as our root vertex, there are no source-turn moves that immediately bring $T$ closer to $T^{goal}$. However, if we choose a vertex $x$ such that every $v \prec_{T_\rt} x$ is in the correct position, then we can turn these rotors until $x$ is the source of a source-turn pair, and then turn the rotor at $x$. This gives us the spanning tree $T'$. Notice that $T' \prec_\xi T$ as desired.  }
\label{fig:sourcelemma}
\end{center}
\end{figure}
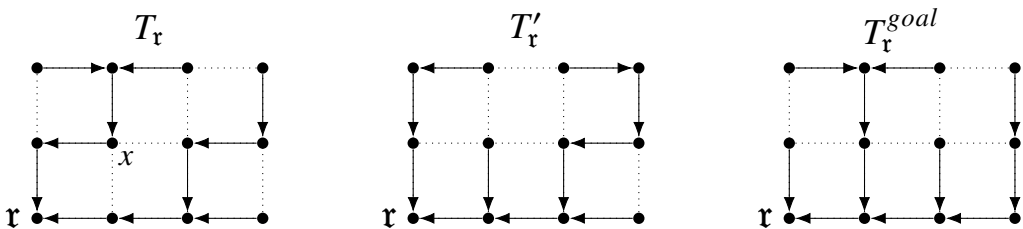

Notice that it is possible to rotate a rotor at a leaf vertex to any position through a series of source-turn moves. This means that the ribbon structure is not actually relevant for Theorem~\ref{thm:sourceturn}. In particular, we get the following corollary, which may be of independent interest. 

\begin{corollary}\label{cor:leafswap}
Let $G$ be a 2-connected graph and $T^{start}, T^{goal}\in \T(G)$. For some $k\in \Z_{\ge 0}$, there exists a sequence  $(T^0 = T^{start},T^1,\dots,T^k = T^{goal}) \in \T(G)^{k+1}$ such that for every $i \in [1,k]$, we have $|T^{i-1}\setminus T^{i}| = |T^i \setminus T^{i-1}| = 1$ and $T^{i-1}\setminus T^{i}$ is a leaf edge of $T^{i-1}$. 
\end{corollary}

\begin{remark}
Corollary~\ref{cor:leafswap} is a generalization of a well-known matroidal result. If we removed the requirement that $T^{i-1}\setminus T^{i}$ is a leaf edge, then the result would follow from the fact that the \emph{fundamental graph} of any \emph{connected} matroid is connected~\cite[Proposition~4.3.2]{Oxley}, and the \emph{cycle matroid} of a 2-connected graph is connected ~\cite[Proposition~4.1.7]{Oxley}.
\end{remark}
\begin{corollary}\label{cor:sourceenough}
Let $\alpha$ be a consistent sandpile torsor algorithm and $\rho$ be the rotor-routing algorithm. The following are equivalent:
\begin{enumerate}
\item $\alpha = r$.
\item For all 2-connected plane graphs $(G,\chi)$ and all \emph{single-step pairs} $(c-s,T)$, we have \[\alpha_{(G,\chi)}([c-s],T) = r_{(G,\chi)}([c-s],T).\]
\item For all 2-connected plane graphs $(G,\chi)$ and all \emph{source-turn pairs} $(c-s,T)$, we have \[\alpha_{(G,\chi)}([c-s],T) = r_{(G,\chi)}([c-s],T).\]
\end{enumerate}
\end{corollary}
\begin{proof}
$(1) \to (2)$ is immediate. $(2) \to (3)$ follows from Lemma~\ref{lem:singswap}. $(3) \to (1)$ follows from Theorem~\ref{thm:sourceturn} and Lemma~\ref{lem:2congood}. 
\end{proof}

Corollary~\ref{cor:sourceenough} greatly simplifies the task of proving that a sandpile torsor algorithm is equivalent to rotor-routing and will be essential in our proof of Theorem~\ref{thm:onestructure}. Before proving this theorem, we give one more definition that will aid our proof. 

\begin{definition}\label{def:revstep}
We say that $(s-c,T)$ is a \emph{reverse single-step pair from $f$ to $g$} if $f \in T$, $g \not\in T$, and $(c-s,T\setminus f \cup g)$ is a single-step pair from $g$ to $f$. 
\end{definition}

\begin{lemma}\label{lem:revstep}
If $(s-c,T)$ is a reverse single-step pair from $f$ to $g$, then $r_{(G,\chi)}([s-c],T) = T\setminus f \cup g$. 
\end{lemma}
\begin{proof}
\begin{gather*}r_{(G,\chi)}([s-c],T) = r_{(G,\chi)}\left([s-c], r_{(G,\chi)}([c-s],T\setminus f \cup g)\right) \\
= r_{(G,\chi)}([0],T\setminus f \cup g)) = T\setminus f \cup g.\qedhere\end{gather*}
\end{proof}

We are now ready to prove the majority of Theorem~\ref{thm:onestructure}. However, the final case of the theorem is particularly involved, so we relegate its proof to Appendix~\ref{app:case4}. 

\begin{proof}[Proof of Theorem~\ref{thm:onestructure}]

Let $\alpha$ be an arbitrary consistent sandpile torsor algorithm. We will prove the theorem by induction, but we first consider a few special plane graphs. 

Let $C_3$ be the 3 cycle along with its unique ribbon structure, and let $E_3$ be the triple edge with a planar ribbon structure. Each of these plane graphs has two sandpile torsor actions. On $C_3$, one of these actions is equivalent to both $r_{C_3}$ and $\overline r_{C_3}$, while the other is equivalent to both $r^{-1}_{C_3}$ and $\overline r^{-1}_{C_3}$. On $E_3$, one of the actions is equivalent to both $r_{E_3}$ and $\overline r^{-1}_{E_3}$, while the other is equivalent to both $\overline r_{E_3}$ and $r^{-1}_{E_3}$. This means that $\alpha_{C_3}$ and $\alpha_{E_3}$ will simultaneously match precisely one of the four sandpile torsor algorithms with the same structure as rotor-routing. For this proof, we assume that $\alpha_{C_3} = r_{C_3}$ and $\alpha_{E_3} = r_{E_3}$. By Proposition~\ref{prop:otherts}, we can make this assumption without loss of generality. 

Let $(G,\chi)$ be an arbitrary 2-connected plane graph. For the induction hypothesis, assume that $\alpha_{(G',\chi')} = r_{(G',\chi')}$ for all proper minors $(G',\chi')$ of $(G,\chi)$. By Corollary~\ref{cor:sourceenough}, it suffices to show that $\alpha_{(G,\chi)}([c-s],T) = r_{(G,\chi)}([c-s],T)$ for every source-turn pair $(c-s,T)$. Throughout this proof, $(c-s,T)$ is an arbitrary source-turn pair from $g$ to $f$ and 
\[\widehat T := \alpha_{(G,\chi)}([c-s],T).\]
To prove the theorem, we need to show that $\widehat T = r_{(G,\chi)}([c-s],T) = T \setminus g\cup f $. 

First, suppose that there exists some $e\not= g$ such that $e \in T\cap \widehat T$. This implies that $\i(e) \not= \{c,s\}$ because $T_s\la c \ra = g$ (if $\i(e) = \{c,s\}$ and $e \in T$, then $T_s\la c \ra = e \neq g$ which presents a contradiction). By the condition that $\alpha$ is consistent, we must have $\alpha_{(G/e,\chi/e)}([c-s],T \setminus e) = \widehat T \setminus e$. By the induction hypothesis, $\alpha_{(G/e,\chi/e)}([c-s],T \setminus e) = r_{(G/e,\chi/e)}([c-s],T \setminus e)$. Furthermore, it follows from Lemmas~\ref{lem:delconswap} and \ref{lem:rotatable} that $(c-s,T)$ is a source-turn pair for $(G/e,\chi/e)$. This means that $r_{(G/e,\chi/e)}([c-s],T) = (T \setminus g \cup f)\setminus e$. It follows that $\widehat T = T \setminus g \cup f = r_{(G,\chi)}([c-s],T)$. Similarly, suppose that there exists some $e \not= f$ such that $e \not\in T \cup \widehat T$. Then, by an analogous argument, we have $\alpha_{(G\setminus e,\chi\setminus e)}([c-s],T) = \widehat T$ and $\widehat T =T \setminus g \cup f= r_{(G,\chi)}([c-s],T)$. In both cases, the theorem holds by induction. Therefore we may assume that the only shared edges or non-edges of $T$ and $\what T$ are in $\{f,g\}$. In particular, we have the relation
\begin{equation}
\label{eq:nonsharededges}
E(G)\setminus (T\Delta \what T) \subseteq \{f,g\},
\end{equation}
where $\Delta$ denotes the symmetric difference operator. Note that this argument still holds when $(c-s,T)$ is a single-step pair or a reverse single-step pair. This leaves us with 4 cases to consider corresponding to the 4 subsets of $\{f,g\}$.\\

\textbf{Case 1:} We have $f \in \widehat T$ and $g \in \widehat T$.\\

In this case, the number of edges of $G$ must be one less than double the size of each spanning tree. 

Let $\i(g) = \{x,c\}$. Because $g \in T \cap \alpha_{(G,\chi)}([c-s],T)$, we can apply Definition \ref{def:consistent} and the induction hypothesis to show that
\[\alpha_{(G,\chi)}([c-s],T)\setminus g = \alpha_{(G/g,\chi/g)}([c-s],T\setminus g) = r_{(G/g,\chi/g)}([c-s],T\setminus g).\] 

Because $(c-s,T)$ is a source-turn pair (and thus also a single-step pair) from $g$ to $f$, we must have $T_s\la c \ra = g$. By Definition~\ref{def:vertorder} and Lemma~\ref{lem:whichrotor}, this implies that the path of edges $(e_1,e_2,\dots,e_k)$ on $T$ from $x$ to $s$ does not contain $g$ or any edges parallel to $g$. Furthermore, these edges must also form a path from $z$ to $s$ on $G/g$ (where $z$ is the vertex formed by contracting $g$). By Lemma~\ref{lem:looprev}, any edges to the right of the directed cycle formed by $(e_1,e_2,\dots,e_k,f)$ are in $r_{(G/g,\chi/g)}([z-s],T\setminus g) = \what T\setminus g$ if and only if they are in $T \setminus g$. This implies that for any such edge $e$, we have $e \notin \{f,g\}$ and $e \notin T\Delta \what{T}$, which contradicts \eqref{eq:nonsharededges}. Thus, there can be no edges to the right of the cycle. In particular, this means that $(\chi/g)(z,f) = e_1$ and $\chi(c,f) = g$. 

By Definition \ref{def:singstep}, we know that $\chi(c,g) = f$.  Together, the facts that $\chi(c,g) = f$ and $\chi(c,f) = g$ imply that $c$ is a vertex of degree 2. This means that a single firing of vertex $c$ sends $2c -s - x$ to $0$. Thus, $[2c - s - x] = [0]$  and $[c-s] = [x-c]$.

Let $T' = T\setminus g \cup f$. It follows from Definition \ref{def:singstep} that $(c-x,T')$ is a source-turn pair from $f$ to $g$. Suppose that $\alpha_{(G,\chi)}([c-x],T') = T$. Then, this would mean that $\alpha_{(G,\chi)}([x-c],T) = T'$. However, we showed that \[\alpha_{(G,\chi)}([x-c],T) = \alpha_{(G,\chi)}([c-s],T) = \widehat T\neq T'.\]

Because $\alpha_{(G,\chi)}([c-x],T') \not= T$, this action must swap every edge of $T'$ other than $f$ and $g$. Furthermore, since the number of edges in $G$ is one less than double the size of each spanning tree, both $f$ and $g$ must be in $\alpha_{(G,\chi)}([c-x],T')$. It follows that $\alpha_{(G,\chi)}([c-x],T') = \widehat T$. 

Notice that the edges $(g,e_k,e_{k-1},\dots,e_1)$ form a directed cycle on the graph $G/f$. Using our earlier argument, we can conclude that there are no edges to the right of this cycle. Furthermore, the rotors of this cycle are  oriented in the opposite direction as they were on $G/g$. Together with the fact that $c$ is a degree 2 vertex, this implies that $\{e_1,e_2,\dots,e_k,f,g\}$ forms a cycle in $G$ with no edges on either side. In particular, this means that $G$ is a cycle. 

It follows that $\what T = \{f,g\}$. The only cycle with a 2 edge spanning tree is the 3-cycle $C_3$. Thus, $\alpha_{(G,\chi)} = r_{(G,\chi)}$ by our assumption that $\alpha_{C_3} = r_{C_3}$.\\

\textbf{Case 2:} We have $f \notin \widehat T$ and $g \in \widehat T$.\\

Let $\i(g) = \{x,c\}$. By the same argument that we used in the previous proof, $c$ has degree $2$ and $[c-s] = [x-c]$. This means that 

\[\alpha_{(G,\chi)}([x-c],T) = \alpha_{(G,\chi)}([c-s],T) = \widehat T.\]

Because $f\not\in T \cup \widehat T$, we can use consistency to conclude that $\alpha_{(G\setminus f,\chi\setminus f)}([x-c],T) = \widehat T$. However, $g$ is a cut edge of $G \setminus f$. Thus, $[x-c] = 0$ on $G\setminus f$. This is a contradiction by \eqref{eq:nonsharededges}, since $T \not= \widehat T$ unless $E(G) = \{f,g\}$. If $E(G) = \{f,g\}$, we still have a contradiction because $[c-s]$ is a nontrivial element of $\Pic^0(G)$, so it cannot fix any spanning trees. Thus, Case 2 is impossible.\\

\textbf{Case 3:} We have $f \notin \widehat T$ and $g \notin \widehat T$.\\

This case is a bit more complicated than the previous two, so we begin with a brief overview of our argument. 
\begin{description}
    \item[Step 1] First, we show that for any source turn pair $(c''-s'',T'')$, we either have $\alpha_{(G,\chi)}([c'' - s''],T'') = r_{(G,\chi)}([c'' - s''],T'')$ or $\alpha_{(G,\chi)}([2(c'' - s'')],T'') = r_{(G,\chi)}([c'' - s''],T'')$ (see~\eqref{eq:case3}).
    \item[Step 2] This observation allows us to define a sequence of sandpile elements such that the action by $\alpha_{(G,\chi)}$ on $T''$ rotates the rotor at $c''$ all the way around until returning to $T''$. In particular, these sandpile elements must add to the identity (see~\eqref{eq:case3identity}). 
    \item[Step 3] Using properties of the sandpile group, we prove that all the sandpile elements defined in the previous step have a particular form.
    \item[Step 4] Next, we use the fact that $\alpha_{(G,\chi)}$ must always output a spanning tree in order to deduce that there must be exactly $3$ edges incident to $c$.
    \item[Step 5] After this, we show that when Case 3 occurs and $c$ is incident to exactly $3$ edges, these must be the only edges of $G$. In particular, it follows that $(G,\chi) = E_3$ (a planar embedding of the triple edge). 
    \item[Step 6] Finally, we know that $r_{E_3} = \alpha_{E_3}$ by a base case of our induction. 
\end{description} 

Let $T' = T \setminus g$. By assumption, $\alpha_{(G,\chi)}([c-s],T) \not= T'\cup f$. Thus, $\alpha_{(G,\chi)}([s-c],T'\cup f) \not= T$. However, note that $(s-c,T'\cup f)$ is a reverse single-step pair. Because \eqref{eq:nonsharededges} holds for reverse single-step pairs, every edge other than $f$ and $g$ must swap. In addition, because the total number of edges in each spanning tree is consistent, neither $f$ nor $g$ can be in $\alpha_{(G,\chi)}([s-c],T'\cup f)$. This means that $\alpha_{(G,\chi)}([s-c],T'\cup f) = \what{T}$ and $\alpha_{(G,\chi)}([c-s],\what{T}) = T'\cup f$. We obtain the following chain of equalities:
\[\alpha_{(G,\chi)}([2c-2s],T) = \alpha_{(G,\chi)}([c-s],\what T) = T' \cup f = r_{(G,\chi)}([c-s], T).\] 

Note that the reasoning in the previous paragraph holds for any source-turn pair. In particular, if $(c''-s'',T'')$ is a source-turn pair, then \begin{equation}\label{eq:case3} \alpha_{(G,\chi)}([a(c'' - s'')],T'') = r_{(G,\chi)}([c'' - s''],T'') \text{ for some $a \in \{1,2\}$}.\end{equation}

Let $\chi(c) = (e_1,e_2,\dots,e_k)$ where $e_1 = g$, and set $\i(e_i) = \{c,s_i\}$. Note that $e_2 = f$, $s_2 = s$, and the $s_i$ are allowed to coincide. Because $(c-s,T)$ is a source-turn pair, it follows that $(c-s_{i+1},T'\cup e_i)$ is also a source-turn pair for any $i\in[1,k]$. Thus, by \eqref{eq:case3}, there exists a function $\phi: [1,k] \to \{1,2\}$ such that for all $i \in [1,k]$, 
\[\alpha_{(G,\chi)}([\phi(i)(c-s_{i+1})],T' \cup e_i) = r_{(G,\chi)}([c-s_{i+1}],T' \cup e_i),\]
where $i+1$ is replaced with $1$ for $i = k$. It follows that

\begin{equation}\label{eq:case3identity}\alpha_{(G,\chi)}\left(\left[\sum_{i=1}^k\phi(i)(c-s_i)\right],T\right) = r_{(G,\chi)}\left(\left[\sum_{i=1}^k(c-s_i)\right],T\right) = r_{(G,\chi)}([0],T) = T,\end{equation}
where we use the fact that

\[\left[\sum_{i=1}^k(c-s_i)\right] = [0], \]
because after firing $c$, we return to the identity. Since $\alpha$ is free, this implies that

\[\left[\sum_{i=1}^k\phi(i)(c-s_i)\right] = \left[\sum_{i=1}^k(\phi(i)-1)(c-s_i)\right]= [0]. \]

The divisor $\sum_{i=1}^k(\phi(i)-1)(c-s_i)$ has $-1$ chips on each $s_i$ for which $\phi(i) = 2$, chips on $c$ equal to the number of such $i$, and no chips elsewhere. By~\cite[Lemma 17]{alextorsors}, and the assumption that $c$ is not a cut vertex, this divisor is only equivalent to the identity if $\phi(i) = 2$ for all $i$ or $\phi(i) = 1$ for all $i$. By the assumption that we are in Case 3, $\phi(1) \not= 1$, which means that $\phi(i) = 2$ for all $i \in [1,k]$.

Next, we consider the possible values of $k$ (the number of edges incident to $c$). Recall that $\widehat T = \alpha_{(G,\chi)}([c-s],T)$. Because $(c-s,T)$ is a source-turn pair, $e_j\not\in T$ for $j > 1$. Furthermore, since we are in Case 3, this implies that $e_j \in \widehat T$ if and only if $j>2$. If $k \le 2$, then $\widehat T$ has no edges connecting to $c$, which is impossible. Since $e_1 \not\in \what T$, there must be a path from $s_1$ to some other $s_i\neq s_1$ along $\widehat T$ that does not pass through the vertex $c$. If $k>4$, then by pigeonhole principle, there must be a pair $s_j$ and $s_{j+1}$ such that neither vertex is $s_1$ or $s_i$. 

We know that $\phi(j) = 2$, which means that $\alpha_{(G,\chi)}([c-s_{j+1}],T' \cup e_j) \not= T' \cup e_{j+1}$. The only other possibility is that
\[ \alpha_{(G,\chi)}([c-s_{j+1}],T' \cup e_j) = \widehat T \cup \{e_1,e_2\} \setminus \{e_j, e_{j+1}\}. \]

However, $\widehat T \cup \{e_1,e_2\} \setminus \{e_j, e_{j+1}\}$ contains $e_1$, $e_i$, and a second path from $s_1$ to $s_i$. This means that it must contain a cycle, which is a contradiction.

When $k=4$, the situation is similar. The argument that we used for $k>4$ still works if we can show that there is a path along $\widehat T$ from $s_1$ to $s_2$ or $s_4$ that does not pass through $c$. In particular, we let $j = 3$ in the first case and $j = 2$ in the second. We will show that such a path always exists (where throughout this paragraph, all of our paths are along $\widehat T$). First, we note that if every path from $s_1$ to $s_4$ goes through $c$, then there must be a path $P_1$ from $s_1$ to $s_3$ that does not pass through $c$. If there is a path from $s_3$ to $s_2$ that does not pass through $c$, then by combining that path with $P_1$, we get a path from $s_1$ to $s_2$ that does not pass through $c$ as desired. Otherwise, there must be a path $P_2$ from $s_2$ to $s_4$ that does not pass through $c$. Furthermore, note that $s_4$ and $s_2$ are on opposite sides of the cycle $P_1 \cup \{e_1,e_3\}$. By planarity, this implies that $P_1$ and $P_2$ must intersect at some point $v$. If we combine the path from $s_1$ to $v$ along $P_1$ and the path from $v$ to $s_2$ along $P_2$, we get a path from $s_1$ to $s_2$ that does not pass through $c$.

We have now shown that we must have $k=3$ for this case to be possible. Recall that $\widehat T := \alpha_{(G,\chi)}([c-s_2],T) = (E(G) \setminus T)\setminus e_2$. Because $k=3$, $\widehat T$ contains exactly one edge incident to $c$, namely $e_3$. In particular, $([c-s_1],\widehat T)$ is a source-turn pair. Let $\widetilde T = \alpha_{(G,\chi)}([c-s_1],\widehat T)$. By \eqref{eq:nonsharededges}, there are only 2 possibilities for $\widetilde T$: either $\widetilde T = \widehat T \setminus e_3 \cup e_1$ or $\widetilde T = T' \cup e_2$. 

In either case, we have the following:
\[\alpha_{(G,\chi)}([c-s_3],\widetilde T) = \alpha_{(G,\chi)}([(c-s_3) + (c -s_1) + (c - s_2)],T) = \alpha_{(G,\chi)}([0],T) = T.\]

If $\widetilde T = \widehat T \setminus e_3 \cup e_1$, then $e_2 \not\in \widetilde T \cup T$. Thus, by consistency, we have
\[\alpha_{(G\setminus e_2,\chi\setminus e_2)}([c-s_3],\widetilde T) = T.\] By our induction hypothesis, 
\[\alpha_{(G\setminus e_2,\chi\setminus e_2)}([c-s_3],\widetilde T) = r_{(G\setminus e_2,\chi\setminus e_2)}([c-s_3],\widetilde T) = \widetilde T \setminus e_1 \cup e_3.\]

This is a contradiction because $e_3 \not\in T$. It follows that we must have $\widetilde T = T' \cup e_2$. Notice that $(c-s_3,T' \cup e_2)$ is a source-turn pair, and $r_{(G,\chi)}([c-s_3],T' \cup e_2) = T' \cup e_3 \not= T$. By \eqref{eq:nonsharededges}, this means that $\alpha_{(G,\chi)}([c-s_3],T' \cup e_2)$ cannot share any edges with $T' \cup e_2$, and the only shared non-edge is $e_3$. However, we showed above that 
\[\alpha_{(G,\chi)}([c-s_3],T' \cup e_2) = \alpha_{(G,\chi)}([c-s_3],\widetilde T) = T = T' \cup e_1.\]

It follows that $E(G) = \{e_1,e_2,e_3\}$ and, by the 2-connectedness condition, $(G,\chi)$ must be $E_3$, a planar embedding of the triple edge. Thus, $\alpha_{(G,\chi)} = r_{(G,\chi)}$ by our assumption that $\alpha_{E_3} = r_{E_3}$.\\

\textbf{Case 4:} We have $f \in \widehat T$ and $g \notin \widehat T$.\\

Most of the work for this case is given in Appendix~\ref{app:case4}, and there is an overview of our argument just before Lemma~\ref{lem:singstepnocase2}. If $(G,\chi)$ is a \emph{telescope graph} (see Definition~\ref{def:telescope}), then $\alpha_{(G,\chi)}([c-s],T) = r_{(G,\chi)}([c-s],T)$ by Lemma~\ref{lem:scopegood}. If $(G,\chi)$ is not a telescope graph, then $\alpha_{(G,\chi)}([c-s],T) = r_{(G,\chi)}([c-s],T)$ by Lemma~\ref{lem:noscopegood}.
\end{proof}

\begin{definition}
A \emph{consistent sandpile torsor action} on a plane graph $(G,\chi)$ is a sandpile torsor action obtained from a consistent sandpile torsor algorithm.
\end{definition}

\begin{corollary}
Let $(G,\chi)$ be a 2-connected plane graph. For $k\ge 1$, let $E_k$ be the plane graph with $2$ vertices and $k$ parallel edges. For $k\ge 1$, let $C_k$ be the cycle graph with $k$ edges (where we define $C_1 = E_1$ and $C_2 = E_2$). 
\begin{itemize}
    \item If $(G,\chi) = E_1=C_1$ or $(G,\chi) = E_2=C_2$, then $(G,\chi)$ has exactly 1 consistent sandpile torsor action.
    \item If $(G,\chi) = E_k$ or $(G,\chi) = C_k$ for $k\ge 3$, then $(G,\chi)$ has exactly 2 consistent sandpile torsor actions. 
    \item If $(G,\chi)$ is any other 2-connected plane graph, then $(G,\chi)$ has exactly 4 consistent sandpile torsor actions. 
\end{itemize}
\end{corollary}
\begin{proof}
We showed in Theorem~\ref{thm:onestructure} that every consistent sandpile torsor action must be equivalent to $r_{(G,\chi)}$, $\overline{r}_{(G,\chi)}$, $r^{-1}_{(G,\chi)}$, or $\overline{r}^{-1}_{(G,\chi)}$. 

If $(G,\chi)$ is 2-connected, but not equal to $C_k$ for some $k$, then it must have a vertex $c$ with degree at least $3$. By 2-connectedness, there exists some $T \in \T(G)$ such that $c$ is a leaf vertex of $T$. Let $g$ be the unique edge of $T$ incident to $x$. Furthermore, let $f = \chi(c,g)$, $h = \chi(c,f)$, and $\{c,s\} = \i(f)$. It follows from the definition of rotor-routing that
\[r_{(G,\chi)}([c-s],T) = T \setminus g \cup f = \overline{r}_{(G,\chi)}([c-s],T\setminus g \cup h).\]
Because $T \not= T\setminus g \cup h$, we can conclude that $r_{(G,\chi)} \not= \overline{r}_{(G,\chi)}$ and $r^{-1}_{(G,\chi)} \not= \overline{r}_{(G,\chi)}$ for any 2-connected $(G,\chi)$ not equal to $C_k$ for some $k$.  

If $(G,\chi)$ is 2-connected, but not equal to $E_k$ for some $k$, then it must have at least $3$ vertices. If we remove all but one edge out of every set of parallel edges, the resulting graph must contain a cycle or there would be a cut edge. After adding in the parallel edges (taking the innermost edge of each), the cycle remains and forms a face of $G$ that is bounded by at least $3$ edges. Let $C$ be the set of edges which bound this face and choose an arbitrary $f \in C$. Since $C \setminus f$ is acyclic, there must exist some $T \in \T(G)$ such that $C \setminus f \in T$. Suppose that when considering the edges of $C$ in counterclockwise order, the edge $h$ comes just before $f$, and the edge $g$ comes right after $f$. Furthermore, let $c = \i(f) \cap \i(g)$ and $s = \i(f) \cap \i(h)$. It follows from the definition of rotor-routing that
\[r_{(G,\chi)}([c-s],T) = T \setminus g \cup f \text{ and } \overline{r}^{-1}_{(G,\chi)}([c-s],T) =\overline{r}_{(G,\chi)}([s-c],T) = T \setminus h \cup f.\]
Because $T\setminus g \cup f \not= T\setminus h \cup f$, we can conclude that $r_{(G,\chi)} \not= \overline{r}^{-1}_{(G,\chi)}$ and $r^{-1}_{(G,\chi)} \not= \overline{r}_{(G,\chi)}$ for any 2-connected $(G,\chi)$ not equal to $E_k$ for any $k$. 

When $(G,\chi) = E_k$ for some $k$, the actions $r_{(G,\chi)}$ and $\overline{r}^{-1}_{(G,\chi)}$ as well as $\overline{r}_{(G,\chi)}$ and $r^{-1}_{(G,\chi)}$ are equivalent since both vertices are incident to the same edges, but in the opposite order. When $(G,\chi) = C_k$ for some $k$, the actions $r_{(G,\chi)}$ and $\overline{r}_{(G,\chi)}$ as well as $r^{-1}_{(G,\chi)}$ and $\overline{r}^{-1}_{(G,\chi)}$ are equivalent since reversing the ribbon structure has no effect on $C_k$. 

The result follows if we show that $r_{(G,\chi)} \not= r^{-1}_{(G,\chi)}$ when $(G,\chi)$ is a 2-connected ribbon graph other than $E_1=C_1$ or $E_2=C_2$. In particular, we need to show that for some $D \in \Div^0(G)$, we have $[2D] \not= [0]$. There are a few ways to show this, but we will use the chip-firing perspective. 

First, suppose that every vertex of $G$ has degree at most $2$. By our previous reasoning, $(G,\chi)$ must be equal to $C_k$ for some $k$. It is well known and straightforward to check that $\Pic^0(C_k) = \Z/k\Z$ (\cite[Problem 2.7]{divisorsbook}). Thus, for $k>2$, there are elements not equal to their own inverse. This means that we can assume that there is some $s \in V(G)$ such that $\deg(s) > 2$. Let $c\in V(G)$ be any other vertex. We will show that $[2c-2s] \not= [0]$ and the result follows. First, recall that $[2c-2s] = [0]$ if and only if there is a sequence of firing moves from $2c-2s$ to $0$. Suppose that such a sequence of firings exist. Since firing every vertex is equivalent to firing no vertices, we can assume that not every vertex is fired. The only way to reduce the number of chips on a vertex is by firing it, which means that $c$ must fire. After $c$ fires, any adjacent vertices other than $s$ have a positive number of chips, so they must also fire. By recursion, it follows that a vertex $v$ must fire if there is a path from $c$ to $v$ that does not pass through $s$. By 2-connectedness, this means that every vertex other than $s$ must fire at least once. After these firings, $s$ has $\deg(s)-2$ chips on it, which is positive by assumption. Thus $s$ must fire as well. However, this is a contradiction because we assumed that not every vertex is fired. Thus, $[2c-2s] \not= [0]$ and $r_{(G,\chi)} \not= r^{-1}_{(G,\chi)}$. 
\end{proof}

\section{Extension to Regular Matroids}\label{sec:matroids}

Many properties of graphs can be generalized to \emph{regular matroids}. In this section, we will explore how the ideas from the previous sections can be applied to regular matroids and we give conjectures for extensions of Theorems~\ref{thm:rrconsistent} and~\ref{thm:onestructure}.

For this section, we assume that the reader is familiar with oriented matroid theory (for necessary definitions, see e.g.~\cite{Oxley,Oriented}). However, we will not need any of the results from previous sections. 

\subsection{The BBY Algorithm}

Suppose that $M = (E,\Bcal, \chi)$ is an oriented matroid that is also regular. In particular, $E$ is the \emph{ground set}, $\Bcal$ is the set of \emph{bases}, and $\chi$ is a \emph{chirotope} which gives $M$ its orientation. %As a running example, we will consider the graphic oriented matroid associated with $K_4$ minus an edge with the orientations given in Figure~\ref{fig:k4minus1}. This matroid has a ground set of size $5$, $3$ pairs of signed circuits, $6$ pairs of signed cocircuits, and $8$ bases. 

\begin{remark}
Throughout this section, all of our regular matroids will be oriented even if we don't explicitly say so. This orientation is arbitrary and has only a minor cosmetic effect on the results. 
\end{remark}

Bacher, de la Harpe, and Nagnibeda showed that the sandpile group of a graph can also be described in terms of a graph's \emph{flow lattice} and \emph{cut lattice}~\cite{Bacher}. The analogue on regular matroids are the \emph{circuit lattice} and \emph{cocircuit lattice}. Consider the set of \emph{signed circuits} of $M$, which we write as vectors in $\{+1,-1,0\}^{E}$. These vectors generate a sublattice of $\Z^{E}$ called the \emph{circuit lattice}. Similarly, the \emph{signed cocircuits} of $M$ generate a sublattice of $\Z^{E}$ called the \emph{cocircuit lattice}.

\begin{definition}\label{def:sandmat}
The \emph{sandpile group} of a regular matroid $M$, written $\Scal(M)$, is the quotient of $\Z^{E}$ by the direct sum of the circuit and cocircuit lattices. 
\end{definition}

This sandpile group is also called the \emph{Jacobian} or \emph{critical group} of $M$. When $M$ is the \emph{cycle matroid} associated with a connected graph $G$, the groups $\Scal(M)$ and $\Pic^0(G)$ are isomorphic (see~\cite{Bacher,Biggs,alexthesis} for more details on this connection). 

\begin{theorem}[sandpile matrix-tree theorem for regular matroids~{\cite[Theorem 4.6.1]{Merino}}]\label{thm:regmtt}
For any regular matroid $M = (E,\Bcal, \chi)$, we have $|\Scal(M)| = |\Bcal|$. 
\end{theorem}

For a graph $G$, we defined a sandpile torsor action to be a free transitive action of $\Pic^0(G)$ on $\T(G)$. Additionally, for our action to be ``natural'' on a graph, we needed to introduce a ribbon structure and restrict to plane graphs. The analogue for regular matroids is to find a free transitive action of $\Scal(M)$ on $\Bcal$ after introducing some additional structure on $M$. Backman, Baker, and Yuen constructed a family of such actions which depend on a choice of \emph{acyclic circuit and cocircuit signatures}~\cite{BBY,Yuenthesis}. We will describe their construction below, but first we need a few definitions. 

For convenience, we will sometimes express elements of $\Z^E$ as integer linear combinations of ground set elements and sometimes as integer vectors. For example, if $E = \{e_1,e_2,e_3,e_4,e_5\}$, then $e_1 + 3e_3$ and $(1,0,3,0,0)$ represent equivalent elements of $\Z^E$. Given $D \in \Z^E$, we write $[D]$ for the equivalence class of $\Scal (M)$ containing $D$. Notice that $\Scal(M)$ is generated by elements of the form $[e]$ for $e \in E$.  

For each circuit (resp. cocircuit), there are two signed circuits (resp. cocircuits) which differ by multiplication by $-1$. A \emph{circuit signature} (resp. \emph{cocircuit signature}) is a choice of one signed circuit (resp. cocircuit) for each unsigned circuit (resp. cocircuit). A circuit signature (resp. cocircuit signature) is called \emph{acyclic} if no positive linear combination of signed circuits (resp. cocircuits) adds to zero. We will call $(\sigma,\sigma^*)$ a \emph{pair of acyclic signatures} if $\sigma$ is an acyclic circuit signature and $\sigma^*$ is an acyclic cocircuit signature. 

%\begin{example}
%Consider the oriented matroid represented by the graph in Figure~\ref{fig:k4minus1} with the orientations given. The following is an acyclic circuit and cocircuit signature for this matroid:

%\[\sigma= \begin{array}{ccccccc} & e_1 & e_2 & e_3 & e_4 & e_5 &\\ (&+1,&+1,&0,&0,&-1&)\\ (&+1,&+1,&+1,&-1,&0&)\\ (&0,&0,&+1,&-1,&+1&)\end{array} \hspace{.3 cm} \sigma^* = \begin{array}{ccccccc} & e_1 & e_2 & e_3 & e_4 & e_5 &\\(&+1,&-1,&0,&0,&0&)\\ (&+1, &0, &-1, &0, &+1&)\\ (&+1, &0, &0, &+1, &+1&)\\ (&0,&+1,&-1,&0,&+1&)\\ (&0, &+1, &0, &+1, &+1&)\\ (&0,&0,&+1,&+1,&0&)\end{array}\]
%\vspace{.3 cm}

%Note that the lowest index element is always chosen to be positive in this example. It is not hard to see that such a choice will always give a pair of acyclic signatures. \anote{cite chi ho thesis here?}
%\end{example}

\begin{definition}\label{def:msta}
A \emph{matroidal sandpile torsor algorithm} $\alpha$ is a function whose input is a regular matroid $M = (E,\Bcal,\chi)$ and a pair of acyclic signatures $(\sigma,\sigma^*)$, and whose output is a free transitive action of $\Scal(M)$ on $\Bcal$. For a specific $M$ and signatures $(\sigma,\sigma^*)$, we write $\alpha_{(M,\sigma,\sigma^*)}$ for the free transitive action, which we call a \emph{matroidal sandpile torsor action} on $M$.  
\end{definition}

\begin{remark}
The acyclic signatures in Definition~\ref{def:msta} play a similar role to the ribbon structure in Definition~\ref{def:torsor}: they are necessary to resolve ambiguity and allow us to define ``natural'' torsor actions. Definition~\ref{def:msta} should also require a form of automorphism invariance, but this detail is more distracting than enlightening for our purposes.
\end{remark}

Backman, Baker, and Yuen introduced a matroidal sandpile torsor algorithm which will denote \emph{$\BBY$}~\cite{BBY}. Their construction is a generalization of the \emph{Bernardi algorithm} which we discussed in Section~\ref{sec:background}~\cite{Yuengeometric}. 

The key insight for defining $\BBY$ is that, given a pair of acyclic signatures, every $B \in \Bcal$ can be associated with a $\{0,1\}^E$ vector such that these vectors are all distinct as elements of $\Scal(M)$. By Theorem~\ref{thm:regmtt}, this means that every equivalence class of $\Scal(M)$ is associated with a unique element of $\Bcal$. Thus, we obtain a free transitive action through pointwise addition. All that is left is to show how to associate a basis with a $\{0,1\}^E$ vector. 

For any $e \not\in B$, there is a unique circuit contained in $e \cup B$, which is called the \emph{fundamental circuit of $e$} (see~\cite[Corollary 1.2.6]{Oxley}). Similarly, for any $e \in B$, there is a unique cocircuit contained in $e \cup (E \setminus B)$, which is called the \emph{fundamental cocircuit of $e$} . To determine whether the $e$ entry of our vector should be $1$ or $0$, first check whether or not $e$ is in $B$. If $e \not\in B$, then consider the fundamental circuit $C$ of $e$. The signature $\sigma$ contains one of the the two signed circuits associated with $C$. If $e$ is positive in this signed circuit, then our $e$ entry is $1$, otherwise our $e$ entry is $0$. Similarly, if $e \in B$, then consider the fundamental cocircuit of $e$, and assign the value based on the associated signed cocircuit in $\sigma^*$. See Example~\ref{ex:matroid1} for a demonstration of this algorithm. 

\begin{remark}
The construction given by Backman, Baker, and Yuen is stated differently than our construction. In particular, the original formulation involves the set $\Gcal$ of \emph{circuit-cocircuit minimal orientations}. Nevertheless, these two perspectives are equivalent (see \cite[Section 6.5]{alexthesis}).
\end{remark}
\begin{theorem}[{\cite[Theorem 1.2.2]{BBY}}]\label{thm:BBYbij}
The above construction produces a matroidal sandpile torsor algorithm. 
\end{theorem}

Backman, Baker, and Yuen prove this theorem geometrically, using \emph{zonotopal tilings}.

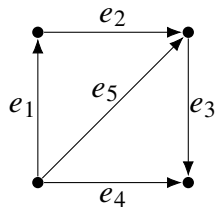
\begin{figure}
\begin{center}
\begin{tikzpicture}
    
    \tikzstyle{every node} = [circle,fill,inner sep=1pt,minimum size = 1.5mm]
    \node(a) at (0,0) {};
    \node(b) at (0,2){};
    \node(c) at (2,2) {};
    \node(d) at (2,0){};

    \tikzstyle{every node} = [draw = none,fill = none]
    \draw [-{Latex[length=2mm,width=1.5mm]}](a) -- (b);
    \draw [-{Latex[length=2mm,width=1.5mm]}](b) -- (c);
    \draw [-{Latex[length=2mm,width=1.5mm]}](c) -- (d);
    \draw [-{Latex[length=2mm,width=1.5mm]}](a) -- (c);
    \draw [-{Latex[length=2mm,width=1.5mm]}](a) -- (d);
    
    \node (o) at (-.2,1){$e_1$};
    \node (o) at (1,2.2){$e_2$};
    \node (o) at (2.2,1){$e_3$};
    \node (o) at (1,-.2){$e_4$};
    \node (o) at (.9,1.2){$e_5$};

\end{tikzpicture}
\caption{In Example~\ref{ex:matroid1}, we consider the oriented matroid represented by the oriented graph above.}\label{fig:k4minus1}
\end{center}
\end{figure}

\begin{example}\label{ex:matroid1}
The oriented graph in Figure~\ref{fig:k4minus1} \emph{represents} a regular matroid, which we call $M$. The pair $(\sigma,\sigma^*)$ below is a pair of acyclic signatures for $M$.\footnote{We obtain $\sigma$ and $\sigma^*$ by ordering the elements of $E$ and then setting the minimal nonzero element of each circuit and cocircuit to be positive. This method will always produce a pair of acyclic signatures, but not all acyclic signatures are produced this way (see~\cite[Example 5.1.4]{Yuenthesis}).}

\[\sigma= \begin{array}{ccccccc} & e_1 & e_2 & e_3 & e_4 & e_5 &\\ (&+1,&+1,&0,&0,&-1&)\\ (&+1,&+1,&+1,&-1,&0&)\\ (&0,&0,&+1,&-1,&+1&)\end{array} \hspace{.3 cm} \sigma^* = \begin{array}{ccccccc} & e_1 & e_2 & e_3 & e_4 & e_5 &\\(&+1,&-1,&0,&0,&0&)\\ (&+1, &0, &-1, &0, &+1&)\\ (&+1, &0, &0, &+1, &+1&)\\ (&0,&+1,&-1,&0,&+1&)\\ (&0, &+1, &0, &+1, &+1&)\\ (&0,&0,&+1,&+1,&0&).\end{array}\]

Suppose that we wish to compute $\BBY_{(M,\sigma,\sigma^*)}([e_3],\{e_2,e_3,e_5\})$. First, we find the $\{0,1\}^E$ vector associated with $\{e_2,e_3,e_5\}$ with respect to the pair $(\sigma,\sigma^*)$. To do this, we look at the fundamental circuit or cocircuit of each element. This gives the vector $(1,0,1,0,1)$. Next, we add $(0,0,1,0,0)$ and get $(1,0,2,0,1)$. Finally, we need to find which basis is associated with the equivalence class of $\Scal(M)$ containing $(1,0,2,0,1)$. Through trial and error (or more clever means), one can find that $[(1,0,2,0,1)] = [(1,1,0,1,1)]$, and this vector is associated with the basis $\{e_1,e_3,e_4\}$. It follows that
\[\BBY_{(M,\sigma,\sigma^*)}([e_3],\{e_2,e_3,e_5\}) = \{e_1,e_3,e_4\}.\]
\end{example}
\subsection{Consistency Conjectures}

Let $M = (E,\Bcal,\chi)$ be an oriented regular matroid and $e \in E$. Define $M\setminus e$ and $M/e$ as given in~\cite[Propositions 3.3.1 and 3.3.2]{Oriented}. It is well known that $M\setminus e$ and $M/e$ are also regular matroids (\cite[Proposition 3.2.5]{Oxley}). Furthermore, as discussed in~\cite[Section 3.3]{Oriented}, for any pair $(\sigma,\sigma^*)$ of acyclic signatures on $M$, there is a pair $(\sigma\setminus e,\sigma^*\setminus e)$ of induced acyclic signatures on $M \setminus e$ and a pair $(\sigma/e,\sigma^*/ e)$ of induced acyclic signatures on $M /e$.

%\begin{example}
%In our running example, from $(\sigma,\sigma^*)$, we get acyclic signatures $(\sigma_{e_4},\sigma^*_{e_4})$ for $M/e_4$ and $({}_{e_3}\sigma,{}_{e_3}\sigma^*)$ for $M \setminus e_3$. In particular, these are the following:

%\[\sigma_{e_4}= \begin{array}{cccccc} & e_1 & e_2 & e_3 & e_5 &\\ (&+1,&+1,&0,&-1&)\\ (&+1,&+1,&+1,&0&)\\ (&0,&0,&+1,&+1&)\end{array} \hspace{.3 cm} \sigma_{e_4}^* = \begin{array}{cccccc} & e_1 & e_2 & e_3 & e_5 &\\(&+1,&-1,&0,&0&)\\ (&+1, &0, &-1, &+1&)\\  (&0,&+1,&-1,&+1&)\end{array}\]
%\vspace{.3 cm}
%\[{}_{e_3}\sigma= \begin{array}{cccccc} & e_1 & e_2 & e_4 & e_5 &\\ (&+1,&+1,&0,&-1&)\\ \end{array} \hspace{.3 cm} {}_{e_3}\sigma^* = \begin{array}{cccccc} & e_1 & e_2 & e_4 & e_5 &\\(&+1,&-1,&0,&0&)\\ (&+1, &0, &0, &+1&)\\ (&0,&+1,&0,&+1&)\\ (&0,&0,&+1,&0&)\end{array}\]
%\end{example}

Notice that the representatives of equivalence classes of $\Scal(M)$ are generated by the elements of $E$. The following definition is a matroidal analogue to consistency (Definition~\ref{def:consistent}). 

\begin{definition}\label{def:matroidconsistent}
A matroidal sandpile torsor algorithm $\alpha$ is \emph{consistent} if for every regular matroid $M = (E,\Bcal,\chi)$, pair of acyclic signatures $(\sigma,\sigma^*)$, $B \in \Bcal$, and $f \in E$, the following equalities hold (where $B' = \alpha_{(M,\sigma,\sigma^*)}([f],B)$):
\begin{enumerate}
    \item For $e \in (B \cap B') \setminus f$, we have 
    \[ \alpha_{(M/e,\sigma/ e,\sigma^*/ e)}([f],B\setminus e) = B'\setminus e.\]
    \item For any $e \not\in B \cup B'\cup f$, we have
    \[ \alpha_{(M\setminus e,\sigma\setminus e,\sigma^*\setminus e)}([f],B) = B'.\]
\end{enumerate}
\end{definition}

\begin{remark}
It is entirely possible that Definition~\ref{def:matroidconsistent} needs to be tweaked a bit in order for Conjectures~\ref{conj:BBYconsistent} and~\ref{conj:allconsistent} to be plausible. For example, it may be necessary (or convenient) to include an analogue of Condition 3 from Definition~\ref{def:consistent}. 
\end{remark}
\begin{conj}\label{conj:BBYconsistent}
The $\BBY$ matroidal sandpile torsor algorithm is consistent. 
\end{conj}

For a circuit signature $\sigma$, let $\overline{\sigma}$ be the circuit signature made up of the signed circuits not in $\sigma$. Define $\overline{\sigma}^*$ similarly for cocircuit signatures.  

\begin{definition}
Suppose $\alpha$ is a matroidal sandpile torsor algorithm. Define $\alpha'$, $\alpha''$, and $\alpha'''$ such that for any regular matroid $M = (E,\Bcal,\chi)$ and pair of acyclic signatures $(\sigma,\sigma^*)$, we have
\[\alpha_{(M,\sigma,\sigma^*)} = \alpha'_{(M,\overline{\sigma},\sigma^*)} = \alpha''_{(M,\overline{\sigma},\overline{\sigma}^*)} = \alpha'''_{(M,\sigma,\overline{\sigma}^*)}.\]
\end{definition}

\begin{definition}
Matroidal sandpile torsor algorithms $\alpha$ and $\beta$ have the same \emph{structure} if $\beta \in \{\alpha,\alpha',\alpha'',\alpha'''\}$.
\end{definition}

\begin{conj}\label{conj:allconsistent}
Every consistent matroidal sandpile torsor algorithm has the same structure as $\BBY$. 
\end{conj}

\section*{Acknowledgements}
The first author was partially supported by the United States Army Research Office under grant number W911NF2010133. We would like to thank Matt Baker, Richard Kenyon, Caroline Klivans, Cyrus Peterpaul, Lilla T\'othm\'er\'esz, Chi Ho Yuen, and the anonymous referees for helpful conversation, comments and suggestions. In particular, we are grateful to Caroline Klivans for calling our attention to the main conjecture, Richard Kenyon for suggesting we look for a consistency condition related to contraction and deletion, and Lilla T\'othm\'er\'esz for Example~\ref{ex:lilla}.

\appendix
\section{Sink-free Rotor Configurations and Unicycles}\label{app:unicycles}

Here, we introduce \emph{sink-free rotor configurations} and \emph{unicycles}, which were explored in both \cite{Holroyd} and \cite{CCG}. These objects allow us to prove the remaining results from Section~\ref{sec:funfacts}. 

\begin{definition} A \emph{sink-free rotor configuration} is an assignment of an incident edge to every vertex of $G$. 
\end{definition}

We obtain a sink-free rotor configuration from a rotor configuration by assigning a rotor to the sink vertex. 

\begin{definition}\label{def:unicycle}
A \emph{unicycle} is a pair $(\rho^*,x)$ where $\rho^*$ is a sink-free rotor configuration with exactly one directed cycle and $x$ is a vertex on this cycle. 
\end{definition}

If we are given a unicycle $(\rho^*,x)$, we can get a new unicycle by replacing $\rho^*\la x\ra$ with $\chi(x,\rho^*\la x \ra)$ and then replacing $x$ with the other vertex incident to the new value of $\rho^* \la x \ra$. In other words, we apply one step of Algorithm~\ref{alg:rotor}. By~\cite[Lemma 3.3]{Holroyd}, this will always output a new unicycle. This action of Algorithm~\ref{alg:rotor} on unicycles is what Holroyd et al. call the \emph{rotor-routing process}. For clarity, we will only use the word \emph{process} when working with unicycles. 

\begin{lemma}[{\cite[Lemma 4.9]{Holroyd}}]\label{lem:fullspins} Let $(\rho^*,x)$ be a unicycle on a ribbon graph with $m$ edges.
If we iterate the rotor-routing process $2|E(G)|$ times, the chip traverses each edge
of $G$ exactly once in each direction, each rotor makes exactly one full turn, and the final unicycle is $(\rho^*,x)$.
\end{lemma}

Let $(G,\chi)$ be a ribbon graph, $T \in \T(G)$, and $c,s \in V(G)$. Suppose that there is some $f \in E(G)$ such that $\i(f) = \{c,s\}$. Let $T^{*f}_s$ be the sink-free rotor configuration defined by:
\begin{equation}\label{eq:tstardef}T^{*f}_s\la x \ra := \begin{cases} T_s\la x \ra&\text{if $x \not=s$,}\\f &\text{if $x=s$.}\end{cases}\end{equation}

\begin{lemma}The pair $(T^{*f}_s,c)$ is a unicycle.
\end{lemma}
\begin{proof}
The rotor configuration $T_s$ contains a directed path from $c$ to $s$ and no cycles. After adding a rotor from $s$ to $c$, we have a unique cycle which must contain $c$. 
\end{proof}

This leads directly to a unicycle interpretation of the rotor-routing algorithm. Suppose that Algorithm~\ref{alg:rotor} requires $n$ iterations of the while loop to terminate given the inputs $T$ and $c-s$. For each $k\in [0,n]$, let $\rho_k$ be the rotor configuration after $k$ iterations of the while loop in Algorithm \ref{alg:rotor}. Similarly, let $(\rho^*_k,c_k)$ be the unicycle obtained from $(T^{*f},c)$ after $k$ steps of the rotor-routing process. 

\begin{lemma}\label{lem:rruni}
For any $k\in [0,n]$, we have
\[\rho^*_k\rotdir{v} = \begin{cases}\rho_k\rotdir{v} & \text{ for $v \neq s$}\\f & \text{ for $v = s$}.\end{cases} \]
\end{lemma}
\begin{proof}
This follows immediately from the fact that, until the chip leaves $s$ on the $(n+1)$ step of the rotor-routing process, the rotor at $s$ remains fixed and does not affect the path of the chip or the rotation of the other rotors.
\end{proof}

\begin{proof}[Proof of Proposition~\ref{prop:nooverspins}]
For each $k$, let $(\rho^*_k,c_k)$ denote the $k$th step of the rotor-routing process initialized at $(T^{*f},c)$ and suppose that $n$ is the smallest integer such that $c_n = s$. As a consequence of Lemma \ref{lem:fullspins}, the chip must visit every vertex before the unicycle returns to its original value. Thus, this must happen after step $n$. Then by Lemma \ref{lem:fullspins}, no rotor can have completed more than a full rotation by step $n$. By the equivalence established in Lemma \ref{lem:rruni}, this implies that no rotor can complete more than a full rotation before Algorithm \ref{alg:rotor} terminates. 
\end{proof}
 
For a sink-free rotor configuration $\rho^*$ with a unique directed cycle $C$, let $\overline{\rho^*}$ be the sink-free rotor configuration obtained by reversing the direction of the rotors that make up $C$ and keeping the others fixed. 

\begin{prop}\cite[Proposition 9]{CCG}\label{prop:reverse}
The following are equivalent:
\begin{itemize}
\item $(G,\chi)$ is a plane graph.
\item For every unicycle $(\rho^*,x)$ on $(G,\chi)$, repeated applications of the rotor-routing process eventually produce the unicycle $(\overline{\rho^*},x)$.  
\end{itemize}
\end{prop}

\begin{lemma}\cite[Corollary 4.11]{Holroyd}\label{lem:leftright}
Let $(G,\chi)$ be a plane graph and $(\rho^*,x)$ be a unicycle. Suppose that the rotor-routing process is performed until we obtain the unicycle $(\overline{\rho^*},x)$. In the process of rotor-routing, the chip crosses every edge to the left of $C$ in both directions and no edges to the right of $C$.
\end{lemma}

\begin{remark}The authors of~\cite{CCG} and \cite{Holroyd} use a clockwise convention, so their arguments relating to Proposition \ref{prop:reverse} and Lemma~\ref{lem:leftright} switch left and right vertices. When $G$ is restricted to the square lattice, Lemma~\ref{lem:leftright} was first proven in~\cite{eulerianwalkers2}. 
\end{remark}

We conclude with proofs of Lemmas~\ref{lem:looprev} and~\ref{lem:looprev2}.

\begin{proof}[Proof of Lemma~\ref{lem:looprev}]
Let $T_s^{*f}$ be the sink-free rotor configuration defined in~(\ref{eq:tstardef}) and let $\overline{T_s^{*f}}$ be the sink-free rotor configuration defined above Proposition \ref{prop:reverse}. Apply the rotor-routing process to $(T_s^{*f},c)$ until the chip reaches $s$. This must occur before we reach the unicycle $(\overline{T_s^{*f}},c)$ because the rotor at $s$ must move to reach $\overline{T_s^{*f}}$. By Lemma~\ref{lem:leftright}, none of the edges to the right of $C$ will be crossed. The proof follows from Lemma~\ref{lem:rruni}
\end{proof}

\begin{proof}[Proof of Lemma~\ref{lem:looprev2}]
For each $k\in \Z_{\ge 0}$, let $(\rho^*_k,c_k)$ be the sink-free rotor configuration obtained after $k$ steps of the rotor-routing process initialized at $(T^{*f},c)$. Let $m:= |E(G)|$ and let $C_k$ be the unique directed cycle in $\rho^*_k$. It is immediate that, for $k\in [0,n]$, the cycle $C_k$ is also a directed cycle of $\rho_k$ if and only if $s \not\in V(C_k)$. Thus, the lemma follows if we can show that for any $i \in [0,n]$ such that $s \not\in C_i$, there exists some $j \in [0,n]$ where $(\rho^*_j,c_j) = (\overline{\rho^*_i},c_i)$. 

By Proposition~\ref{prop:reverse} and Lemma \ref{lem:fullspins}, there exists such a $j$ in the interval $[0,2m-1]$. If $i > j$, then the result is trivially true, so we may assume without loss of generality that $i < j$. If $s$ is to the right of $C_i$, then by Lemma~\ref{lem:leftright}, the cycle will reverse before $s$ is reached and the result follows. Alternatively, if $s$ is to the left of $C_i$, then the unicycle $(\rho^*_j,c_j)$ is not reached until after the chip enters $s$ for the $\deg(s)$th time. We claim that $c_{2m-1} = s$. It follows from the claim that $j>2m-1$, because the chip only reaches $s$ a total of $\deg(s)$ times after $2m$ steps of the rotor-routing process (by Lemma~\ref{lem:fullspins}). This is a contradiction and the result follows. 

To prove the claim, we first note that there must be a directed edge from $c_{2m-1}$ to $c_{2m} = c_0 = c$ in $\rho^*_{2m} = \rho^*_0$. If $c_{2m-1} \not\in V(C_0)$, then $C_0$ is also a directed cycle in $\rho^*_{2m-1}$, and it follows that $C_0 = C_{2m-1}$. This is impossible because $c_{2m-1}\in V(C_{2m-1})$ by Definition~\ref{def:unicycle}. The claim follows from the fact that $s$ is the only vertex in $V(C_0)$ whose rotor is directed towards $c$ in $\rho^*_0$. 
\end{proof}

\section{Case 4 of Theorem~\ref{thm:onestructure}}\label{app:case4}

Case 4 of Theorem~\ref{thm:onestructure} is a lot more complicated than the other cases because no edges are shared between $T$ and $\alpha_{(G,\chi)}([c-s],T)$, so we cannot directly apply the definition of consistency. In this appendix, we prove the results necessary for this final case.  

At a high level, our proof of Case 4 of Theorem \ref{thm:onestructure} is as follows. We suppose there exists a counterexample: a non-rotor-routing consistent sandpile torsor algorithm $\alpha$. In Section \ref{case4:lems}, we use several lemmas to construct highly restrictive conditions that $\alpha$ must satisfy. Then in Section \ref{case4:pf}, we show that the conditions established in Section \ref{case4:lems} give rise to a parameterized class of graphs that we call \emph{telescope graphs}. Then if $(G,\chi)$ (as defined in the statement of case 4) is not a telescoping graph and $\alpha_{(G,\chi)} \neq r_{(G,\chi)}$, then $\alpha$ trivially violates the conditions established in Section \ref{case4:lems}. We finish with a direct proof $\alpha_{(G,\chi)} = r_{(G,\chi)}$ if $(G,\chi)$ is a telescope graph.

\subsection{Properties of a Potential Counterexample}
\label{case4:lems}

Throughout Section \ref{case4:lems}, assume $G$ is a 2-connected plane graph such that $|E(G)| = 2|V(G)| - 2$. In other words, any spanning tree of $G$ contains exactly half the edges of $G$. We additionally assume that $\alpha$ is a consistent sandpile torsor algorithm such that $\alpha_{(G',\chi')} = r_{(G',\chi')}$ for any proper minor $(G',\chi')$ of $(G,\chi)$.

Throughout this subsection, we will allow $(c-s,T)$ to denote \emph{any} single-step pair or source-turn pair as needed by the argument. Note that this differs from the definition of $(c-s,T)$ given in the statement of Case 4 in the proof of Theorem \ref{thm:onestructure}.

The goal of this subsection is to establish Corollary \ref{cor:anystepflips}, which states that if $\alpha_{(G,\chi)}$ disagrees with $r_{(G,\chi)}$ on any single-step pair $(c-s,T)$, then for any $T'$ such that $(c-s,T')$ is a single step pair, $\alpha_{(G,\chi)}([c-s],T') = E(G)\setminus T'$. A similar result holds for reverse single-step pairs. We do this in several steps:
\begin{description}
\item[Step 1] We first show that for any single-step pair $(c-s,T)$, if $\alpha_{(G,\chi)}([c-s],T) \not= r_{(G,\chi)}([c-s],T)$, then we must have $\alpha_{(G,\chi)}([c-s],T) = E(G)\setminus T$. An analogous result holds for reverse single-step pairs (see Lemmas \ref{lem:singstepnocase2} and \ref{lem:revstepnocase2}).
\item[Step 2] Given the conditions of step 1, we show that $\alpha_{(G,\chi)}([c-s],T') = E(G)\setminus T'$ for any $T'$ that can be reached from $T$ via single-step moves and reverse single-step moves that do not affect the rotor at $g$ (see Lemma \ref{lem:commutepairs}).
\item[Step 3] Using similar arguments to those used to prove Theorem \ref{thm:sourceturn}, we show that the set of trees listed in Step 2 include all trees $T'$ for which $(c-s,T')$ is a single-step pair or for which $(s-c,T')$ is a reverse single-step pair (see Lemma \ref{lem:anyvalid}). 
\item[Step 4] The corollary follows directly.
\end{description}

\begin{lemma}\label{lem:singstepnocase2}
Let $(c-s,T)$ be a single-step pair from $g$ to $f$. Then, \[\alpha_{(G,\chi)} ([c-s],T) \in \{T\setminus g \cup f, E(G) \setminus T\}.\]
\end{lemma}
\begin{proof}
Let $\what T = \alpha_{(G,\chi)} ([c-s],T)$. Suppose that $\what T \not= T\setminus g \cup f$. Then, by the same logic we used in the proof of Theorem~\ref{thm:onestructure},  we know that $V(G) \setminus (T\Delta\what T) \subseteq \{f,g\}$ (i.e. (\ref{eq:nonsharededges}) holds for all single-step pairs, not just source-turn pairs). Furthermore, by the condition that $2|T| = |E(G)|$, and since all spanning trees are the same size, we must have \[\what T \in \{E(G) \setminus T, (E(G) \setminus T) \setminus f \cup g\}.\]

This means that we just need to show that $\what T \not= (E(G) \setminus T) \setminus f \cup g$. This is true for precisely the reasoning we used in Case 2 of Theorem~\ref{thm:onestructure}. 
\end{proof}

\begin{lemma}\label{lem:revstepnocase2}
Let $(s-c,T)$ be a reverse single-step pair from $f$ to $g$. Then, \[\alpha_{(G,\chi)} ([s-c],T) \in \{T\setminus f \cup g, E(G) \setminus T\}.\]
\end{lemma}
\begin{proof}
Let $\what T = \alpha_{(G,\chi)} ([s-c],T)$. Suppose that $\what T \not= T\setminus f \cup g$. Then, by the same logic we used in the proof of Theorem~\ref{thm:onestructure},  we know that $V(G) \setminus (T\Delta\what T) \subseteq \{f,g\}$ (i.e. (\ref{eq:nonsharededges}) also holds for reverse single-step pairs). Furthermore, by the condition that $2|T| = |E(G)|$, and since all spanning trees are the same size, we must have \[\what T \in \{E(G) \setminus T, (E(G) \setminus T) \setminus g \cup f \}.\]

This means that we just need to show that $\what T \not= (E(G) \setminus T) \setminus g \cup f$. However, we cannot use the same reasoning we used in Case 2 of Theorem~\ref{thm:onestructure}. Instead we give a new proof by contradiction. 

First, suppose that $x=s$ (where $\i(g) = \{x,c\}$) so that $f$ and $g$ are parallel. By the definition of reverse single-step pairs, $(c-s,T\setminus f \cup g)$ is a single-step pair from $g$ to $f$. By Lemma~\ref{lem:singstepnocase2}, we must have $\alpha_{(G,\chi)} ([c-s],T\setminus f \cup g) \in \{T, (E(G) \setminus T)\setminus g \cup f\}$. In the first case, we have $\what T = T\setminus f \cup g$ and the lemma follows. Otherwise, $(E(G) \setminus T)\setminus g \cup f$ must be a tree, which means that $E(G) \setminus T$ is also a tree (since $f$ and $g$ are parallel). Furthermore, since $x = s$, it is immediate that $c \prec_{(E(G) \setminus T)_s} x$. By Lemma~\ref{lem:singcriteria}, this implies that $(c-s,E(G) \setminus T)$ is a single-step pair from $g$ to $f$. 

By Lemma~\ref{lem:singstepnocase2}, we must have 
\[\alpha_{(G,\chi)} ([c-s],E(G) \setminus T) \in \{(E(G) \setminus T)\setminus g \cup f, T\}.\]
In the first case, it follows that
\[\alpha_{(G,\chi)} ([c-s],E(G) \setminus T) = \alpha_{(G,\chi)} ([c-s],T \setminus f \cup g),\]
which implies that $E(G) \setminus T = T \setminus f \cup g$. This is impossible unless $E(G) = \{f,g\}$, in which case the lemma is trivial. Thus we must have $\alpha_{(G,\chi)} ([c-s],E(G) \setminus T) = T$ and $\alpha_{(G,\chi)} ([s-c],T) = E(G) \setminus T$.

Now we consider the case where $x \not=s$ (where once again, $\i(g) = \{x,c\}$). Suppose for the sake of contradiction that $\what T =  (E(G) \setminus T) \setminus g \cup f$. Because $g \not\in T \cup \what T$, it follows by consistency that
\[\what T = \alpha_{(G\setminus g,\chi\setminus g)} ([s-c],T) = r_{(G\setminus g,\chi\setminus g)} ([s-c],T).\]

We can now find a contradiction by applying Algorithm~\ref{alg:rotor} with the sink at $c$ and the chip starting at $s$. Notice that $\what T_c\la s \ra = T_c \la s \ra = f$. Thus, the rotor at $s$ must rotate completely around. Furthermore, by Proposition~\ref{prop:nooverspins}, the chip must enter $s$ precisely $\deg(s)-1$ times (where we subtract 1 because the chip starts at $s$). These is no rotor at $c$, so the chip can never pass through $f$ to $s$. Thus, the chip must enter $s$ along every other incident edge. 

By Lemma~\ref{lem:singcriteria} and the definition of reverse single-step pairs, we know that on $G$, $c \prec_{(T\setminus f \cup g)_s} x$. In particular, this means that there is a path of edges in $T$ from $x$ to $s$ which does not pass through $c$. It follows that there is some $h \in T \setminus f$ that is incident to $s$. 

By our earlier reasoning, the chip must cross $h$ to $s$ at some point during rotor-routing. It follows from Proposition~\ref{prop:nooverspins} that this rotor must turn completely around. However, this means that $h$ must be the final position of this rotor. Because $h \in T$, it follows from \eqref{eq:nonsharededges} that $h \not\in (E(G) \setminus T) \setminus g \cup f$. Thus, we have a contradiction and $\what T \not= (E(G) \setminus T) \setminus g \cup f$.
\end{proof}

\begin{lemma}\label{lem:singrev}
Let $(s-c,T)$ be a reverse single-step pair from $f$ to $g$ (which implies that $(c-s,T \setminus f \cup g)$ is a  single-step pair from $g$ to $f$). Then, 
\[\alpha_{(G,\chi)}([s-c],T) = E(G) \setminus T \Longleftrightarrow \alpha_{(G,\chi)}([c-s],T \setminus f \cup g) = (E(G) \setminus T) \setminus g \cup f.\]
\end{lemma}
\begin{proof}
If $|E(G)| = 2$, the result is trivial. Otherwise, for the forward direction, we apply Lemma~\ref{lem:singstepnocase2}. This lemma says that it suffices to show that $\alpha_{(G,\chi)}([c-s],T \setminus f \cup g)\not= T$. If this inequality does not hold, then we have
\[T \setminus f \cup g = \alpha_{(G,\chi)}([c-s+s-c],T \setminus f \cup g) = \alpha_{(G,\chi)}([s-c],T) = E(G) \setminus T.\]
This is a contradiction because $T\setminus f \cup g \not= E(G) \setminus T$ unless $|E(G)| = 2$. 

The reverse direction is analogous after applying Lemma~\ref{lem:revstepnocase2}. 
\end{proof}

\begin{lemma}\label{lem:commutepairs}
Suppose that $(c-s,T)$ is a single-step pair such that $\alpha_{(G,\chi)}([c-s],T) = E(G) \setminus T$. Further suppose that for some $c',s' \in V(G)$ and $h,k \in E(G) \setminus\{f,g\}$, the pair $(c'-s',T)$ is a single-step or reverse single-step pair from $h$ to $k$. Then,
\[\alpha_{(G,\chi)}([c-s],T \setminus h \cup k) = (E(G) \setminus T) \setminus k \cup h.\]
\end{lemma}
\begin{proof}
By Lemmas~\ref{lem:singstepnocase2} and ~\ref{lem:revstepnocase2}, we know that
\begin{equation*}
\alpha_{(G,\chi)}([c'-s'],T) \in \{T \setminus h \cup k, E(G) \setminus T\}.
\end{equation*}
Furthermore, if $\alpha_{(G,\chi)}([c'-s'],T) = E(G)\setminus T$, then $[c'-s'] = [c-s]$. This is false, because
\[r_{(G,\chi)}([c-s],T) = T \setminus g \cup f \not= T \setminus h \cup k = r_{(G,\chi)}([c'-s'],T).\]
It follows that $\alpha_{(G,\chi)}([c'-s'],T) = T \setminus h \cup k$.

By the condition that $T_s\la c'\ra = h$, and because $(c'-s',T)$ is a single-step pair, we know that $T_s\la v \ra = (T \setminus h \cup k)_s\la v \ra$ for all $v \in V(G) \setminus c'$. In particular, $T_s\la c \ra = (T \setminus h \cup k)_s\la c \ra = g$. thus, by Lemma~\ref{lem:singcriteria}, $(c-s,T \setminus h \cup k)$ is a single-step pair from $g$ to $f$. By an analogous argument, $(c'-s',T\setminus g \cup f)$ is a single-step pair from $h$ to $k$. 

Suppose that the result does not hold (i.e. $\alpha_{(G,\chi)}([c-s],T \setminus h \cup k) \not= (E(G) \setminus T) \setminus k \cup h$). By Lemma~\ref{lem:singstepnocase2}, this means that $\alpha_{(G,\chi)}([c-s],T \setminus h \cup k) = \widetilde T$, where $\widetilde T := T \setminus \{h,g\} \cup \{k,f\}$. By Definition~\ref{def:revstep}, we know that $(s'-c',\widetilde T)$ is a reverse single-step pair from $k$ to $h$. By Lemma~\ref{lem:revstepnocase2}, 
\begin{align*}&\alpha_{(G,\chi)}([s'-c'],\widetilde T) = \widetilde T \setminus k \cup h = T \setminus g \cup f \not= E(G) \setminus T\\
\text{or }&\alpha_{(G,\chi)}([s'-c'],\widetilde T) = (E(G)\setminus T) \setminus \{k,f\} \cup \{h,g\}\not= E(G) \setminus T.
\end{align*}
However, we also have
\begin{gather*}
 E(G) \setminus T = \alpha_{(G,\chi)}([c-s],T) = \alpha_{(G,\chi)}([c'-s' + c - s + s' - c'],T) = \\
\alpha_{(G,\chi)}([c - s + s' - c'],T\setminus h \cup k) = \alpha_{(G,\chi)}([s' - c'],\widetilde T) 
\end{gather*}
This gives a contradiction and we have proven the result.
\end{proof}

The following lemma is a variant of Theorem~\ref{thm:sourceturn} where one rotor must remain fixed, but we are allowed to use single-step and reverse single-step moves, not just source-turn moves. The proof is a bit technical, but uses the same idea as the proof of Theorem~\ref{thm:sourceturn}.

\begin{lemma}\label{lem:anyvalid}
Let $T^{start},T^{goal} \in \T(G)$ be any two spanning trees such that $T^{start}_s\rotdir{c} = T^{goal}_s\rotdir{c} = g$. There exists a sequence $\{(c_i-s_i,T^i)\}_{i \in [0,k]}$ such that $T^0 = T^{start}$, $T^k = T^{goal}$, and for every $i \in [1,k]$:
\begin{itemize}
    \item $(c_i - s_i,T^{i-1})$ is a single-step pair or reverse single-step pair;
    \item $T^i = r_{(G,\chi)}([c_i - s_i], T^{i-1})$;
    \item $T_s^i\la c \ra = g$; and
    \item $c_i \neq s$.
\end{itemize}  
\end{lemma}

Before giving the proof, we give the following Lemma, which is a variation of Lemma~\ref{lem:contracttreetotree} under which the contracted edges are not incident to the root. 

\begin{lemma}
\label{lem:contracttreetotreeB}
Let $(G,\chi)$ be a 2-connected plane graph with two spanning trees $T,T' \in \T(G)$ and some $\rt \in V(G)$. Fix a connected $F \subset T\cap T'$ which does not contain any edges incident to $\rt$. Let $(\widetilde{G},\widetilde{\chi})$ be the ribbon graph obtained by contracting the edges in $F$, and $z$ be the vertex these edges contract into. 

If there exists a sequence of source-rotations that take $(T\setminus F)_{\rt}$ to $(T'\setminus F)_{\rt}$ in $(\widetilde{G},\widetilde{\chi})$ without turning the rotor at $z$, then there exists a sequence of source-rotations that take $T_\rt$ to $T'_\rt$ in $(G,\chi)$ without turning any rotors corresponding to edges in $F$.
\end{lemma}
\begin{proof}
The proof follows analogously to the proof of Lemma~\ref{lem:contracttreetotree}. The only difference is that a source-rotation at $z$ of a rotor configuration $\wtilde T_\rt$ on $(\widetilde{G},\widetilde{\chi})$ might not correspond to a source-rotation of $(\wtilde T\cup F)_\rt$ on $(G,\chi)$. This is not a problem by the condition that the rotor at $z$ remains stationary. 
\end{proof}

\begin{proof}[Proof of Lemma~\ref{lem:anyvalid}]

Recall the partial order $\prec_{\xi}$ from the proof of Theorem~\ref{thm:sourceturn} setting $\rt = s$. By Lemma \ref{lem:singswap}, all source-turn pairs are also single-step pairs. For simplicity, we will call a tree $T \in \T(G)$ \emph{valid} if $T_s\rotdir{c} = g$. For the sake of induction, suppose that that for all 2-connected proper minors $(G',\chi')$ containing $c$ and $s$ (allowing minors where disjoint sets of vertices merge into $c$ and $s$), the results of the lemma hold.

To prove this lemma, it suffices to prove that for every valid $T \in \T(G) \setminus T^{goal}$, there exists a valid $T'' \in \T(G)$ such that $T'' \prec_{\xi} T$ and it is possible to move from $T_s$ to $T''_s$ through a series of single-step and reverse single-step moves which never change the rotor at $c$.

Suppose there exists some $y \in V(G)\setminus s$ such that $\xi(T,y) = 1$ and $y$ is a leaf vertex of $T$. Then, we can freely rotate the rotor at $y$ until we get a tree $T''$ such that $T''\la y \ra = 0$. Then $T'' \prec_\xi T$ and we are done. 

Otherwise, let $y$ be a minimal element with respect to $\prec_{T_s}$ such that $\xi(T,y) = 1$. By minimality, $\xi(T,v) = 0$ for all $v \prec_{T_s} y$. As in the proof of Theorem \ref{thm:sourceturn}, we define 
\[V^y := \{v \mid v\prec_{T_s} y\} \text{ and } E^y := \{T_s\la v \ra \mid v \in V^y\}.\] 
By the same argument that we used in the proof of Theorem \ref{thm:sourceturn}, $V^y \subseteq \{v \mid v \prec_{T^{goal}_s} y\}$. We will consider 2 cases (where $\i(g) = \{x,c\}$).\\

\textbf{Case 1: } $x \npreceq_{T_s} y$.\\

Because $x \npreceq_{T_s} y$, it follows that $x \not\in V^y$. Furthermore, if $c \prec_{T_s} y$, then the path from $c$ to $y$ along $T$ must pass through $g$ (and therefore $x$), which would imply that $x \prec_{T_s} y$. Thus, we must have $c \nprec_{T_s} y$, which implies that $c \not\in V^y$ and $g \not\in E^y$. Then, by the claim stated in the proof of Theorem \ref{thm:sourceturn}, there exists some $T' \in \T(G)$ that can be reached from a series of source-turn moves at vertices in $V^y$ such that $y$ is a leaf of $T'$. Since $c \notin V^y$, this implies it is possible to go from $T$ to $T'$ without moving the rotor at $c$. Since $y$ is a leaf of $T'$, we can rotate the rotor at $y$ freely to get a spanning tree $T''$ such that $T''\la y \ra = T^{goal} \la y \ra$. By construction, $T'' \prec_\xi T$ and the lemma follows.\\

\textbf{Case 2: } $x\preceq_{T_s} y$.\\

Let $P$ be the path of edges in $T$ from $c$ to $y$ and let $V(P)$ be the vertices of this path. By definition of the partial order $\preceq_{T_s}$, we must have $g \in P$. Let $(G/P,\chi/P)$ be the plane graph formed by contracting the edges of $P$ and let $z$ be the contracted vertex. Since $G$ is 2-connected, $z$ is the only vertex of $G/P$ that can be a cut vertex. Furthermore, $s$ cannot be contracted into $z$ because the vertices that make up $P$ are all less than or equal to $y$ with respect to $\preceq_{T_s}$. Let $(\what {G/P},\what {\chi/P})$ be the graph we obtain after removing any edges and vertices that can only be reached from $s$ through paths that pass through $z$. We do not remove $z$. In other words, $\what {G/P}$ is the maximal 2-connected subgraph of $G/P$ which contains $s$. 

Let $\wtilde T$ be the restriction of $T$ to $E(\what {G/P})$. By 2-connectedness of $\what{G/P}$, there exists some $\wtilde T' \in \T(\what {G/P})$ such that $z$ is a leaf vertex. Furthermore, by the (stronger) inductive claim we proved in Theorem \ref{thm:sourceturn} (with $\what{T}^{start}= \wtilde{T}$, $\what{T}^{goal}=\wtilde T'$ and $\rt' = z$), we know that $\wtilde T'$ can be reached from $\wtilde T$ by a series of source-turn moves at vertices that do not involve $z$. Let 
\[T':= \wtilde T' \cup (T \cap (E(G) \setminus E(\what {G/P}))).\] 
In other words, $T'$ is equivalent to $\wtilde T'$ when restricted to $E(\what {G/P})$ and the remaining edges match the edges of $T$. By Lemma~\ref{lem:contracttreetotreeB}, there exist a sequence of source-turn moves which send $T$ to $T'$ on $(G,\chi)$ and do not involve any rotors in $P$ (including $g$). 

Finally, we need to rotate the rotor at $y$ to the correct position. Fix $h = T_s\rotdir{y} = T'_s\rotdir{y}$. Notice that $y$ is not a leaf of $T'$ because it is incident to $h$ and an edge of $P$. Nevertheless, we claim that there exists $T'' \in \T(G)$ such that $\xi(T'',y) = 0$, and we can reach $T''$ from $T'$ through a series of single-step or reverse single-step moves at $y$. Given this claim, it follows that $T'' \prec_\xi T$ because $\xi(T'',y) < \xi(T,y)$ and $\xi(T'',v) = \xi(T,v)$ for every $v \in E(G) \setminus (V^y \cup s)$. 

To prove the claim, we first consider the set $\wtilde V = V(G) \setminus V(\what {G/P})$. By definition, any $v \in \wtilde V$ must satisfy $v \prec_{T_s} w$ for some $w \in V(P)$. Furthermore, for any $w \in V(P)$, we have $w \prec_{T_s} y$. It follows that any $v \in \wtilde V$ must satisfy $v \prec_{T_s} y$ and, by assumption, this means that $T_s\la v \ra = T^{goal}_s \la v \ra$. 

Suppose that $T^{goal}_s \la y \ra \in E(G) \setminus E(\what {G/P})$. Then, by the results of the previous paragraph, we can follow rotors of $T^{goal}_s$ until we return to $y$. This implies that $T^{goal}$ contains a cycle, which is impossible. Thus, we must have $T^{goal}_s \la y \ra \in E(\what {G/P})$. Furthermore, by construction, $h$ is the only edge incident to $y$ in $\wtilde T'$. Thus, we can freely rotate the rotor at $y$ in either direction using single-step or reverse single-step moves as long as we avoid edges in $E(G) \setminus E(\what {G/P})$. Our claim follows if we show that the edges of $\chi(y) \cap E(\what {G/P})$ must all appear sequentially in $\chi(y)$ on $G$. 

Let $\{y,w\} = \i(h)$ and let $e$ be any edge of $E(\what {G/P})\setminus h$ that is incident to $y$. By 2-connectedness of $\what {G/P}$, there is a path $\what P$ in $G$ from $y$ to $w$ that contains $e$ but does not contain $h$ or any edges of $P$. Thus, $\what P \cup h$ forms a cycle $C$. Declare the ``inside'' of $C$ to be the region that does not contain $c$. By planarity, $P$ must be completely outside of $C$. Since $G$ is 2-connected, the edges inside of $C$ must be a subset of $E(\what {G/P})$. This includes all of the edges between $h$ and $e$ with respect to $\chi(y)$. Since $e$ was arbitrary, the claim follows. 
\end{proof}

\begin{corollary}\label{cor:anystepflips}
Suppose that $(c-s,T)$ is a single-step pair from $g$ to $f$ such that $\alpha_{(G,\chi)}([c-s],T) = E(G) \setminus T$. 
\begin{itemize}
    \item For any $T' \in \T(G)$ such that $(c-s,T')$ is a single-step pair from $g$ to $f$, \[\alpha_{(G,\chi)}([c-s],T') = E(G) \setminus T'.\]
    \item For any $T'' \in \T(G)$ such that $(s-c,T'')$ is a reverse single-step pair from $f$ to $g$, \[\alpha_{(G,\chi)}([s-c],T'') = E(G) \setminus T''.\]
\end{itemize}
\end{corollary}
\begin{proof}
For the first claim, apply Lemma~\ref{lem:anyvalid} with $T^{start} = T$ and $T^{goal} = T'$. This produces a sequence $\{(c_i-s_i,T^i)\}_{i \in [0,k]}$. By repeatedly applying Lemma~\ref{lem:commutepairs}, we find that $\alpha_{(G,\chi)}([c-s],T^{goal}) = E(G) \setminus T^{goal}$ as desired.

For the second claim, by definition $(c-s,T'' \setminus f \cup g)$ is a single-step pair from $f$ to $g$. Thus, $\alpha_{(G,\chi)}([c-s],T'' \setminus f \cup g) = (E(G) \setminus T'') \setminus g \cup f$. The result follows from Lemma~\ref{lem:singrev}. 
\end{proof}

\begin{figure}
\begin{center}
\begin{tikzpicture}[scale = .4]
    
    \tikzstyle{pt} = [circle,fill,inner sep=1pt,minimum size = 1.5mm]
    \tikzstyle{coin} = [draw,circle,inner sep=1pt,minimum size = 2mm]
    
    % Labelings are kinda backwards because of Jester inspiration
    \node[pt,label={north:{$c$}}] (c) at (10,8) {};
    \node[pt,label={west:{$z_5=s$}}](z1) at (0,0){} ;
    \node[pt,label={west:{$z_4$}}](z2) at (2.68,-4.64){};
    \node[pt,label={south:{$z_3$}}](z3) at (7.32,-7.32){};
    \node[pt,label={south:{$z_2$}}](z4) at (12.68,-7.32){};
    \node[pt,label={east:{$z_1$}}](z5) at (17.32,-4.64){};
    \node[pt,label={east:{$z_0=x$}}](z6) at (20,0){};
    
    \node[pt,label={west:{$w_0^1$}}] (v1) at (14.3,2) {};
    \node[pt,label={west:{$w_3^1$}}] (v2) at (12,0) {};
    \node[pt,label={west:{$w_3^2$}}] (v3) at (8,0) {};
    \node[pt,label={west:{$w_4^1$}}] (v4) at (5.7,2) {};

%%%%%%%%%%%%%%%%%%%%%%%%%%%%%%%%%%%%%%%%%%%%%%%%%%%%%%%%%%%%%5
\tikzstyle{every node} = [draw = none,fill = none,scale = .8]

    \draw (z6) to[out=210, in = 90]node[left]{$ e_1$}(z5);
    \draw [dotted](z6) to[out=270, in = 30]node[right]{$\what e_1$} (z5);
    \draw (z5) to[out=180, in = 60]node[above]{$ e_2$}(z4);
    \draw [dotted](z5) to[out=240, in = 0]node[below]{$\what e_2$} (z4);
    \draw (z4) to[out=150, in = 30] node[above]{$e_3$}(z3);
    \draw [dotted](z4) to[out=210, in = 330]node[below]{$\what e_3$} (z3);
    \draw (z3) to[out=120, in = 0] node[above]{$e_4$}(z2);
    \draw [dotted](z3) to[out=180, in = 300]node[below]{$\what e_4$} (z2);
    \draw (z2) to[out=90, in = 330] node[right]{$e_5$}(z1);
    \draw [dotted](z2) to[out=150, in = 270]node[left]{$\what e_5$} (z1);

    \draw (z3) to node[fill = white]{$h_3^1$}(v2);
    \draw [dotted](v2) to node[fill = white]{$\what h_3^1$} (c);
    \draw (z3) to node[fill = white]{$h_3^2$}(v3);
    \draw [dotted](v3) to node[fill = white]{$\what h_3^2$} (c);
    \draw (z2) to node[fill = white]{$h_4^1$}(v4);
    \draw [dotted](v4) to node[fill = white]{$\what h_4^1$} (c);
    \draw (z6) to node[fill = white]{$h_0^1$}(v1);
    \draw [dotted](v1) to node[fill = white]{$\what h_0^1$} (c);

    \draw(c) -- node[above]{\large$g$}(z6);
    \draw[dotted](z1) -- node[above]{\large$f$}(c);
    
\end{tikzpicture}
\caption{Above is the telescope graph $\tele{5}{1,0,0,2,1,0}$ that is defined in Definition~\ref{def:telescope}. The solid lines indicate the spanning tree we consider in the proof of Lemma~\ref{lem:scopegood}.}
\label{fig:telescope}
\end{center}
\end{figure}

\subsection{The Counterexample does not exist}
\label{case4:pf}

Throughout this section, $(G,\chi)$ is a 2-connected plane graph and $\alpha$ is a consistent sandpile torsor algorithm such that $\alpha_{(G',\chi')} = r_{(G',\chi')}$ for any proper minor $(G',\chi')$ of $(G,\chi)$. We also fix $c,s \in V(G)$ and $g,f \in E(G)$ such that for some $T \in \T(G)$, 
\begin{itemize}
    \item the pair $(c-s,T)$ is a source-turn pair from $g$ to $f$, and
    \item $\alpha_{(G,\chi)}([c-s],T) = E(G) \setminus T$. 
\end{itemize} 
Furthermore, $x \in V(G)$ is defined such that $\i(g) = \{c,x\}$ (where $x$ and $s$ are allowed to coincide). We also assume that each spanning tree of $G$ contains half of the edges of $G$. Note that we use $T$ as a variable for an arbitrary tree in this section. We do not fix $T$ to be the specific tree satisfying the conditions of Case 4. 

Our general strategy for proving Case 4 makes use of Corollary \ref{cor:anystepflips} and the statement of Case 4 in the following manner:

\begin{description}
    \item[Step 1] Corollary \ref{cor:anystepflips} implies that for any single-step pair of the form $(c-s,T')$, we have $\alpha_{(G,\chi)}([c-s],T') = E(G) \setminus T'$. If there is any such $T'$ where $E(G) \setminus T'\not\in \T(G)$, then this contradicts the fact that $\alpha_{(G,\chi)}$ is a sandpile torsor action (see Lemma~\ref{lem:noscopegood}).
    \item[Step 2] We define and characterize the class of \emph{telescope graphs} which consists of all plane graphs in which the contradiction above does not appear (see Lemma~\ref{lem:flippablescope}). 
    \item[Step 3] We prove the result directly when restricted to telescope graphs (See Lemma~\ref{lem:scopegood}). 
\end{description}

\begin{definition}
\label{def:telescope}
We introduce a family of plane graphs called \emph{telescope graphs}, which all satisfy the conditions at the beginning of the subsection. These graphs are indexed by $\ic \in \Z_{\ge 0}$ and a vector $(k_0,k_1,\dots,k_\ic) \in \Z_{\ge 0}^{\ic+1}$. We begin with $c$, $s$, $x$, $f$, and $g$ as defined above. We then relabel $s$ as $z_n$ and $x$ as $z_0$. Add the vertices $z_1,\dots,z_{\ic-1}$ to the graph and then, for each $i\in [\ic]$, add two parallel edges between $z_{i-1}$ and $z_i$. Finally, for each $i \in [0,\ic]$, we add $k_i$ degree 2 vertices to the graph, which are incident to both $c$ and $z_i$.

Up to ribbon graph isomorphism, there is a unique planar ribbon structure such that $f$ comes directly after $g$ in the cyclic order around $c$. We will write the graph along with this planar ribbon structure as $\tele{\ic}{k_0,\dots,k_\ic}$. 
\end{definition}

Figure~\ref{fig:telescope} shows an example of a telescope graph. For each $i \in [0,\ic]$ where $k_i \not= 0$, we label the degree 2 vertices incident to $z_i$ by $(w_i^1,w_i^2,\dots,w_i^{k_i})$, ordered based on the ribbon structure at $z_i$. We label the edge between $z_i$ and $w_i^j$ by $h_i^j$ and the edge between $c$ and $w_i^j$ by $\what h_i^j$. We also label the two edges between $z_{i-1}$ and $z_i$ by $e_i$ and $\what e_i$, where $e_i$ comes directly before $\what e_i$ in the cyclic order around $z_{i-1}$. 

\begin{definition}\label{def:sstree}
Let $(G,\chi)$ be a plane graph satisfying the conditions discussed at the beginning of the subsection. We say that a spanning tree $T \in \T(G)$ is a \emph{single-step tree} if $(c-s,T)$ is a single-step pair from $g$ to $f$ or $(s-c,T)$ is a reverse single-step pair from $f$ to $g$.
\end{definition}
\begin{lemma}\label{lem:whenflip}
A spanning tree $T \in \T(G)$ is a single-step tree if and only if it contains exactly one of $f$ and $g$ as well as a path from $x$ to $s$ that does not pass through $c$. 
\end{lemma}
\begin{proof}
This follows immediately from Definition~\ref{def:sstree} as well as Lemma~\ref{lem:singcriteria}. 
\end{proof}

\begin{lemma}
\label{lem:flippablescope}
The following are equivalent (when $(G,\chi)$ is a 2-connected plane graph satisfying the conditions from the beginning of the section):
\begin{itemize}
    \item $(G,\chi) = \tele{\ic}{k_0,\dots,k_\ic}$ for some $n \in \Z_{\ge 0}$ and $(k_0,k_1,\dots,k_\ic) \in \Z_{\ge 0}^{\ic+1}$.
    \item For every single-step tree $T \in \T(G)$, we have $E(G) \setminus T \in \T(G)$. 
\end{itemize}
\end{lemma}
\begin{proof}
The forward direction follows from direct argument. For the reverse direction, we apply induction on the number of edges in the graph.

For the forward direction, consider any single-step tree $T \in \T(G)$. For any $i \in [0,n]$ and $j \in [1,k_i]$, $T$ must contain at least one of $h_i^j$ and $\what h_i^j$ or else $w_i^j$ is isolated. By Lemma~\ref{lem:whenflip}, $T$ contains $f$ or $g$ and an path $P$ from $z_0$ to $z_n$ that does not pass through $c$. Notice that $P \cup  \{h_i^j,\what h_i^j\} \cup f$ and $P \cup  \{h_i^j,\what h_i^j\} \cup g$ both contain circuits. Thus, $T$ must contain exactly one of the edges $h_i^j$ and $\what h_i^j$. Similarly, for $i \in [1,n]$, the path $P$ must contain exactly one of $\{e_i,\what e_i\}$. To prevent a circuit, $T$ must also contain exactly one of every such pair. It follows that for $T$ to be a single-step tree, it must contain exactly one of $\{f,g\}$, exactly one of each $\{h_i^j,\what h_i^j\}$, and exactly one of each $\{e_i,\what e_i\}$. It is straightforward to check that any such choice of edges will give a single-step tree. Furthermore, the complement of a single-step tree makes the opposite choice for each pair of edges listed above. Thus, $E(G)\setminus T$ is also a single-step tree and a spanning tree. 

We will prove the reverse direction by induction. The only 2-edge plane graph satisfying the property is the double edge, which is equivalent to $\tele{0}{0}$. Suppose that every plane graph with $2N$ edges satisfying the property is a telescope graph. Let $(G,\chi)$ be a 2-connected plane graph with $2N+2$ edges such that for every single-step tree $T \in \T(G)$, we have $E(G) \setminus T \in \T(G)$. On any graph, the size of each spanning tree is $|V(G)| - 1$. For $E(G) \setminus T$ to ever be a spanning tree, we must have $2(|V(G)| - 1) = |E(G)|$ and $|V(G)| = |E(G)|/2 + 1 = N+2$. 

First, suppose that $x = s$ so that the edges $f$ and $g$ are parallel in $G$. For this case, we do not actually need to use induction. Let $v$ be an arbitrary vertex other than $x$ and $s$. Suppose that $v$ is incident to a pair of parallel edges other than $\{f,g\}$. By 2-connectedness, removing these parallel edges does not disconnect the graph. Thus, there must be some tree $T$ containing $g$ and neither of these parallel edges. Since $x = s$, it follows immediately from Lemma~\ref{lem:whenflip} that $T$ is a single-step tree. We reach a contradiction because $E(G) \setminus T$ contains a pair of parallel edges and cannot be a tree. Thus, $\{f,g\}$ is the only pair of parallel edges in $G$.

Furthermore, if there is some spanning tree $T$ that contains $g$ and all of the edges incident to $v$, then $v$ is isolated in $E(G) \setminus T$ and this set cannot form a spanning tree. Thus, any spanning tree of $G$ containing $g$ must only contain a subset of the edges incident to $v$. Because every subtree of $G$ is the subset of a spanning tree, this implies that the union of $g$ and the edges incident to $v$ must contain a circuit. If $v = c$, then this circuit is simply the double edge $\{f,g\}$. If $v \neq c$, then the fact that $i(g) = \{c,s\}$ implies that there must be an edge incident to $v$ and $c$ and an edge incident to $v$ and $s$. We showed earlier that if there are $N$ vertices other than $c$ and $s$, then there must be $2N+2$ edges. These edges are accounted for by $f,g$ and the two specified edges incident to each of the $N$ vertices $v\notin \{c,s\}$. We already accounted for all of these edges, so $(G,\chi)$ must be equal to $\tele{0}{N}$. 

Now, we consider the case where $s \not= x$. By 2-connectedness, there must be some path $P$ along $G$ from $x$ to $s$ that does not pass through $c$. Then, the union of $P$ with $f$ is also acyclic, so it can be expanded to a tree $T$, which must not contain $g$. It follows from Lemma~\ref{lem:whenflip} that $T$ is a single-step tree, which means that $E(G) \setminus T$ must be a spanning tree by assumption. However, $E(G) \setminus T$ contains all edges incident to both $c$ and $x$. Thus, there cannot be any edges parallel to $g$. 

Consider the set of all edges incident to $x$. If these edges contain no circuits, then we can take their union with $P$ and then add edges to obtain a spanning tree $T$. It follows from Lemma~\ref{lem:whenflip} that $T$ is a single-step tree. However, $E(G) \setminus T$ cannot be a spanning tree because $x$ is isolated. This is a contradiction, so the edges incident to $x$ must contain a circuit. In particular, this must be a double edge $\{e,\what e\}$. By our reasoning in the previous paragraph, these edges must connect $x$ to a vertex $v \not= c$. 

By assumption, any single-step tree $T$ must contain exactly one of $e$ and $\what e$. Otherwise, $T$ or $E(G) \setminus T$ contains a cycle. Furthermore, it follows from Lemma~\ref{lem:whenflip} that $T$ is a single-step tree on $(G,\chi)$ if and only if $T \setminus \{e,\what e\}$ is a single-step tree on $(G/e\setminus \what e, \chi/e \setminus \what e)$. This implies that for every single-step tree $T \in (G/e\setminus \what e, \chi/e \setminus \what e)$, we must have $E(G) \setminus T \in \T(G/e\setminus \what e)$. By the induction hypothesis, this means that $(G/e\setminus \what e, \chi/e \setminus \what e)$ must be a telescope graph.

We have shown that after contracting a double edge incident to $x$, we obtain a telescope graph. Let $z$ be the vertex formed by contraction. To recover $G$, we need to determine which edges incident to $z$ in $G/e\setminus \what e$ are incident to $x$ in $G$ and which are incident to $v$. If $v= s$, then it is immediate that $(G,\chi)$ is a telescope graph with $n=1$. Otherwise, $G/e\setminus \what e$ must contain a pair of parallel edges $e'$ and $\what e'$ which are incident to $z$. By definition, $(G,\chi)$ is a telescope graph if and only if neither $e'$ nor $\what e'$ are incident to $x$ on $(G,\chi)$.

Suppose for the sake of contradiction that $e'$ is incident to $x$ on $G$. Then, $e'$ is part of an path $P$ from $x$ to $s$ that does not pass through $c$. In particular, $f \not\in P$ and $g \cup P$ contains no cycles. Thus, $P \cup g$ is contained in a spanning tree that does not contain $f$, $e$, or $\what e$. This spanning tree is a single-step tree but the complement contains the cycle $\{e,\what e\}$. 

The argument from the previous paragraph still holds if we replace $e'$ with $\what e'$. Thus, $(G,\chi)$ must be a telescope graph. 
\end{proof}
\begin{lemma}
\label{lem:noscopegood}
If $(G,\chi)$ is not a telescope graph, then for any $T\in \T(G)$ such that $(c-s,T)$ is a single-step pair  from $g$ to $f$, we have $\alpha_{(G,\chi)}([c-s],T) = T \setminus g \cup f$. 
\end{lemma}
\begin{proof}
Suppose that this condition does not hold. By Lemma~\ref{lem:singstepnocase2}, this means that $\alpha_{(G,\chi)}([c-s],T) = E(G) \setminus T$. By Lemma~\ref{lem:flippablescope}, there must be some single-step tree $T'$ such that $E(G) \setminus T'$ is not a spanning tree. We arrive at a contradiction from Corollary~\ref{cor:anystepflips} and the definition of single-step trees, because the output of a sandpile torsor action must always be a spanning tree.  
\end{proof}

\begin{lemma}\label{lem:scopegood}
If $(G,\chi)$ is a telescope graph, then for any $T\in \T(G)$ such that $(c-s,T)$ is a single-step pair from $g$ to $f$, we have $\alpha_{(G,\chi)}([c-s],T) = T \setminus g \cup f$. 
\end{lemma}
\begin{proof}
Suppose for the sake of contradiction that the lemma does not hold. Then, by Lemma~\ref{lem:singstepnocase2}, there exists some $T \in \T(G)$ such that $(c-s,T)$ is a single-step pair from $g$ to $f$ and $\alpha_{(G,\chi)}([c-s],T) = E(G) \setminus T$. By Corollary~\ref{cor:anystepflips}, this equality holds for any single-step tree. 

To obtain a contradiction, we explicitly construct a divisor $D_{\Sigma}$ such that $\alpha_{(G,\chi)}([D_{\Sigma}],T) = \alpha_{(G,\chi)}([c-s],T) = E(G)\setminus T$, but $[D_{\Sigma}]\neq [c-s]$.\\

\textbf{Case 1: } $n>0$.\\

Let $\tele{n}{k_0,\dots,k_n} = (G,\chi)$ and let $T'$ be the spanning tree 
\[T' := g \cup \{e_i \mid i \in [1,n]\} \cup \{h_i^j \mid i \in [0,n], j \in [1,k_i]\}.\]
See Figure~\ref{fig:telescope} for this tree on $\tele{5}{1,0,0,2,1,0}$. 

We consider several sequences of divisors (where we ignore any divisors involving $w_i^j$ for $k_i = 0$ or $z_i$ for $i <0$):
\begin{alignat*}{2}
&\mathcal E_0 &&= ( w_0^1 - c, \dots, w_0^{k_0} -c, w_1^1 - c, \dots,w_1^{k_1}-c,\dots,w_n^{k_n}-c)\\
&\mathcal E_1 &&= (z_n - w_n^1, \dots, z_n - w_n^{k_n}, z_n - c)\\
&\mathcal E_2 &&= (z_{n-1} - w_{n-1}^1, \dots, z_{n-1} - w_{n-1}^{k_{n-1}},z_{n-1} - z_n,z_{n-1} - z_n)\\
&\mathcal E_3 &&= (z_{n-2} - w_{n-2}^1, \dots, z_{n-2} - w_{n-2}^{k_{n-2}},z_{n-2} - z_{n-1},z_{n-2} - z_{n-1})\\
&\dots && = \dots\\
&\mathcal E_{n+1} &&= (z_0 - w_0^1, \dots, z_0 - w_0^{k_{0}},z_{0} - z_{1},z_{0} - z_{1})\\
&\mathcal E &&= (\mathcal E_0,\mathcal E_1,\dots,\mathcal E_{n+1}).
\end{alignat*}

Let $D_l$ be the $l^{th}$ entry in $\mathcal E$. Let $T^0 = T'$ and recursively define the sequence \[(T^1,T^2,T^3,\dots, T^{|\mathcal E|}) \in \T(G)^{|\mathcal E|}\] such that $r_{(G,\chi)}([D_l],T^{l-1}) = T^l$. It is straightforward to see that $T^{|\mathcal E|} = E(G) \setminus T'$, and for any $l \in [1,|\mathcal E|]$, the pair $(D_l, T^{l-1})$ is a single-step pair. We claim that for all $l \in [1,|\mathcal E|]$, we also have $\alpha_{(G,\chi)}([D_l],T^{l-1}) = T^l$.

For any $l\in [1,|\mathcal{E}|]$, let $c',s' \in V(G)$ such that $D_l = c' - s'$. Notice by the definition of $\mathcal E$ that $c' \not= c$. By Lemma~\ref{lem:noscopegood}, $\alpha_{(G,\chi)}([D_l],T^{l-1}) = r_{(G,\chi)}([D_l],T^{l-1}) = T^l$ unless $(G,\chi)$ remains a telescope graph when $c'$ plays the role of $c$ and $s'$ plays the role of $s$ (i.e. when there is a ribbon graph automorphism which maps $c$ to $c'$ and $s$ to $s'$). It is easy to verify that when $n>0$, the only ribbon graph automorphism is trivial. 

Let $D_\Sigma$ be the sum of all of the divisors that make up $\mathcal E$. We have shown that \[\alpha_{(G,\chi)}([D_\Sigma],T') = r_{(G,\chi)}([D_{\Sigma}],T') = T^{|\mathcal E|} = E(G) \setminus T'.\] If $\alpha_{(G\chi)}([c-s],T') = E(G) \setminus T'$, then this implies that $[D_{\Sigma}] = [c-s]$. This is false because 
\[r_{(G,\chi)}([D_{\Sigma}],T') = E(G) \setminus T' \not= T' \setminus g \cup f = r_{(G,\chi)}([c-s],T').\]\\

\textbf{Case 2: } $n=0$.\\

If $k_0 = 0$, then $E(G) = \{f,g\}$ and the proof is trivial. Otherwise, let $T' = g \cup \{ h_0^j\mid j \in [1,k_i]\}$ and $\mathcal E = (w_0^1-c ,w_0^2-c,\dots, w_0^{k_0}-c, s-w_0^1, s-w_0^2, \dots, s-w_0^{k_0}, s - c)$. Define $D_l$, $T^l$, and $D_\Sigma$ as in Case 1, but with this new choice of $\mathcal E$. 

The argument that we used in Case 1 is mostly still valid, except that there does exist a non-trivial ribbon graph automorphism on $(G,\chi)$. Nevertheless, the only such automorphism must flip $c$ and $s$. Thus, the only additional equality that we need to check is  $\alpha_{(G,\chi)}([s-c], T^{2k_0}) = T^{2k_0+1}$. However, $T^{2k_0}$ contains neither $g$ nor $f$, so $E(G)\setminus T^{2k_0}$ is not a spanning tree. By Lemma~\ref{lem:singstepnocase2}, we know that $\alpha_{(G,\chi)}([s-c], T^{2k_0}) = r_{(G,\chi)}([s-c], T^{2k_0}) = T^{2k_0+1}$. The remainder of the proof follows using the same reasoning from Case 1. 
\end{proof}

\bibliographystyle{alpha}
\bibliography{biblio.bib}

\end{document}